\documentclass[a4paper]{amsart}
\usepackage{amssymb,lmodern,amsmath,amsthm,mathrsfs,enumitem,color,hyperref,url,graphicx}
\usepackage[T1]{fontenc}      
\usepackage[utf8]{inputenc}   
\usepackage{lmodern}          

\usepackage{chngcntr}
\counterwithin{section}{part} 
 
\newtheorem{thm}{Theorem}
\newtheorem{ques}[thm]{Question} 
\newtheorem{lemma}[thm]{Lemma} 
\newtheorem{prop}[thm]{Proposition} 
\newtheorem{cor}[thm]{Corollary}

\newtheorem{claim}{Claim}

\theoremstyle{definition} 
 
\newtheorem{remark}[thm]{Remark} 

\newtheorem{definition}[thm]{Definition} 

\newtheorem{notation}[thm]{Notation}
\newtheorem*{definition*}{Definition}

\makeatletter 
\renewcommand{\thesection}{%
  \ifnum\value{part}>0
    \thepart.%
  \fi
  \arabic{section}%
}
\makeatother

\makeatletter 
\renewcommand{\thethm}{%
\ifnum\value{part}>0
    \thepart.\arabic{thm}%
  \else
    \arabic{thm}%
  \fi
}
\makeatother
\renewcommand{\thepart}{\Roman{part}} 

\makeatletter
\@addtoreset{thm}{part}
\makeatother

\newtheorem{thmx}{Theorem}

\def\Nat{\mathbb N} 
\def\Rea{\mathbb R}

\def\B{\mathcal B} 
\def\C{\mathcal C}

\def\F{\mathcal F} 
\def\N{\mathcal N}

\def\PP{\mathcal P} 
\def\pot{\textnormal{Pow}}    

\def\P{\mathcal P}

\def\diam{\operatorname{diam}} 
\def\co{\operatorname{conv}} 
\def\aco{\operatorname{aco}} 
\def\height{\operatorname{ht}} 
\def\ep{\varepsilon}

\def\K{\mathcal K} 
\def\Nat{\mathbb N} 
 
\def\er{\mathbb R} 
\def\Rat{\mathbb Q}

\def \sgn{\operatorname{sgn}}

\def\Clop{\operatorname{Clop}} 
\def \At {\operatorname{At}}

\def \ext {\operatorname{ext}}

\def\span{\operatorname{span}} 
\def\closedSpan{\overline{\span}}

\def\iff{\Longleftrightarrow}

\def \1{\boldsymbol 1}

\newcommand{\norm}[1]{\left\|#1\right\|}

\newcommand{\abs}[1]{\left|#1\right|}
\newcommand{\setsep}{;\,}

\newcommand\isomtrclass[2][]{\langle #2\rangle_{\equiv}^{#1}}

\newcommand\homeoclass[1]{\langle #1\rangle_{\sim}}

\makeatletter
\newcommand{\sectionnotoc}[1]{%
  \begingroup
  \let\addcontentsline\@gobblethree
  \section*{#1}%
  \endgroup
}
\makeatother

\title[Borel Complexity of $C(K)$ Spaces with $K$ Countable]{Borel Complexity of Isometry Classes of $C(K)$ Spaces with Countable Compacta}
\author[{M. C\' uth}]{Marek C\'uth}
\author[{M. Dole\v zal}]{Martin Dole\v zal}
\author[O. Kurka]{Ond\v rej Kurka}
\author[J.~Rondo\v s]{Jakub Rondo\v s}

\email{cuth@karlin.mff.cuni.cz}
\email{dolezal@math.cas.cz}
\email{kurka.ondrej@seznam.cz}
\email{rondojak@fel.cvut.cz}

\address[M.~C\' uth]{Charles University, Faculty of Mathematics and Physics, Department of Mathematical Analysis, Sokolovsk\'a 83, 186 75 Prague 8, Czech Republic}
\address[M.~Dole\v zal, O.~Kurka]{Institute of Mathematics of the Czech Academy of Sciences, \v Zitn\'a 25, 115 67 Prague 1, Czech Republic}
\address[J.~Rondo\v s]{Department of Mathematics, Faculty of Electrical Engineering, Czech Technical University in Prague, Technick\'a 2,
166 27 Prague 6, Czech Republic.}

\subjclass[2010] {54H05,  46B04, 28A05 (primary), 46B20, 46B25, 54E15 (secondary)}

\keywords{isometry classes of Banach spaces, Borel complexity, $\C(K)$ space, $L_1$-predual, homeomorphism classes of compact metric spaces, Szlenk derivative} 

\thanks{M. C\'uth was supported by the Czech Science Foundation, project no. GACR 24-10705S. Research of Martin Dole\v{z}al and Ond\v{r}ej Kurka was supported by the Academy of Sciences of the Czech Republic (RVO 67985840). 
}

\begin{document}

\begin{abstract}
For every countable compact space $K$, we determine the exact Borel complexity of the isometry class of the Banach space $\C(K)$. As a byproduct, we also determine the precise Borel complexity of the homeomorphism class of a fixed countable compact space $K$, improving earlier results of Cenzer and Mauldin. The above results provide concrete and natural examples of sets with arbitrarily high, still exactly determined, Borel complexity.

Moreover, we find a new characterization of those real $L_1$-preduals that are isometric to $\C(K)$ for some zero-dimensional compact space $K$ and we determine the precise Borel complexity of $\C(2^\Nat)$.
\end{abstract}

\maketitle

\setcounter{tocdepth}{1}
\tableofcontents

\sectionnotoc{Introduction} 

The classification of various classes of separable Banach spaces is a natural problem in functional analysis. A systematic framework for studying and comparing such classifications is provided by descriptive set theory, which enables one to assess their relative complexity via the notion of Borel complexity. Within this framework, one can meaningfully distinguish between ``simpler'' and ``more complex'' classification problems. The main aim of this work is, given an infinite countable compact space $K$, to determine the exact complexity of the isometry class of the real Banach space $\C(K)$.

Before presenting our results, we note that apart from the Banach space theory, this work is closely related to the research program in invariant descriptive set theory (IDST), which provides a unified framework for comparing classification problems across mathematics and revealing connections between them. We refer e.g. to the monograph \cite{Gao09}, where the appropriate context may be found. Both classical and invariant DST
in Banach space theory have been intensively studied in recent years, leading to a range of interesting results and applications. These include (a) the study of universality problems, initiated by the work of Bourgain \cite{Bourgain1980} and subsequently developed through numerous contributions, including, for instance, the recent work \cite{AntBeanBraga2021}; (b) the analysis of the complexity of various classification problems via Borel reducibility, an approach rooted in Bossard’s coding of separable Banach spaces \cite{Bossard2002} and further advanced, for example, in the recent work \cite{CRF23}; and (c) the analysis of the complexity of various classes of Banach spaces, for which we refer, e.g., to the survey \cite{G17} as well as to more recent contributions such as \cite{GodSR, K19}.

A standard approach -- introduced by Bossard in~\cite{Bossard1993,Bossard2002} -- is to code separable Banach spaces by elements of the standard Borel space \( SB(X) \) of closed linear subspaces of \( X \),
where \( X \) is a universal separable Banach space (for instance \( X = \C(2^\mathbb{N}) \)). In this work, we use the codings $\P_\infty$ and $\B$ introduced in \cite{CDDK1}, which appears more natural in our context; see Remark~\ref{rem:admissible} for a detailed discussion of this choice. Let us recall the coding $\B$, for its variant $\P_\infty$ we refer to the introduction to Part~\ref{part2}, where more details may be found.\\

\begin{definition*}
    Let $V$ be the countable $\Rat$-linear vector
space of elements of $c_{00}$ having rational entries, and define $\B$ to be the set of
all functions $\mu\in [0,\infty)^V$ which can be extended (uniquely) to a norm on $c_{00}$. For $\mu\in\B$ we denote by $X_\mu$ the completion of $(V,\mu)$.
\end{definition*}
It is easy to check that then for any infinite-dimensional separable Banach space $X$ there exists $\mu\in\B$ such that $X$ is linearly isometric to $X_\mu$ and that $\B\subseteq [0,\infty)^V$ is a $G_\delta$ set and therefore $\B$ is Polish space when endowed with the product topology inherited from $[0,\infty)^V$. We refer to the introduction to Part~\ref{part2} for more details including further references.

Given an infinite-dimensional separable Banach space $X$ we denote \[\isomtrclass{X}:=\{\mu\in\B\setsep X\text{ is linearly isometric to }X_\mu\}\]
and we say that the set $\isomtrclass{X}$ is the \emph{isometry class of $X$}. In \cite{CDDK2} the authors proved that $\ell_2$ is the only space with closed isometry class, that $\isomtrclass{L_p([0,1])}$  is $G_\delta$-complete (hence, $G_\delta$ and not $F_\sigma$) whenever $p\in[1,\infty)\setminus\{2\}$, while the classes of $\ell_p$ and $c_0$ are $F_{\sigma\delta}$-complete (hence, $F_{\sigma\delta}$ and not $G_{\delta\sigma}$). Other closely related papers dealing with similar problems using the coding $\B$ are e.g. \cite{CDR26} (characterizing Banach spaces with $G_\delta$ isometry class), \cite{complexityDaugavet} (devoted to the classification of the Daugavet and several related diameter two properties), or \cite{S26} (devoted to the study of complexities in the context of Lipschitz-free spaces).

In order to shorten the notation, in what follows given an ordinal $\alpha$ we write $\C(\alpha)$ instead of $\C([0,\alpha])$. The Cantor set is denoted as $2^\Nat$. Following the classical monograph \cite{Kechrisbook}, $\boldsymbol{\Pi}^0_\alpha$ are multiplicative Borel classes and by $D_2(\boldsymbol{\Gamma})$ we denote the class of differences of two sets from \( \boldsymbol{\Gamma} \). We note that the statement that a set is $\boldsymbol{\Gamma}$-complete means that it belongs to the class 
$\boldsymbol{\Gamma}$, and that this bound is optimal, we refer to Notation~\ref{not:dstBasics} for more details. Our main results are the following.

\begin{thmx}\label{thm:Intro1}Let $\beta$ be either $0$ or a countable limit ordinal and let $n\in\Nat\cup\{0\}$ be such that $\beta+n>0$.
\begin{itemize}
    \item The isometry class of \( \C(\omega^{\beta+n}) \) is a \( \boldsymbol{\Pi}^0_{\beta+2n+1} \)-complete set.
    \item The isometry class of \( \C(\omega^{\beta + n}\cdot k) \) is a \( D_2(\boldsymbol{\Pi}^0_{\beta+2n+1}) \)-complete set for every $k\in\Nat$ with $k\geq 2$.
    \item The isometry class of \( \C(2^\Nat) \) is a \( \boldsymbol{\Pi}^0_3 \)-complete set.
\end{itemize}
\end{thmx}

By the Mazurkiewicz-Sierpi\'{n}ski theorem, for every infinite countable compact space $K$ there is a unique countable ordinal $\alpha$ and unique natural number $k$ such that $K$ is homeomorphic to $[0,\omega^\alpha\cdot k]$. Thus, Theorem~\ref{thm:Intro1} covers complexity of all the isometry classes of $\C(K)$ spaces with $K$ infinite countable compacta or $K$ being the Cantor set.

As a byproduct of our method of proof, we obtain the following additional results. Answering \cite[Question 6]{CDDK2} in the negative, we show that the class of Banach spaces with a summable Szlenk index is $\boldsymbol{\Sigma}^0_\omega$-hard (see Corollary~\ref{cor:szlenk}). Furthermore, we establish a precise correspondence between topological and isometric complexity: given any uncountable compact metric space $X$, the class of compact subsets of $X$ that are homeomorphic to \( [0,\omega^{\beta + n}] \),
respectively to \( [0,\omega^{\beta + n}\cdot k] \) with \( k \ge 2 \), has exactly the same Borel complexity as the isometry class of \( \C(\omega^{\beta + n}) \),
respectively of \( \C(\omega^{\beta + n}\cdot k) \) (see Theorem~\ref{thm:main2Part2}). In contrast, for the Cantor set $2^\Nat$, the isometry class of $\C(2^\Nat)$ is different from the homeomorphism class of $2^\Nat$. Indeed, it follows from known results that the latter is always a $G_\delta$ set, regardless of the ambient space $X$; see Remark~\ref{rem:cantorHomeo}.

The highlight of our result is that it provides concrete and natural examples of classes whose complexities are arbitrarily high, still exactly determined, in the Borel hierarchy.
Such examples are rather scarce in analysis and topology.
For instance, in the classical monograph by Kechris \cite[p.~189]{Kechrisbook} the author asks for natural sets which are, at least, from the \( 5 \)th Borel class but not in any lower Borel class.

For the homeomorphism classes of compact sets, examples have been known. Indeed, Cenzer and Mauldin \cite{CM82,CM83} (see also \cite[p.~182]{Kechrisbook}) proved that given $\beta=0$ or $\beta$ being a countable limit ordinal and $n\in\Nat\cup\{0\}$, the set of compact spaces with Cantor-Bendixson derivative less than $\beta + n$ is $\boldsymbol{\Sigma}^0_{\beta + 2n}$-complete. 
Moreover, it follows essentially from their proofs (see Remark~\ref{rem:cenzerMauldinLoweEstimate} for further details) that the set of those compact subsets of \( 2^\mathbb{N} \)
which are homeomorphic to \( [0,\omega^{\beta + n}] \) is \( \boldsymbol{\Pi}^0_{\beta + 2n} \)-hard.
However, we improve this in three directions:
a) we improve the lower bound on the Borel complexity by \( 1 \) (by proving that the set is, in fact, \( \boldsymbol{\Pi}^0_{\beta + 2n +1} \)-hard);
b) we show that this improved lower bound is optimal, by proving that it is also an upper bound (that is, we show that the set is, in fact, \( \boldsymbol{\Pi}^0_{\beta + 2n +1} \)-complete);
c) we provide an analogous result also for \( [0,\omega^{\beta + n}\cdot k] \) with \( k \ge 2 \).

Although our main motivation for this paper, as explained above, lies in the heart of descriptive set theory,
Part~\ref{part1} may be of independent benefit for readers primarily interested in functional analysis. We start by correcting a minor inaccuracy from the proof of \cite[Theorem 4.13]{CDDK2}, where part of the argument was false, we find a new argument which fills in this gap. The main goal of this part is to provide a suitable characterization of the Banach spaces isometric to \( \C(\omega^\alpha\cdot k) \),
for a fixed countable ordinal \( \alpha \) and \( k \in \mathbb{N} \) (see e.g. Corollary~\ref{cor:isometricCharacterizationCKcountable})
which is later applied to obtain the upper bound on the Borel complexities of the corresponding isometry classes (see Theorem~\ref{thm:upperBoundFinal}).

While proving this auxiliary result, we obtain several facts concerning Banach spaces \( \C ( K ) \) (where \( K \) is a compact space), as well as general isometric \( L_1 \)-preduals, which may be of independent interest. The key new ingredient is the following surprising condition on a Banach space $X$, which makes it possible to identify $\C(K)$ spaces with zero-dimensional $K$ within the class of isometric real $L_1$-preduals.

\begin{equation}\label{eq:newCondition}
    \overline{\aco}\Big\{a\in S_X\colon\; \max\{\|a-y\|,\|a+y\|\} = 2\text{ for every $y\in S_X$}\Big\} = B_X.
\end{equation}

Our second main result is then the following. 

\begin{thmx}\label{thm:Intro2}
  Let $X$ be a real Banach space such that $X^*$ is isometric to $L_1(\mu)$ for some measure $\mu$. Then $X$ satisfies condition \eqref{eq:newCondition} if and only if there exists a zero-dimensional compact space $K$ such that $X$ is isometric to $\C(K)$.
\end{thmx}

This result enables us to find new characterizations of $\C(K)$ spaces with $K$ countable as well as, for example, an isometric characterization of the space of continuous functions on the Cantor set. We refer to Part~\ref{part1} for further details and related results. Let us note that it is well-known a real Banach space $X$ is isometric to a $\C(K)$ space if and only if it is $L_1$-predual, $\ext B_{X^*}$ is $w^*$-closed set and $\ext B_X \neq \emptyset$, see \cite[Corollary on p. 337]{LW69} or also \cite[Theorem 6.6]{L64}. There are even other deep Banach space characterizations of spaces isometric to $\C(K)$ spaces (see, e.g., \cite[Chapter 3]{LaceyBook} for a list of such conditions). However, none of these seem to easily provide the upper bound on the Borel complexity which we need, as all of them involve extreme points or another additional structure. Our condition \eqref{eq:newCondition} may be viewed as a variant that replaces the notion of an extreme point with a more tractable alternative, which in particular naturally yields a an optimal estimate on the Borel complexity.

The paper is organized as follows. In Part~\ref{part1}, we collect results related to characterizations and properties of $\C(K)$ spaces. In particular, this part contains the proof of Theorem~\ref{thm:Intro2} as well as a new isometric characterization of $\C(K)$ spaces with $K$ countable or $K$ being the Cantor set (see Theorem~\ref{thm:charactCkcountableInCk} and Theorem~\ref{thm:mainPart1Strong}). In Part~\ref{part2}, the main objective is to establish Theorem~\ref{thm:Intro1}, together with several related results. Each part begins with a more detailed overview of its contents and includes the preliminaries needed specifically for that part. The division into two parts is intended to make the paper more accessible and to highlight aspects that may be of independent interest to different groups of readers. At the same time, the paper remains logically unified: Part~\ref{part2} relies substantially on some of the main results obtained in Part~\ref{part1}, and the development of Part~\ref{part1} is in turn motivated by the needs of Part~\ref{part2}.

\subsection*{Notation:} Let us conclude the introduction by setting up some basic notation used throughout both parts of the paper. Basic notation and terminology are adopted as in \cite{BiortogBook} for Banach space theory and in \cite{Kechrisbook} for descriptive set theory. All Banach spaces in this paper are considered over the field of real numbers.

As mentioned already above, given an ordinal $\alpha$ we write $\C(\alpha)$ instead of $\C([0,\alpha])$ and the Cantor set is denoted as $2^\Nat$. Given Banach spaces $X$ and $Y$, the symbol $X \equiv Y$ means that $X$ and $Y$ are isometrically isomorphic. A Banach space $X$ is called \emph{$L_1$-predual} (or a \emph{Lindenstrauss space}) if $X^*$ is isometric to the space $L_1(\mu)$ for some measure space $(\Omega, \Sigma, \mu)$. Further, $X$ belongs to the class $\mathcal{L}_{\infty, 1+}$ if for every finite-dimensional subspace $E\subseteq X$ and every $\ep>0$, there exists a finite-dimensional subspace $F\subseteq X$ containing $E$ such that $d_{BM}(F,\ell_\infty^{\dim F})<1+\ep$, where $d_{BM}$ denotes the Banach-Mazur distance. We recall that a Banach space is  $\mathcal{L}_{\infty, 1+}$ space if and only if it is an $L_1$-predual, see \cite[\S 23 Theorem 2]{LaceyBook}. We say $X$ has the \emph{Daugavet property} if every rank-one bounded linear operator $T \colon X \rightarrow X$ satisfies the Daugavet equation $\norm{T+Id}=1+\norm{T}$. 

Given $\ep>0$, the \emph{Szlenk derivative} of a $w^*$ compact subset $K$ of $X^*$ is defined as
\[s_{\ep}(K)=s_{\ep}^1(K)=\{ x \in K \setsep \text{ for each }w^* \text{neighbourhood } V \text{ of } x, \diam(V \cap K) \geq \ep \},\]
and, for an ordinal $\alpha$ we set $s^0_{\ep}=K$, $s_{\ep}^{\alpha+1}(K)=s_{\ep}(s_{\ep}^\alpha(K))$, and $s_{\ep}^\alpha(K)=\bigcap_{\beta< \alpha} s_{\ep}^\beta(K)$ if $\alpha$ is a limit ordinal.

Given a set $S$, we write $[S]^{<\omega}$ for the collection of all finite subsets of $S$ and $\pot(S)$ for the collection of all of its subsets.

All compact spaces in this paper are assumed to be Hausdorff. Recall that a compact space is \emph{zero-dimensional} if it has a basis consisting of clopen sets. Further, for a compact space $K$, let $D^0(K)=K$ and let $D^1(K) = D(K)$, the \emph{Cantor-Bendixson derivative of} $K$, be the set of all accumulation points of $K$ or, equivalently, $D(K)=K\setminus\{x\in K\setsep x\text{ is isolated in }K\}$. Furthermore, for an ordinal number $\alpha>1$, let $D^\alpha(K)=D(D^\beta(K))$ if $\alpha=\beta+1$, and $D^\alpha(K)=\bigcap_{\beta<\alpha} D^\beta(K)$ if $\alpha$ is a limit ordinal. $K$ is \emph{perfect} if $D(K)=K$, and $K$ is said to be \emph{scattered} if there exists an ordinal $\alpha$ such that $D^\alpha(K)=\emptyset$. In this case the minimal ordinal $\alpha$ such that $D^\alpha(K)=\emptyset$ is called the \emph{height} of $K$ and is denoted by $\height(K)$. Note that $\height(K)$ is a successor ordinal and $D^{\height(K)-1}(K)$ is a finite set. Also, we recall that a compact metric space is scattered if and only if it is countable. Further, we note for each ordinal number $\alpha$, each $k\in\Nat$ and the ordinal interval $[0,\omega^{\alpha}k]$, we have $D^\alpha([0, \omega^{\alpha}k])=\{\omega^{\alpha}i\setsep i=1,\ldots,k\}$, thus $\height([0, \omega^{\alpha}k])=\alpha+1$. By the classical  Mazurkiewicz--Sierpi\'{n}ski classification, each countable compact space $K$ is homeomorphic to the ordinal interval $[0, \omega^{\alpha}k]$, where $\alpha=\height(K)-1$ and $k=\abs{D^\alpha(K)}$.

\part{Characterization of spaces isometric to \texorpdfstring{$\C(K)$}{C(K)} for \texorpdfstring{$K$}{K} countable}\label{part1}

The main results of this part are Theorem~\ref{thm:Intro2} together with a new isometric characterization of $\C(K)$ spaces, where $K$ is either the Cantor set or a countable infinite compact set. Our motivation comes from the fact that, when studying upper bounds on the Borel complexity of the isometry class of a Banach space, one must first find a suitable isometric characterization of that space. In many cases, this requires introducing a new characterization designed specifically for this purpose. In \cite[Theorem 4.13]{CDDK2} the authors obtained isometric characterization of the Banach space $c_0$, which suggests pursuing analogous ideas in our setting. A careful analysis of the proof of \cite[Theorem 4.13]{CDDK2} reveals a gap in the part of the argument where a condition ensuring that $X^*$ is isometric to $\ell_1$ is established. We address this issue by providing an argument that is useful in our setting, and by supplying details that fill the gap in the proof of \cite[Theorem 4.13]{CDDK2}. This is the content of the very short Section~\ref{sec:ell1Predual}. The next step is to exploit Szlenk derivatives similarly as it was done in \cite[Theorem 4.13]{CDDK2}. Accordingly, in Section~\ref{sec:szlenk} we compute the Szlenk derivatives for the unit ball in $\C(K)$ spaces with $K$ being a scattered compact space (not even necessarily metrizable) and, as a consequence, we obtain isometric characterization of $\C(\omega^\alpha\cdot k)$ spaces among separable $\C(K)$ spaces, see Theorem~\ref{thm:charactCkcountableInCk}. However, when trying to characterize the space $c$, we observed that the Szlenk derivative alone is insufficient to distinguish it from other $L_1$-preduals such as e.g. the space $c\oplus_\infty c_0$. After a non-negligible effort we obtained Theorem~\ref{thm:mainPart1Strong}, which implies Theorem~\ref{thm:Intro2} and provides a characterization of $\C(K)$ spaces with $K$ zero-dimensional within the class of $L_1$-preduals. Notably, our method applies even in the nonseparable setting. This constitutes the central and technically most demanding part of Part~\ref{part1}, and is developed in Section~\ref{sec:mainPart1}. As a byproduct of our proofs we also obtain an isometric characterization of $\C(2^\Nat)$. The isometric characterization of $\C(\omega^\alpha\cdot k)$ spaces is then obtained as a combination of Theorem~\ref{thm:charactCkcountableInCk} with Theorem~\ref{thm:mainPart1Strong}, see Corollary~\ref{cor:isometricCharacterizationCKcountable}.

We finish this introductory part with a short list of additional notation which we use below. Given a compact space $K$, we identify the dual of the Banach space $\C(K)$ with the space $M(K)$ of all regular signed Borel measures on $K$ endowed with the variation norm. Unless otherwise stated, the ball $B_{M(K)}$ is endowed with the weak$^*$ topology given by this duality. For $\mu \in M(K)$, its variation is denoted by $\abs{\mu}$.

\section{Conditions guaranteeing that \texorpdfstring{$X^*$}{X*} is isometric to \texorpdfstring{$\ell_1$}{ell1}}\label{sec:ell1Predual}

In this section we provide a condition, which we shall use later, guaranteeing that a separable Banach space has dual isometric to $\ell_1$, see Proposition~\ref{prop:sufficientForDualEll1}.
It also seems to be an appropriate place to fill in a gap from the proof of \cite[Theorem 4.13]{CDDK2}, see Proposition~\ref{prop:correction} below. The following basic tool will be used throughout without further explicit reference.

\begin{lemma}\label{lem:suffEll1Predual}Let $X$ be a separable infinite-dimensional $\mathcal{L}_{\infty,1+}$-space. Then the following conditions are equivalent.
\begin{itemize}
    \item $X^*$ is isometric to $\ell_1$,
    \item $X^*$ is separable,
    \item $X^*$ does not contain a subspace isometric to $L_1([0,1])$.
\end{itemize} 
\end{lemma}
\begin{proof}Recall that any $\mathcal{L}_{\infty,1+}$-space is $L_1$-predual (see \cite[\S 23 Theorem 2]{LaceyBook}) and for any separable infinite-dimensional $L_1$-predual $X$ we have either $X^*\equiv \ell_1$ or $X^*\equiv (\C[0,1])^*$ (see \cite[\S 22 Theorem 5]{LaceyBook}). Assuming that $X^*\equiv (\C[0,1])^*$, $X^*$ is not separable and it contains a subspace isometric to $L_1[0,1]$ (see e.g. \cite[p. 226]{LaceyBook}). On the other hand, if $X^*\equiv \ell_1$ then of course $X^*$ does not contain a subspace isometric to $L_1[0,1]$ (e.g. because $\ell_1$ has the Schur property while $L_1[0,1]$ does not).
\end{proof}

Recall that a Banach space $X$ has the \emph{the diameter two property} ($D2P$) if each nonempty weakly open subset of $B_X$ has diameter $2$. The following sufficient condition for the separability of $X^*$ is a corollary of the well-known fact that $L_1[0,1]$ has the $D2P$ (which follows, for example, from the fact that $L_1[0,1]$ has the Daugavet property and every Banach space with the Daugavet property has the $D2P$, see e.g. \cite[Theorem 1.2.1 and p. 354-355]{DaugavetBook}).

\begin{prop}\label{prop:sufficientForDualEll1}
    Let $X$ be a separable infinite-dimensional $\mathcal{L}_{\infty,1+}$-space and let $s_2^\alpha(B_{X^*})$ be subset of a finite-dimensional space for some ordinal $\alpha$. Then $X^*$ is isometric to $\ell_1$.
\end{prop}
\begin{proof}In order to get a contradiction, assume $X^*$ contains a subspace $Z$ isometric to $L_1[0,1]$. We pick any $w^*$-open set $U$ in $X^*$ intersecting $B_Z$. Then $U$ is weakly open in $X^*$, and so $U\cap Z$ is weakly open in $Z$ and $U\cap B_Z$ is weakly open subset of $B_Z$, and therefore $U\cap B_Z$ has diameter $2$ (because, as mentioned above, $Z\equiv L_1[0,1]$ has $D2P$). Thus, $B_Z\subseteq s_2(B_{X^*})$. Now, a simple inductive argument implies that $B_Z\subseteq s_2^\alpha(B_{X^*})$ for every ordinal $\alpha$, so $s_2^\alpha(B_{X^*})$ cannot be a subset of a finite-dimensional space.
\end{proof}

Now, we proceed to fill a gap from the proof of \cite[Theorem 4.13]{CDDK2}. The error in the proof of the mentioned result is caused by the fact that, under the assumption that $X$ is a separable $\mathcal{L}_{\infty,1+}$-space which for some $\varepsilon>0$ satisfies the equality $s_{2 \varepsilon}(B_{X^*}) = (1-\varepsilon)B_{X^*}$, it is deduced in \cite{CDDK2} using \cite[Proposition 3 and Theorem 1]{Lan06} that $X^*$ is separable. However, Proposition 3 from \cite{Lan06} could only be used to deduce that conclusion if the above equality for the Szlenk derivative was true for \textbf{each} $\varepsilon>0$. Thus, in order to correct the error in question, we need to prove the following. 

\begin{prop}\label{prop:correction}
    Let $X$ be a separable $\mathcal{L}_{\infty,1+}$-space and assume that there exists $\varepsilon>0$ such that 
    \[
    s_{2\varepsilon}(B_{X^*}) = (1-\varepsilon)B_{X^*}.
    \]
    Then $X^*$ is separable.
\end{prop}
\begin{proof}
Assume that $ X^{*} $ is not separable. In the same way as in the proof of Proposition~\ref{prop:sufficientForDualEll1}, we can find a subspace $ Z \subseteq X^{*} $ isometric to $ L_{1}[0, 1] $ and find out that $ B_{Z} \subseteq s_{2}(B_{X^{*}}) $. Then $ B_{Z} \subseteq s_{2}(B_{X^{*}}) \subseteq s_{2\epsilon}(B_{X^{*}}) = (1-\varepsilon)B_{X^{*}} $, which is not possible.
\end{proof}

\section{The Szlenk derivative of \texorpdfstring{$\C(K)$}{C(K)}, \texorpdfstring{$K$}{K} scattered}\label{sec:szlenk}
In this section, we prove a relation between the concepts of the Cantor-Bendixson derivative of a compact space $K$ and the Szlenk derivative of the ball $B_{M(K)}$, which we use afterwards in order to characterize $\C(\omega^\alpha\cdot k)$ spaces among other separable $\C(K)$ spaces, see Theorem~\ref{thm:charactCkcountableInCk}. While it is known that the Szlenk index of $\C(K)$ can be computed directly from the Cantor-Bendixson height of $K$, see \cite{CAUSEY_C(K)_index} or \cite{Causey_szlenk_hulls}, we were not able to find anywhere in literature an exact formula for the Szlenk derivative of $B_{M(K)}$ in terms of the Cantor-Bendixson derivative of $K$. 

Actually, the results that we prove in this section are stronger than we will need later. Indeed, we will apply Corollary~\ref{cor:SzlenkDerivative2} only in the case when the compact $K$ is countable. However, the proofs are not much more complicated in the more general setting. Before embarking on the results, let us recall that the sets
\[U[\mu, f_1, \ldots, f_n, \eta]=\Big\{\nu \in B_{M(K)} \setsep \abs{\mu(f_j)-\nu(f_j)}<\eta \text{ for every $j=1,\ldots,n$}\Big\},\]
where $\eta>0$ and $f_1, \ldots, f_n \in B_{\C(K)}$, form a basis of open neighborhoods of a measure $\mu$ in $B_{M(K)}$ endowed with the $w^*$ topology.

\begin{prop}
\label{prop:derivative_of_C(K)}
If $K$ is an infinite scattered compact space, then for each $0<\varepsilon\leq 1$,
\[s_{2 \varepsilon}(B_{\C(K)^*})=\{\mu \in B_{M(K)}\setsep \abs{\mu}(K \setminus D(K)) \leq 1-\varepsilon\}.\]
\end{prop}

\begin{proof}
``$\supseteq $'': We pick a measure $\mu \in B_{M(K)}$ satisfying $\abs{\mu}(K \setminus D(K)) \leq 1-\varepsilon$. We also fix $f_1, \ldots, f_n \in B_{\C(K)}$ and $\eta>0$.
We wish to show that $U[\mu, f_1, \ldots, f_n, \eta] \cap B_{M(K)}$
has diameter at least $2 \varepsilon$.

To this end, we first pick an arbitrary $0<\delta<\eta$, and we put $C=\frac{1-\norm{\mu}}{2}$. Using the regularity of the measure $\mu$, we find a compact set $F \subseteq K \setminus D(K)$ such that $\abs{\mu}(K \setminus (D(K) \cup F))<\frac{\delta}{6}$, and, by the standard Hahn-Jordan decomposition theorem and using the regularity of $\mu$ again, we find disjoint compact sets $L_s \subseteq D(K)$ for $s=1, 2$ such that $\abs{\mu}(D(K) \setminus (L_1 \cup L_2))<\frac{\delta}{6}$, $\mu_{|L_1} \geq 0$ and $\mu_{|L_2} \leq 0$. Then 
\[\abs{\mu}(K \setminus (F \cup L_1 \cup L_2))=\abs{\mu}(K \setminus (D(K)\cup F))+\abs{\mu}(D(K) \setminus (L_1 \cup L_2)) < \frac{\delta}{6}+\frac{\delta}{6}=\frac{\delta}{3}.\]

Next, we find an open cover $(U_{i}^{1})_{i=1}^{k_1}$ of the compact set $L_1$ such that for each $j=1, \ldots, n$, $i=1, \ldots, k_1$, and $x, y \in U_i^{1}$, it holds $\abs{f_j(x)-f_j(y)}<\frac{\delta}{3}$, $U_i^1 \cap F=\emptyset$ and $\overline{U_i^1} \cap L_2=\emptyset$. 
Similarly, we find an open cover $(U_{i}^{2})_{i=1}^{k_2}$ of the compact set $L_2$ such that for each $j=1, \ldots, n$, $i=1, \ldots, k_2$, and $x, y \in U_i^{2}$, it holds $\abs{f_j(x)-f_j(y)}<\frac{\delta}{3}$, $U_i^2 \cap F=\emptyset$ and $U_i^2 \cap \overline{\bigcup_{m=1}^{k_1} U_m^1}=\emptyset$. 

Further, for each $s=1, 2$ and $i=1, \ldots, k_{s}$, we find a point $x_i^s \in U_i^s$ isolated in $K$, and we also pick two more points $y_1, y_2 \in U_1^1$ isolated in $K$, such that all the points $x_1^1, \ldots, x_{k_1}^1, x_1^2, \ldots, x_{k_2}^2, y_1, y_2$ are distinct (notice that we may assume that each $U_i^s$ contains infinitely many isolated points of $K$, since every open set containing an accumulation point of $K$ has this property). Next, let for $s=1, 2$,
\[V_1^s=U_1^s \cap L_s \text{ and } V_i^s=(U_i^s \cap L_s) \setminus \bigcup_{m < i} U_m^s \text{ for } i=2, \ldots, k_s.\]
Notice that $L_s=\bigcup_{i=1}^{k_s} V_i^s$ for $s=1, 2$.

Now, we consider the measures
\[\nu_1=\mu_{|F}+\mu_{|L_2}+\sum_{i=1}^{k_1} \mu(V_i^1) \delta_{x_i^1}+C\delta_{y_1}-C\delta_{y_2}\]
and 
\[\nu_2=\mu_{|F}+\mu_{|L_1}+\sum_{i=1}^{k_2} \mu(V_i^2) \delta_{x_i^2}-C\delta_{y_1}+C\delta_{y_2}.\]
Then, we have 
\begin{equation}
\nonumber
\begin{aligned}
\norm{\nu_1} &\leq \abs{\mu}(F \cup L_2)+\sum_{i=1}^{k_1} \abs{\mu(V_i^1)}+2C
\leq \abs{\mu}(F \cup L_1 \cup L_2)+2C \leq \abs{\mu}(K)+2C=1,   
\end{aligned}
\end{equation}
and similarly $\norm{\nu_2} \leq 1$.

Next, we show that the measures $\nu_1, \nu_2$ belong to the set $U[\mu, f_1, \ldots, f_n, \eta]$. We show this for the measure $\nu_1$, the proof for the measure $\nu_2$ is completely analogous.
So, let $j \in \{1, \ldots, n\}$ be arbitrary. Then
\begin{equation}
\nonumber
\begin{aligned}
&\abs{\mu(f_j)-\nu_1(f_j)}=\abs{\int_{K \setminus (F \cup L_2)} f_j d\mu-\sum_{i=1}^{k_1} \mu(V_i^1)f_j(x_i^1)-Cf_j(y_1)+Cf_j(y_2)}=\\&=
\abs{\int_{K \setminus (F \cup L_1 \cup L_2)} f_j d\mu+\sum_{i=1}^{k_1} \int_{V_i^1} (f_j(x)-f_j(x_i^1)) d\mu(x)+C(f_j(y_2)-f_j(y_1))}
\leq \\& \leq
\abs{\mu} (K \setminus (F \cup L_1 \cup L_2))+\sum_{i=1}^{k_1} \abs{\mu}(V_i^1) \cdot \sup_{x \in V_i^1}\abs{f_j(x)-f_j(x_i^1)}+\\&+C\abs{f_j(y_2)-f_j(y_1)}
\leq \frac{\delta}{3}+\abs{\mu}(K) \frac{\delta}{3}+C\frac{\delta}{3} \leq \delta<\eta.
\end{aligned}
\end{equation}
Finally, we have
\[\nu_1-\nu_2 = \mu_{|L_2}-\mu_{|L_1}+ \sum_{i=1}^{k_1} \mu(V_i^1)\delta_{x_i^1}-\sum_{i=1}^{k_2} \mu(V_i^2)\delta_{x_i^2}+2C(\delta_{y_1} - \delta_{y_2}),\]
thus, since all the measures $\mu|_{L_1}, \mu|_{L_2}, \delta_{x_1^1}, \ldots, \delta_{x_{k_1}^1}, \delta_{x_1^2}, \ldots, \delta_{x_{k_2}^2}, \delta_{y_1}, \delta_{y_2}$ have disjoint compact supports, and since for each $s=1, 2$ and $i \in \{1, \ldots, k_s\}$, $\abs{\mu(V_i^s)}=\abs{\mu}(V_i^s)$, we have  
\begin{equation}
 \nonumber
 \begin{aligned}
\norm{\nu_1-\nu_2}&=\abs{\mu}(L_1)+\abs{\mu}(L_2)+\sum_{s=1}^2\sum_{i=1}^{k_s} \abs{\mu}(V_i^s)+4C=\\&=
2(\abs{\mu}(L_1 \cup L_2)+2C) \geq 2\Big(\abs{\mu}(D(K))-\frac{\delta}{6}+1-\abs{\mu}(K)\Big)
=\\&=2\big(1-\abs{\mu}(K \setminus D(K))\big)- \frac{\delta}{3} \geq 2\varepsilon- \frac{\delta}{3}.     
 \end{aligned}
\end{equation}
Since $\delta>0$ was arbitrary, this shows that $\diam(U[\mu, f_1, \ldots, f_n, \eta]\cap B_{M(K)})\geq 2\varepsilon$ and therefore $\mu\in s_{{2\varepsilon}}(B_{\C(K)^*})$.

\noindent``$\subseteq $'': Pick $\mu \in B_{M(K)}$ such that $\abs{\mu}(K \setminus D(K))>1-\ep$. Find $0<\varepsilon'<\ep$ and $F \subseteq K \setminus D(K)$ compact such that $\abs{\mu}(F)>1-\varepsilon'$. Since $F$ is a compact set consisting of isolated points, it is finite, say, $F=\{x_1, \ldots, x_n\}$. Fix $\delta>0$ such that $4\delta n+2\varepsilon'<2 \varepsilon$. We prove that $\diam(U[\mu, \chi_{x_1}, \ldots, \chi_{x_n}, \delta]\cap B_{M(K)})<2\varepsilon$. To this end, notice that for each $\nu \in U[\mu, \chi_{x_1}, \ldots, \chi_{x_n}, \delta]\cap B_{M(K)}$, 
\[\abs{\nu}(K \setminus F)\leq 1 - \abs{\nu}(F) = 1 - \sum_{i=1}^n |\nu(\{x_i\})|\leq 1 - \abs{\mu}(F) + \delta n\leq \delta n + \varepsilon'.\]
Thus, for arbitrary $\nu_1, \nu_2 \in U(\mu, \chi_{x_1}, \ldots, \chi_{x_n}, \delta)\cap B_{M(K)}$ we have 
\begin{equation}
\nonumber
\begin{aligned}
& \norm{\nu_1-\nu_2} \leq \sum_{i=1}^n \abs{\nu_1(\{x_i\})-\nu_2(\{x_i\})}+\abs{\nu_1}(K \setminus F)+\abs{\nu_2}(K \setminus F) \leq 4\delta n+2\varepsilon'<2 \varepsilon, 
\end{aligned}
\end{equation}
which finishes the proof.
\end{proof}

Combining the above result with a simple transfinite argument, we obtain the following corollary. Let us note that, for a compact set $F \subseteq K$, we identify a measure $\mu \in M(K)$ carried by $F$ with a member of $M(F)$.

\begin{cor}\label{cor:SzlenkDerivative2}
If $K$ is a scattered compact space, then for each ordinal $\alpha$,
\[s_2^{\alpha}(B_{\C(K)^*})=B_{M(D^{\alpha}(K))}.\] 
\end{cor}

\begin{proof}
We proceed by transfinite induction. The case $\alpha=0$ is trivial. When $\alpha=1$ the equality follows by Proposition \ref{prop:derivative_of_C(K)} in the case when $K$ is infinite, and in the finite case it follows by standard arguments that both sets are empty. 

If $\alpha=\beta+1$ is a successor ordinal and the statement is true for $\beta$, then again by Proposition \ref{prop:derivative_of_C(K)} we have 
\[s_2^{\alpha}(B_{\C(K)^*})=s_2(s_{2}^{\beta}(B_{\C(K)^*}))=s_2(B_{M(D^{\beta}(K))})=B_{M(D(D^{\beta}(K)))}=B_{M(D^{\alpha}(K))}.\]
Likely, if $\alpha$ is a limit ordinal and the statement holds for each $\beta<\alpha$, then 
\[s_2^{\alpha}(B_{\C(K)^*})=\bigcap_{\beta<\alpha} s_2^{\beta}(B_{\C(K)^*})=\bigcap_{\beta<\alpha} B_{M(D^{\beta}(K))}=B_{M(D^{\alpha}(K))}.\]
\end{proof}

Now, we are ready to formulate and prove the characterization of $\C(\omega^\alpha\cdot k)$ spaces among all separable $\C(K)$ spaces.

\begin{thm}\label{thm:charactCkcountableInCk}
Let $X$ be a separable Banach space, $\alpha<\omega_1$ countable nonzero ordinal and $k\in\Nat$. Then $X$ is isometric to $\C(\omega^\alpha\cdot k)$ if and only if the following conditions hold.
\begin{enumerate}[label=(\roman*)]
    \item\label{it:CkAny} $X$ is isometric to some $\C(K)$ space.
    \item\label{it:derivativeCharacterization} $s_2^\alpha(B_{X^*})$ is a nonempty subset of $k$-dimensional space, and if $k>1$ then it is not a subset of $(k-1)$-dimensional space.
\end{enumerate}
In particular, if $X$ is isometric to a $\C(K)$ space and it satisfies condition \ref{it:derivativeCharacterization} for some ordinal $\alpha$ and $k \in \Nat$, then $K$ is countable.
\end{thm}
\begin{proof}First, assume $X$ is a $\C(K)$ space satisfying condition \ref{it:derivativeCharacterization}. Since any $\C(K)$ space is a $\mathcal{L}_{\infty,1+}$-space, by Proposition~\ref{prop:sufficientForDualEll1} we obtain that $X^*$ is isometric to $\ell_1$, which implies $X$ is Asplund, therefore $K$ is scattered (see e.g. \cite[Theorem 14.25]{FHHMZ}) and since it is metrizable, it is countable (see e.g. \cite[Lemma 14.21]{FHHMZ}). By the Mazurkiewicz-Sierpi\'{n}ski classification (see e.g. \cite[Theorem 2.56]{BiortogBook}), there are a countable ordinal $\beta<\omega_1$ and $m\in\Nat$ with $X$ being isometric to $\C(\omega^\beta\cdot m)$. By Corollary~\ref{cor:SzlenkDerivative2}, we obtain that $\beta = \alpha$ (because $\beta<\alpha$ implies that $s_2^\alpha(B_{X^*})$ is not even one-dimensional while $\beta>\alpha$ implies $s_2^\alpha(B_{X^*})$ is not contained in a finite-dimensional space) and $m=k$. Thus, $X$ is isometric to $\C(\omega^\alpha\cdot k)$.

On the other hand, if $X$ is isometric to $\C(\omega^\alpha\cdot k)$, then it is of course a $\C(K)$ space and it satisfies condition \ref{it:derivativeCharacterization} by Corollary~\ref{cor:SzlenkDerivative2}.
\end{proof}

\section{Isometric characterizations of spaces \texorpdfstring{$\C(K)$}{C(K)}, \texorpdfstring{$K$}{K} zero-dimensional}\label{sec:mainPart1}

In this section we aim at characterization of those $L_1$-preduals which are isometric to a $\C(K)$ space with $K$ zero-dimensional. The main result of this section is Theorem~\ref{thm:mainPart1Strong} which in particular implies Theorem~\ref{thm:Intro2}. The key here is the condition \eqref{eq:newCondition} from the Introduction. Let us briefly explain the idea behind. The starting observation is that conditions obtained so far in previous sections do not distinguish between e.g. the Banach spaces $c$ and $c\oplus_\infty c_0$. Now, we may consider the condition
\begin{equation}\label{eq:babyVersion}
\exists a\in S_{X} \; \forall y\in S_X:\quad \max\{\|a-y\|,\|a+y\|\} = 2.
\end{equation}
We note that \eqref{eq:babyVersion} is satisfied in the space $X=c$ with $a$ being e.g. the constant one function, while one can prove that \eqref{eq:babyVersion} does not hold in the space $X=c\oplus_\infty c_0$. On the other hand, there are still some other isometric $\ell_1$ preduals which satisfy the condition above and still they are not isometric to $c$, those can be obtained e.g. by considering several isometric $\ell_1$ preduals described by Casini, Miglierina and Piasecki in \cite{CMP24}. After considering some variants of the condition above, we arrived at \eqref{eq:newCondition} which eventually led to the results we needed.

We start by explaining some notation. In this section, if $X$ is a Banach space with $X^*$ being isometric to some $L_1(\mu)$ space, by the $w^*$ topology on $L_1(\mu)$ we understand the topology generated by $X$, that is, $w^* = \sigma(L_1(\mu),X)$. A fact which we shall use frequently below is that given any Banach space $X$, by the Krein-Milman theorem we have $B_{X^*} = \overline{\co}^{w^*} (\ext B_{X^*})$ and therefore $\ext B_{X^*}$ is \emph{boundary} of $B_{X^*}$, that is,
\[
\|x\| = \max_{x^*\in \ext B_{X^*}} |x^*(x)|,\quad x\in X.
\]

We start with the most demanding part of the argument, which shows that whenever $X$ is an $L_1$-predual satisfying slightly weaker version of our condition \eqref{eq:newCondition}, the set $\ext B_{X^*}$ is $w^*$-closed, see Proposition~\ref{prop:closedExtremePoints}. To this end, we first establish the following.

\begin{lemma}\label{lem:bigInEveryExtremePoint}
Let $ X $ be an $L_1$-predual and $\varepsilon\in (0,\tfrac{1}{4})$. Let $ x \in S_{X} $ have the property that
$$ \forall y \in S_{X} : \max \{ \Vert x + y \Vert, \Vert x - y \Vert \} \geq 2 - \varepsilon. $$
Then $ |z^{*}(x)| \geq 1 - 4 \varepsilon $ for each $ z^{*} \in \ext B_{X^{*}} $.
\end{lemma}

\begin{proof}
Let us denote
$$ A = \{ z^{*} \in \mathrm{ext} \, B_{X^{*}} \setsep |z^{*}(x)| \geq 1 - 2\varepsilon \}, $$
$$ A^{+} = \{ z^{*} \in \mathrm{ext} \, B_{X^{*}} \setsep z^{*}(x) \geq 1 - 2\varepsilon \}, \quad A^{-} = \{ z^{*} \in \mathrm{ext} \, B_{X^{*}} \setsep z^{*}(x) \leq -(1 - 2\varepsilon) \}. $$
For $ y \in S_{X} $, we have
$$ \sup_{z^{*} \in \mathrm{ext} \, B_{X^{*}}} (|z^{*}(x)| + |z^{*}(y)|) = \max \{ \Vert x + y \Vert, \Vert x - y \Vert \} \geq 2 - \varepsilon. $$
The supremum can be taken over $ z^{*} \in A $, since for $ z^{*} \notin A $ we have $ |z^{*}(x)| + |z^{*}(y)| < 2 - 2\varepsilon $. This in turn implies that
$$ \sup_{z^{*} \in A} |z^{*}(y)| \geq 1 - \varepsilon. $$
For a general $ y \in X $, we then obtain $ \sup_{z^{*} \in A} |z^{*}(y)| \geq (1 - \varepsilon) \Vert y \Vert $, which by the Hahn-Banach theorem provides
$$ (1 - \varepsilon) B_{X^{*}} \subseteq \overline{\mathrm{co}}^{w^{*}} A. $$

Let us fix $ z^{*} \in \mathrm{ext} \, B_{X^{*}} $. Then there are $ u^{*}_{i} \in \mathrm{co} \, A^{+}, v^{*}_{i} \in \mathrm{co} \, A^{-} $ and $ \alpha_{i}, \beta_{i} \geq 0 $ with $ \alpha_{i} + \beta_{i} = 1 $ such that
$$ \alpha_{i} u^{*}_{i} + \beta_{i} v^{*}_{i} \to (1 - \varepsilon) z^{*} $$
in the $ w^{*} $-topology.
We can assume that there are $ u^{*} \in \overline{\mathrm{co}}^{w^{*}} A^{+}, v^{*} \in \overline{\mathrm{co}}^{w^{*}} A^{-} $ and $ \alpha, \beta \geq 0 $ with $ \alpha + \beta = 1 $ such that $ u^{*}_{i} \to u^{*}, v^{*}_{i} \to v^{*} $ and $ \alpha_{i} \to \alpha, \beta_{i} \to \beta $. It follows that
$$ \alpha u^{*} + \beta v^{*} = (1 - \varepsilon) z^{*}. $$
Note that $ u^{*}(x) \geq 1 - 2\varepsilon $ and $ v^{*}(x) \leq -(1 - 2\varepsilon) $, since $ u^{*} \in \overline{\mathrm{co}}^{w^{*}} A^{+}, v^{*} \in \overline{\mathrm{co}}^{w^{*}} A^{-} $. Recall that extreme points of $ B_{L_{1}(\mu)} $ are characteristic functions of atoms of $\mu$ multiplied by a suitable constant. Thus, if we identify $ X^{*} $ with $ L_{1}(\mu) $, there is an atom $ S $ of $\mu$ such that $ z^{*} $ is a multiple of the characteristic function $ \chi_{S} $. Without loss of generality, we suppose that $ z^{*}(s) > 0 $ for every $s\in S$. We can express $ u^{*} $ and $ v^{*} $ in the form $ u^{*} = u^{*}_{I} + u^{*}_{II} $ and $ v^{*} = v^{*}_{I} + v^{*}_{II} $, where $ u^{*}_{I}, v^{*}_{I} $ are multiples of $\chi_S$ and $ u^{*}_{II}, v^{*}_{II} $ are equal to zero on the set $ S $. Then for $s\in S$ we obtain $ \alpha u^{*}_{I}(s) + \beta v^{*}_{I}(s) = \alpha u^{*}(s) + \beta v^{*}(s) = (1 - \varepsilon) z^{*}(s) $, and it follows that $ u^{*}_{I}(s) \geq (1 - \varepsilon) z^{*}(s) $ or $ v^{*}_{I}(s) \geq (1 - \varepsilon) z^{*}(s) $, so $\|u^{*}_{I}\|\geq 1-\varepsilon$ or $\|v^{*}_{I}\|\geq 1-\varepsilon$. In the case that $ u^{*}_{I}(s) \geq (1 - \varepsilon) z^{*}(s) $, using that $S$ is atom and therefore both $u_I$ and $z^*$ are constant a.e. on $S$, we obtain 
\[
\Vert u^{*} - z^{*} \Vert = \Vert u^{*}_{I} - z^{*} \Vert + \Vert u^{*}_{II} \Vert = \Vert z^{*} \Vert - \Vert u^{*}_{I} \Vert + \Vert u^{*} \Vert - \Vert u^{*}_{I} \Vert \leq 2(1 - \Vert u^{*}_{I} \Vert) \leq 2\varepsilon ,
\]
and thus $ z^{*}(x) \geq u^{*}(x) - 2\varepsilon \geq 1 - 4\varepsilon $. Similarly, in the case that $ v^{*}_{I}(s) \geq (1 - \varepsilon) z^{*}(s) $, we obtain $ \Vert v^{*} - z^{*} \Vert \leq 2\varepsilon $, and thus $ z^{*}(x) \leq v^{*}(x) + 2\varepsilon \leq -(1 - 4\varepsilon) $.
\end{proof}

The following is the central result of this section.

\begin{prop}\label{prop:closedExtremePoints}
Let $ X $ be an $L_1$-predual. Assume further that the following formula holds  \begin{equation}\label{eq:FsigmaDeltaCondition}
    \forall\varepsilon>0:\; \overline{\aco}\Big\{a\in S_X\setsep \forall y\in S_X:\;\max\{\|a-y\|,\|a+y\|\}\geq 2-\varepsilon\Big\} = B_X.
\end{equation}
Then $\ext B_{X^{*}} $ is closed in the $ w^{*} $-topology.
\end{prop}

\begin{proof}In order to get a contradiction, pick $\psi \in \overline{\ext B_{X^*}}^{w^*}\setminus \ext B_{X^*}$. Then there are distinct $ u^{*}, v^{*} \in B_{X^{*}} $ and $ \alpha, \beta > 0 $ with $ \alpha + \beta = 1 $ such that
$$ \psi = \alpha u^{*} + \beta v^{*}. $$
Let us choose $ \varepsilon\in (0,\tfrac{1}{4}) $
such that $ 4\varepsilon \leq \min \{ \alpha, \beta \} (1 - 2\varepsilon) \Vert v^{*} - u^{*} \Vert $. We define
$$ z^{*} = \frac{1}{\Vert v^{*} - u^{*} \Vert} (v^{*} - u^{*}). $$
There is $ x_0 \in S_{X} $ such that $ z^{*}(x_0) > 1 - \varepsilon $. By the assumption, there is a finite set $A\subseteq S_X$ such that
\begin{itemize}
    \item for every $a\in A$ and $y\in S_X$ we have $\max\{\|a-y\|,\|a+y\|\}\geq 2-\varepsilon$,
    \item there exists $q \in\Rea^A$ with $\sum_{a\in A} |q_a| \leq 1$ and $\|\sum_{a\in A}q_a a-x_0\|<\varepsilon$.
\end{itemize}

Then
$$ z^{*} \Big( \sum_{a\in A} q_a a \Big) > 1 - 2\varepsilon, $$
so there must be $ a_0\in A \cup (-A)$ such that $ z^{*}(a_0) > 1 - 2\varepsilon $, that is,
$$ (v^{*} - u^{*})(a_0) > (1 - 2\varepsilon) \Vert v^{*} - u^{*} \Vert. $$
Since 
\[ \psi(a_0) = \alpha u^{*}(a_0) + \beta v^{*}(a_0) = v^{*}(a_0) - \alpha (v^{*}(a_0) - u^{*}(a_0)) < 1 - \alpha (1 - 2\varepsilon) \Vert v^{*} - u^{*} \Vert \leq 1 - 4\varepsilon \]
and similarly
\[ \psi(a_0) = u^{*}(a_0) + \beta (v^{*}(a_0) - u^{*}(a_0)) > -1 + \beta (1 - 2\varepsilon) \Vert v^{*} - u^{*} \Vert \geq - (1 - 4\varepsilon) ,\]
we obtain
$$ |\psi(a_0)| < 1 - 4\varepsilon. $$
On the other hand, we have $\max\{\|a+y\|,\|a-y\|\}\geq 2-\varepsilon$ for every $y\in S_X$ and $a\in A$. Thus, using Lemma~\ref{lem:bigInEveryExtremePoint} we obtain that $ |z^{*}(a_0)| \geq 1 - 4 \varepsilon $ for each $ z^{*} \in \mathrm{ext} \, B_{X^{*}} $, which implies that $ |\psi(a_0)| \geq 1 - 4 \varepsilon $, a contradiction.
\end{proof}

The last technical step is to observe that \eqref{eq:FsigmaDeltaCondition} holds in a $\C(K)$ space if and only if $K$ is zero-dimensional.

\begin{prop}\label{prop:CKcountable}
    Let $K$ be a compact space. Then for the Banach space $X:=\C(K)$ the following conditions are equivalent.
    \begin{enumerate}[label=(\roman*)]
        \item\label{it:zeroDim} $K$ is zero-dimensional.
        \item\label{it:FsigmaDeltaStrong} Condition \eqref{eq:newCondition}.
        \item\label{it:ZeroDimFSigmadelta} Condition~\eqref{eq:FsigmaDeltaCondition}.
        \item\label{it:exVariant_ZeroDimFSigmadelta} We have 
        \[
    \exists \varepsilon\in (0,1): \overline{\aco}\Big\{a\in S_X\setsep \max\{\|a-y\|,\|a+y\|\}\geq 2-\varepsilon\text{ for every $y\in S_X$}\Big\} = B_X.
    \]
    \end{enumerate}
\end{prop}
\begin{proof}\ref{it:zeroDim}$\Rightarrow$\ref{it:FsigmaDeltaStrong}: Assume $K$ is zero-dimensional. Consider the set
\[
J:=\{f\in \C(K)\setsep |f(x)|=1\text{ for every }x\in K\}.
\]
Recall that since $K$ is zero-dimensional, for any $\mu\in\C(K)^* = M(K)$ we have
\[
\|\mu\|:=\sup\Big\{\sum_{i\in F}|\mu(C_i)|\setsep \{C_i\}_{i\in F}\text{ is a finite partition of $K$ consisting of clopen sets}\Big\}
\]
and therefore, $\|\mu\| = \sup_{f\in J} |\mu(f)|$ (because given any finite partition $\{C_i\}_{i\in F}$ of $K$ consisting of clopen sets, we obtain $\sum_{i\in F}|\mu(C_i)| = \mu(f)$ for a function $f$ defined as $f = \sum_{i\in F}\sgn(\mu(C_i))\chi_{C_i}$). This by the Hahn-Banach theorem implies that $\overline{\aco} J = B_X$, because otherwise there would exist $f\in B_X\setminus \overline{\aco} J$ and $\mu\in X^*=M(K)$ with $\sup_{g\in J}|\langle \mu,g\rangle| < \mu(f)$, contradiction. Finally, notice that for every $j\in J$ and $y\in S_X$ we have $\max\{\|y-j\|,\|y+j\|\} = \sup_{x\in K}\{|y(x)-1|,|y(x)+1|\} = 2$. Since we have already proved that $\overline{\aco} J = B_X$, this finishes the proof that \ref{it:FsigmaDeltaStrong} holds.

\smallskip

\noindent \ref{it:FsigmaDeltaStrong}$\Rightarrow$\ref{it:ZeroDimFSigmadelta}$\Rightarrow$\ref{it:exVariant_ZeroDimFSigmadelta} is obvious.

\smallskip

\noindent \ref{it:exVariant_ZeroDimFSigmadelta}$\Rightarrow$\ref{it:zeroDim}: Assume that \ref{it:exVariant_ZeroDimFSigmadelta} holds with some $\varepsilon\in (0,1)$. In order to get a contradiction, assume $K$ is not zero-dimensional. Then, since $K$ is compact, $K$ is not hereditarily disconnected, that is, there exists a closed connected subset $S\subseteq K$ with at least two distinct points $x_0$ and $y_0$. Now, we pick a function $f\in S_{\C(K)}$ with $f(x_0)=-1$ and $f(y_0)=1$ and since $S$ is connected, we have $[-1,1]\subseteq f(S)$. Let us denote
\[
J:=\{h\in S_X\setsep \max\{\|h-g\|,\|h+g\|\}\geq 2-\varepsilon\text{ for every }g\in S_X\}.
\]
We shall note that every $h\in J$ satisfies that $|h(x)|\geq 1-\varepsilon$ for every $x\in K$.

Indeed, given $x\in K$ and $h\in J$, for any $U\in \mathcal{U}(x)$ (where by $\mathcal{U}(x)$ we denote the set of open neighborhoods of $x$) we pick function $g\in S_{\C(K)}$ with $g\equiv 0$ on $K\setminus U$ and then, since $\max\{\|h-g\|,\|h+g\|\}\geq 2-\varepsilon$, we obtain $y_U\in U$ with $|h(y_U)|\geq 1-\varepsilon$. Then the net $(y_U)_{U\in\mathcal{U}(x)}$ converges to $x$ and so we obtain $|h(x)|\geq 1-\varepsilon$.

Thus, since $S$ is connected, for any $h\in J$ we have either $h(S)\subseteq [1-\varepsilon,1]$ or $h(S)\subseteq [-1,-1+\varepsilon]$. Hence, we obtain $|h(x_0)-h(y_0)|\leq \varepsilon$ for every $h\in J$. Since the set $E:=\{f \in B_{\C(K)}\setsep |f(x_{0})-f(y_{0})| \leq \varepsilon\}$ is closed and absolutely convex we therefore obtain that $B_{\C(K)} = \overline{\aco} J \subseteq E $, which is not possible. Thus, $K$ must be zero-dimensional.
\end{proof}

Now, we are ready to formulate the result which we shall need in order to obtain later characterization of $\C(K)$ spaces with $K$ infinite countable compact. We suspect that the equivalence between conditions \ref{it:CKcountable} and \ref{it:weakStarClosed} below is known, but we did not find a direct reference.

\begin{thm}\label{thm:charakterizace_CKcountable}Let $X$ be a separable Banach space with $X^*$ isometric to $\ell_1$. Then the following conditions are equivalent.
\begin{enumerate}[label=(\roman*)]
    \item\label{it:CKcountable} There exists a countable compact space $K$ such that $X\equiv \C(K)$.
    \item\label{it:FsigmaDeltaCondition} Condition \eqref{eq:FsigmaDeltaCondition} holds.
    \item\label{it:weakStarClosed} Given a basis $(e_n^*)_{n=1}^\infty$ of $X^*$ that is $1$-equivalent to the canonical basis of $\ell_1$, the set $\{\pm e_n^*\setsep n\in\Nat\}$ is $w^*$-closed in $X^*$.
\end{enumerate}
\end{thm}
\begin{proof}Implication \ref{it:CKcountable}$\implies$\ref{it:FsigmaDeltaCondition} follows from Proposition~\ref{prop:CKcountable} and implication \ref{it:FsigmaDeltaCondition}$\implies$\ref{it:weakStarClosed} is Proposition~\ref{prop:closedExtremePoints}, because extreme points in $B_{\ell_1}$ are exactly $\{\pm e_n\setsep n\in\Nat\}\subseteq \ell_1$. Finally, assume that \ref{it:weakStarClosed} holds and put $K_0:=\{\pm e_n^*\setsep n\in\Nat\}$. By the Baire category theorem, there is $ z \in X $ such that $ e^{*}_{n}(z) \neq 0 $ for every $ n $. Then the set
$$ K = \{ \sigma_{n} e^{*}_{n} \setsep n \in \mathbb{N} \}, \quad \textrm{where $ \sigma_{n} = \mathrm {sgn} \, e^{*}_{n}(z) $,} $$
is also $ w^{*} $-closed, as $ K = \{ x^{*} \in K_{0} \setsep x^{*}(z) \geq 0 \} $. We claim that $ X $ is isometric to $ \C(K) $.

Let $ T : X \to \C(K) $ be given by $ (Tx)(\sigma_{n} e^{*}_{n}) = \sigma_{n} e^{*}_{n}(x) $. Then $ T $ is an isometric embedding. Moreover, the adjoint $ T^{*} : \C(K)^{*} \to X^{*} $ satisfies $ T^{*} \delta_{\sigma_{n} e^{*}_{n}} = \sigma_{n} e^{*}_{n} $, and it follows that $ T^{*} $ is an isometry as well. In particular, $ T^{*} $ is injective, hence $ T $ is surjective.
\end{proof}

In order to prove an analogous result for $L_1$-preduals, we shall use the following result which follows from \cite[Theorem 2.7]{SpurnyKalenda16}. We would like to thank J.~Spurn\'y for suggesting this strategy to prove the second part of Theorem~\ref{thm:mainPart1Strong}.

\begin{thm}[Kalenda, Spurn\'y]\label{thm:spurnyKalenda}
Let $ X $ be an $L_1$-predual such that $\ext B_{X^*}$ is $w^*$-closed and let $\widetilde{S}:(\ext B_{X^*},w^*)\to \Rea$ be continuous bounded function, which is odd, that is, $\widetilde{S}(x^*) = -\widetilde{S}(-x^*)$ for every $x^*\in \ext B_{X^*}$. Then, there exists a continuous affine odd function $S:(B_{X^*},w^*)\to \Rea$ extending $\widetilde{S}$.
\end{thm}

The following variant of Theorem~\ref{thm:charakterizace_CKcountable} for general $L_1$-preduals is the main result of this section.

\begin{thm}\label{thm:mainPart1Strong}Let $ X $ be an $L_1$-predual. Then the following conditions are equivalent.
\begin{enumerate}[label=(\roman*)]
    \item\label{it:CkZeroDim} There is a zero-dimensional compact space $ K $ such that $ X \equiv \C(K)$.
    \item\label{it:newConditionInThm} Condition \eqref{eq:newCondition}.
    \item\label{it:newConditionOldVersionInThm} Condition \eqref{eq:FsigmaDeltaCondition}.
\end{enumerate}
Moreover, $X$ is isometric to $\C(2^\Nat)$ if and only if $X$ is a separable $L_1$-predual with the Daugavet property satisfying condition \eqref{eq:FsigmaDeltaCondition}.
\end{thm}
\begin{proof}Implication \ref{it:CkZeroDim}$\implies$\ref{it:newConditionInThm} follows from Proposition~\ref{prop:CKcountable}, implication \ref{it:newConditionInThm}$\implies$\ref{it:newConditionOldVersionInThm} is trivial. Assume now that \ref{it:newConditionOldVersionInThm} holds. By Proposition~\ref{prop:closedExtremePoints}, the set $ \mathrm{ext} \, B_{X^{*}} $ is closed in the $ w^{*} $-topology. Obviously, condition \eqref{eq:FsigmaDeltaCondition} implies there is some $ x \in S_X $ such that
$$ \forall y \in S_{X} : \max \{ \Vert x + y \Vert, \Vert x - y \Vert \} \geq 2 - \tfrac{1}{8}. $$
We define
$$ K = \{ z^{*} \in \mathrm{ext} \, B_{X^{*}} \setsep z^{*}(x) \geq \tfrac{1}{2} \}. $$
This is indeed a compact set in the $w^*$ topology. By Lemma~\ref{lem:bigInEveryExtremePoint}, $ \mathrm{ext} \, B_{X^{*}} $ is the disjoint union of $ K $ and $ -K $. We claim that $ X $ is isometric to $ \C(K) $.

Let $ T : X \to \C(K) $ be given by $ (Ty)(z^{*}) = z^{*}(y) $, $y\in X$. Then $ T $ is an isometric embedding. We claim $T$ is surjective. Given $f \in \C(K)$, we define a function $g:\ext B_{X^*} \rightarrow \er$ by setting
\[g(z^*)=\begin{cases}
 f(z^*), \quad z^* \in K,\\
 -f(-z^*), \quad z^* \in -K.\\
\end{cases}\]
Then, $g$ is a continuous and odd function on $\ext B_{X^*}$. By Theorem~\ref{thm:spurnyKalenda}, we can extend $g$ to an odd $w^*$-continuous affine function on $B_{X^*}$, so, by the Banach-Dieudonn\'e theorem, to a restriction of an element $x \in X$  (considered as an element of $X^{**}$) to $B_{X^*}$. Now, it is easy to check that $Tx=f$. Indeed, given $z^* \in K$, we have 
\[T(x)(z^*)=z^*(x)=g(z^*)=f(z^*).\] 
Thus, $T$ is surjective and therefore it is an isometry between $X$ and $\C(K)$. Finally, from Proposition~\ref{prop:CKcountable} we obtain that $K$ is zero-dimensional.

The ``Moreover'' part follows from the classical result that a $\C(K)$ space has Daugavet property if and only if $K$ does not have isolated points, see e.g. \cite[p. 3]{DaugavetBook}. We just need to realize that a zero-dimensional compact metric space does not have isolated points if and only if it is homeomorphic to the Cantor set, see e.g. \cite[Theorem I.7.4]{Kechrisbook}.
\end{proof}

Finally, let us formulate the isometric characterization of $\C(\omega^\alpha\cdot k)$ spaces which follows directly from Theorem~\ref{thm:charactCkcountableInCk} and Theorem~\ref{thm:mainPart1Strong} (note that Theorem~\ref{thm:charakterizace_CKcountable} would be enough, because condition \ref{it:derivativeCharacterization} from Theorem~\ref{thm:charactCkcountableInCk} implies by Proposition~\ref{prop:sufficientForDualEll1} that $X^*$ is isometric to $\ell_1$).

\begin{cor}\label{cor:isometricCharacterizationCKcountable}Let $X$ be a separable Banach space, $\alpha<\omega_1$ countable nonzero ordinal and $k\in\Nat$. Then $X$ is isometric to $\C(\omega^\alpha\cdot k)$ if and only if it is an $L_1$-predual satisfying condition \eqref{eq:newCondition} and condition \ref{it:derivativeCharacterization} from Theorem~\ref{thm:charactCkcountableInCk}.
\end{cor}

\part{Borel complexities of \texorpdfstring{$\C(K)$}{C(K)} and \texorpdfstring{$\homeoclass{K}$}{<K>} for \texorpdfstring{$K$}{K} countable or the Cantor set}\label{part2}

The aim of this part is to obtain results needed for the proof of Theorem~\ref{thm:Intro1}. The content of Part~\ref{part2} is almost independent of Part~\ref{part1}; the only results from Part~\ref{part1} required here are summarized in Theorems~\ref{thm:part1Main1} and \ref{thm:part1Main2}, which we recall below, see Section~\ref{sec:upperBounds}.

To formulate and explain the main results of Part~\ref{part2}, we begin by introducing the framework in which we shall work.

\begin{definition}{\cite[Definition 2.1]{CDDK1}}\label{def:coding}
By $V$ we denote the countable $\Rat$-linear vector space of elements of $c_{00}$ having rational entries. By $\PP$ we denote the space of all pseudonorms on the vector space $V$. We often identify $\mu\in\PP$ with its extension to the pseudonorm on the space $c_{00}$, that is, on the $\Rea$-linear vector space of all finitely supported sequences of real numbers.

For every $\mu\in\PP$ we denote by $X_\mu$ the Banach space given as the completion of the quotient space $X/N$, where $X = (c_{00},\mu)$ and $N = \{x\in c_{00}\setsep \mu(x) = 0\}$. In what follows we often consider $V$ as a subspace of $X_\mu$, that is, we identify every $v\in V$ with its equivalence class $[v]_N\in X_\mu$.

By $\PP_\infty$ we denote the set of those $\mu\in\PP$ for which $X_\mu$ is infinite-dimensional Banach space, and by $\B$ we denote the set of those $\mu\in\PP_\infty$ for which the extension of $\mu$ to $c_{00}$ is a norm. We equip $\P$, $\PP_\infty$, and $\B$ with the subspace topologies induced by $\Rea^V$. For any $\N\in\{\P,\P_\infty,\B\}$, a subbasis of this topology is given by sets of the form $U[v,I]\cap \N$, where $U[v,I]:=\{\mu\in\Rea^V\setsep \mu(v)\in I\}$, with $v\in V$ and $I$ an open interval. Recall that the topologies on $\PP$, $\PP_\infty$ and $\B$ are Polish, we refer to \cite[Corollary 2.5]{CDDK1} for a proof.

Given an infinite-dimensional separable Banach space $X$ and $\N\in\{\PP_\infty,\B\}$, we put
\[
\isomtrclass[\N]{X}:=\{\mu\in\N\setsep X_\mu\equiv X\}.
\]
If it is obvious from the context what $\N$ is we write just $\isomtrclass{X}$. We say that $\isomtrclass[\N]{X}$ is the \emph{isometry class of $X$ in $\N$}, we consider it as a subset of the Polish space $\N$.
\end{definition}

More generally, given a class $\C$ of infinite-dimensional separable Banach spaces stable under isometries, we study the Borel complexity of $\langle \C\rangle^\N:=\{\mu\in\N\setsep X_\mu\in\C\}$ in the space $\N$, where $\N\in\{\P_\infty,\B\}$. It is not clear which of the two choices, $\PP_\infty$ or $\B$, is more natural; for this reason, we formulate and prove our results for both. In general, $\B \subseteq \PP_\infty$, and by \cite[Proposition 3.6]{CDDK1} there exists an $F_\sigma$-measurable map $\varphi : \PP_\infty \to \B$ such that $X_\mu \equiv X_{\varphi(\mu)}$ for every $\mu \in \PP_\infty$. Thus, the difference between $\langle \C\rangle^{\PP_\infty}$ and $\langle \C\rangle^{\B}$ is small; however, since we are interested in precise complexity bounds, it remains relevant. Since $\B \subseteq \PP_\infty$, any upper bound on the Borel complexity of $\langle \C\rangle^{\PP_\infty}$ immediately yields the same bound for $\langle \C\rangle^{\B}$. Consequently, it is easier to work with $\B$ for upper bounds and with $\PP_\infty$ for lower bounds, while treating the opposite cases may introduce nontrivial technical difficulties.

\begin{remark}\label{rem:admissible} As we have mentioned in the introduction, a standard approach is to code separable Banach spaces by elements of the standard Borel space \( SB(X) \). This framework allows one to determine whether a given class of separable Banach spaces is Borel; however, it does not provide a canonical topology on \( SB(X) \) and therefore the notion of the exact Borel complexity of a class becomes ambiguous. Godefroy and Saint-Raymond partially addressed this issue in \cite{GodSR} by introducing so-called admissible topologies. Still, there is no canonical choice of an admissible topology. This is the reason why we consider the coding presented in Definition~\ref{def:coding} as being more canonical and more convenient to work with. We refer the interested reader also to \cite[Introduction and Remark 1.3]{CDDK2}, where a more detailed discussion concerning our choice of coding may be found. More detailed comparisons and relations are studied in \cite{CDDK1}.
\end{remark}

Let us recall at this point some notions from the classical descriptive set theory, which we shall use to formulate our results.

\begin{notation}\label{not:dstBasics}We denote by $\boldsymbol{\Sigma}^0_{\alpha}$ and $\boldsymbol{\Pi}^0_{\alpha}$ the additive and multiplicative Borel classes of order $\alpha<\omega_1$ in Polish spaces. 
Given subsets $A\subseteq X$ and $B\subseteq Y$ of Polish spaces $X$ and $Y$, we say that $A$ is \emph{Wadge reducible} to $B$, in symbols $A\leq_W B$, if there is a continuous function $f:X\to Y$ such that $A = f^{-1}(B)$, we say that such $f$ is a \emph{continuous reduction from $A$ to $B$}. Given a class $\boldsymbol{\Gamma}$ of subsets of Polish spaces, we say that a subset $ A $ of a Polish space $ X $ is a $ D_{2}(\boldsymbol{\Gamma}) $ set if $ A = B \setminus C $ for some $ B, C \in \boldsymbol{\Gamma}(X) $. Further, we say that a set $B\subseteq Y$ is $\boldsymbol{\Gamma}$-hard if $A\leq_W B$ whenever $A\in\boldsymbol{\Gamma}(X)$ and $X$ is Polish and zero-dimensional, and $\boldsymbol{\Gamma}$-complete if $B$ is $\boldsymbol{\Gamma}$-hard and in addition belongs to $\boldsymbol{\Gamma}(Y)$.

It is well-known that if $A\subseteq X$ is $\boldsymbol{\Pi}^0_{\alpha}$-complete, then $A\in \boldsymbol{\Pi}^0_{\alpha}(X)$ but $A\notin \boldsymbol{\Sigma}^0_{\alpha}(X)$. Moreover, if $A$ is  $D_2(\boldsymbol{\boldsymbol{\Pi}^0_{\alpha}})$-complete, then $A$ can be expressed as the intersection of a set in $\boldsymbol{\boldsymbol{\Pi}^0_{\alpha}}(X)$ with a set in $\boldsymbol{\boldsymbol{\Sigma}^0_{\alpha}}(X)$, but it cannot be written as the union of such sets. We refer the reader e.g. to \cite[Section 22]{Kechrisbook} for more details.
\end{notation}

The main result of Part~\ref{part2} is the following, from which Theorem~\ref{thm:Intro1} immediately follows.

\begin{thm}\label{thm:main1Part2}Let $\beta$ be either $0$ or a countable limit ordinal and let $n\in\Nat\cup\{0\}$ be such that $\beta+n>0$. Pick $\N\in\{\PP_\infty,\B\}$.
\begin{enumerate}[label=(\roman*)]
    \item\label{it:kIsOneComplete} The isometry class of $\C(\omega^{\beta+n})$ in $\N$ is $\boldsymbol{\Pi}^0_{\beta+2n+1}$-complete.
    \item\label{it:kIsNotOneComplete}  The isometry class of $\C(\omega^{\beta+n}\cdot k)$ in $\N$ is $D_2(\boldsymbol{\Pi}^0_{\beta+2n+1})$-complete for every $k\in\Nat$ with $k\geq 2$.
    \item\label{it:cantorComplete} The isometry class of $\C(2^\Nat)$ in $\N$ is $\boldsymbol{\Pi}^0_3$-complete.
\end{enumerate}
\end{thm}

Apart from the above we also prove an analogous statement concerning the homeomorphism classes of compact spaces. The notation we use in this context is the following.

\begin{notation}
    Given a topological space $X$, we denote by $\K(X)$ the space of all compact subspaces of $X$ endowed with the Vietoris topology (that is, topology with subbasis given by sets $\{K\in\K(X)\setsep K\cap U\neq \emptyset\}$ and $\{K\in \K(X)\setsep K\subseteq U\}$ for $U\subseteq X$ open). It is well-known that if $X$ is a metrizable compact space, then $\K(X)$ is a metrizable compact space as well and if $X$ is moreover uncountable, any zero-dimensional metrizable compact space does embed into $X$. We refer the interested reader to \cite[Sections 4, 6 and 7]{Kechrisbook} for more details.

    Further, given two compact spaces $K$ and $L$, we write $K\sim L$ if they are homeomorphic. Moreover, if $X$ is uncountable metrizable compact space and $K$ is zero-dimensional metrizable compact, we put $\homeoclass{K}^X:=\{L\in \K(X)\setsep L\sim K\}$ and we say that $\homeoclass{K}^X$ is the \emph{homeomorphism class of $K$ in $\K(X)$}, we understand it as a (necessarily nonempty) subset of the Polish space $\K(X)$. In the case that $X$ is clear from the context we write $\homeoclass{K}$ instead of $\homeoclass{K}^X$.
\end{notation}

Our main result concerning homeomorphism classes of compact spaces is the following.

\begin{thm}\label{thm:main2Part2}
Let $X$ be an uncountable metrizable compact space, $\beta$ be either $0$ or a countable limit ordinal and let $n\in\Nat\cup\{0\}$ be such that $\beta+n>0$.
\begin{enumerate}[label=(\roman*)]
    \item The homeomorphism class of $[0,\omega^{\beta+n}]$ in $\K(X)$ is $\boldsymbol{\Pi}^0_{\beta+2n+1}$-complete.
    \item The homeomorphism class of $[0,\omega^{\beta+n}\cdot k]$ in $\K(X)$ is $D_2(\boldsymbol{\Pi}^0_{\beta+2n+1})$-complete for every $k\in\Nat$ with $k\geq 2$.
\end{enumerate}
\end{thm}
\noindent We note that for the Cantor set an analogy of Theorem~\ref{thm:main1Part2}\ref{it:cantorComplete} for the homeomorphism classes does not hold, as it follows from known results that the homeomorphism class of the Cantor set is $G_\delta$, see Remark~\ref{rem:cantorHomeo} for more details.

Apart from the above main results, we also obtain some additional consequences of some parts of our proofs. The first one is related to the open problem of whether the class of Banach spaces isometric to a $\C(K)$ space for some compact $K$ is Borel, see Question~\ref{que:CKBorel}.

\begin{thm}\label{thm:main3Part2}Let $\N\in\{\PP_\infty,\B\}$. Let us denote by $\C$ the class of Banach spaces isometric to some $\C(K)$ space with $K$ zero-dimensional. Then the set $\{\mu\in \N\setsep X_\mu\in \C\}$ is $\boldsymbol{\Pi}^0_3$-complete.
\end{thm}

Another notable consequence of our methods is that the class of Banach spaces with a summable Szlenk index is $\boldsymbol{\Sigma}^0_\omega$-hard (see Corollary~\ref{cor:szlenk}),
thereby answering \cite[Question 6]{CDDK2} in the negative.

Let us outline the content of individual sections. In Section~\ref{sec:upperBounds} we recall what theorems from Part~\ref{part1} are needed and we deduce from those the upper estimates of our Borel classes including the proof of of the upper estimate from  Theorem~\ref{thm:main3Part2}. In Section~\ref{sec:cantorLower} we obtain the lower bound for the Cantor space and deduce Theorem~\ref{thm:main1Part2}\ref{it:cantorComplete} and Theorem~\ref{thm:main3Part2}. In Section~\ref{sec:lowBoundHom} we build on the works of Cenzer and Mauldin \cite{CM82, CM83} and prove Theorem~\ref{thm:main2Part2}, from which the lower bound in Theorem~\ref{thm:main1Part2}\ref{it:kIsOneComplete} and \ref{it:kIsNotOneComplete} can be quite easily deduced for the case of $\N=\PP_\infty$ and also for the case of $\N=\B$ and compacta of infinite height. In Section~\ref{sec:betaKonecne} we develop a method which enables us to obtain the lower bound in Theorem~\ref{thm:main1Part2}\ref{it:kIsOneComplete} and \ref{it:kIsNotOneComplete} for the case of $\N=\B$ and compacta of finite height. We consider this part as the most demanding step in our proof of the lower bound in Theorem~\ref{thm:main1Part2}. Finally, in Section~\ref{sec:appl} we demonstrate how to derive the proof of Theorem~\ref{thm:main1Part2} from the results of the preceding sections and present further applications.

We conclude the introduction to Part~\ref{part2} by presenting several preliminary results. The first establishes a connection between Theorem~\ref{thm:main1Part2} and Theorem~\ref{thm:main2Part2}. It is a generalization of \cite[Example 2.12]{CDDK1}; the proof is essentially identical, requiring only the replacement of the space $[0,1]^\Nat$ with the space $X$. We therefore omit the details.

\begin{lemma}\label{lem:Ck}Let $X$ be a compact metric space. Then there exists a continuous mapping $\rho\colon\mathcal K(X)\rightarrow\mathcal P$ such that $X_{\rho(K)}\equiv \C(K)$ for every $K\in\mathcal K(X)$, where by $\C(\emptyset)$ we understand the trivial space $\{0\}$.
\end{lemma}

The reason why we mention this result is the following immediate consequence.

\begin{cor}\label{cor:Ck}Let $X$ be an uncountable metrizable compact space and let $K$ be an infinite zero-dimensional compact metric space. Let $\boldsymbol{\Gamma}$ be one of the classes $\boldsymbol{\Pi}^0_\alpha$ or $D_2(\boldsymbol{\Pi}^0_\alpha)$ for some countable ordinal $\alpha\geq 1$.
\begin{enumerate}[label=(\roman*)]
    \item If $\isomtrclass[\PP_\infty]{\C(K)}$ belongs to the class $\boldsymbol{\Gamma}$, then the same holds for $\isomtrclass[\B]{\C(K)}$ and also for $\homeoclass{K}^X$.
    \item If $\homeoclass{K}^X$ is $\boldsymbol{\Gamma}$-hard, then the same holds for $\isomtrclass[\PP_\infty]{\C(K)}$.
\end{enumerate}
\end{cor}
\begin{proof}
   Follows directly from definitions, Lemma~\ref{lem:Ck} and the fact that $\B\subseteq \PP_\infty$. 
\end{proof}

We conclude by collecting further preliminaries from descriptive set theory, focusing on general principles for proving that sets are $\boldsymbol{\Gamma}$-hard or $\boldsymbol{\Gamma}$-complete; these will be used throughout the subsequent sections.

\subsection*{Hard and complete sets} We start by recalling the following technical tool, which is well-known and its variant appears, for example, as \cite[Lemma 6]{CM83}, although the proof is omitted there. For the reader’s convenience, we therefore provide a brief sketch of the proof below.

\begin{lemma}\label{lem:disjointUnion}Let $X$ be a zero-dimensional Polish space, $\alpha\geq 2$ a countable ordinal and suppose $(\beta_n)_{n\in\Nat}$ is a non-decreasing sequence of ordinals with $\sup\{\beta_n + 1\setsep n\in\Nat\} = \alpha$. Then for every $A\in \boldsymbol{\Sigma}_\alpha^0(X)$ there is a sequence $(B_n)_{n\in\Nat}$ of pairwise disjoint sets such that $A=\bigcup_{n\in\Nat} B_n$ and $B_n\in \boldsymbol{\Pi}_{\beta_n}^0(X)$ for every $n\in\Nat$.
\end{lemma}
\begin{proof}Since $X$ is zero-dimensional and Polish, any open set may be written as a disjoint union of countably many clopen sets. Hence, if we denote by $\boldsymbol{\Pi}_{0}^0(X)$ the set of all clopen subsets of $X$, the statement of the lemma holds for $\alpha=1$. We proceed further by induction. The case of $\alpha$ being a limit ordinal is easy, so it suffices to handle the case of successor ordinals. Pick $A\in \boldsymbol{\Sigma}^0_{\alpha + 1}(X)$ and without loss of generality assume $\beta_n=\alpha$ for every $n\in\Nat$. Then $A = \bigcup_{n\in\Nat} A_n$ for some $A_n\in \boldsymbol{\Pi}^0_\alpha(X)$, $n\in\Nat$. Replacing each $A_n$ by $\bigcup_{i=1}^{n} A_i$, we may assume $A_n\subseteq A_{n+1}$, $n\in\Nat$. By the inductive assumption, for every $n\in\Nat$, there is a pairwise disjoint sequence $(B_{n,k})_{k\in\Nat}$ of $\boldsymbol{\Pi}^0_{\gamma_k}(X)$-sets for some $\gamma_k<\alpha$ with $X\setminus A_n = \bigcup_{k\in\Nat} B_{n,k}$. Hence, 
\[
A = A_1\cup\bigcup_{n\in\Nat} (A_{n+1}\setminus A_n) = A_1\cup\bigcup_{n\in\Nat} \bigcup_{k\in\Nat} A_{n+1}\cap B_{n,k}
\]
and $A_{n+1}\cap B_{n,k}\in \boldsymbol{\Pi}^0_{\alpha}(X)$ for every $n,k\in\Nat$. Since the sets $A_1$ and $A_{n+1}\cap B_{n,k}$, $n,k\in\Nat$, are pairwise disjoint, this finishes the proof.
\end{proof}

The following presents a method of obtaining $\boldsymbol{\Pi}_{\alpha+1}^0$-complete sets.

\begin{lemma}
\label{lem:completeSetInductionStep}
Let $\alpha\geq 2$ be a countable ordinal and $(\alpha_m)_{m=1}^\infty$ be a non-decreasing sequence of ordinals $\ge1$ such that $\sup\{\alpha_m+1\setsep m\in\Nat\}=\alpha$.
For every $m\in\Nat$, let $X_m$ be a zero-dimensional Polish space and  let $U_m\subseteq X_m$ be a $\boldsymbol{\Pi}_{\alpha_m}^0$-complete set (resp. $\boldsymbol{\Pi}_{\alpha_m}^0$-hard set).
Then the set
\[
U:=\Big\{(x_m)_{m=1}^\infty\in\prod_{m\in\Nat}X_m\setsep x_m\in U_m\text{ for infinitely many $m$'s}\Big\}
\]
is $\boldsymbol{\Pi}_{\alpha+1}^0$-complete (resp. $\boldsymbol{\Pi}_{\alpha+1}^0$-hard) subset of the zero-dimensional Polish space $\prod_{m\in\Nat}X_m$.
\end{lemma}
\begin{proof}
The fact that $U$ is a $\boldsymbol{\Pi}_{\alpha+1}^0$-set whenever $U_m$ are $\boldsymbol{\Pi}_{\alpha_m}^0$-sets is obvious, so it suffices to prove that whenever $U_m$ are $\boldsymbol{\Pi}_{\alpha_m}^0$-hard then $U$ is $\boldsymbol{\Pi}_{\alpha+1}^0$-hard.

Let $\Nat=\bigcup_{i=1}^\infty N_i$ be a decomposition of natural numbers into infinitely many infinite sets.
Now let $Y$ be a zero-dimensional Polish space and let $A\subseteq Y$ be a $\boldsymbol{\Pi}_{\alpha+1}^0$-set.
Then there are increasing sequences $(m_i^j)_{j=1}^\infty$, $i\in\Nat$, of natural numbers and sets $A_i\subseteq Y$, $i\in\Nat$, and $A_i^j\subseteq Y$, $i,j\in\Nat$, such that:
\begin{itemize}
    \item $A=\bigcap_{i=1}^\infty A_i$,
    \item $A_{i+1}\subseteq A_i$ for every $i\in\Nat$,
    \item $A_i=\bigcup_{j=1}^\infty A_i^j$ for every $i\in\Nat$,
    \item $A_i^j$, $j\in\Nat$, are pairwise disjoint for every $i\in\Nat$,
    \item $N_i = \{m_i^j\setsep j\in\Nat\}$ for every $i\in\Nat$,
    \item $A_i^j$ is a $\boldsymbol{\Pi}_{\alpha_{m_i^j}}^0$-subset of $Y$
\end{itemize}
(apply Lemma~\ref{lem:disjointUnion} to satisfy the fourth condition, the other conditions are easy to ensure).
For every $i,j\in\Nat$, fix a continuous reduction $F_i^j\colon Y\to X_{m_i^j}$ from $A_i^j$ to $U_{m_i^j}$. Define $F\colon Y\to\prod_{m\in\Nat}X_m$ by $F(y)(m)=F_i^j(y)$, $y\in Y$,
where $i,j\in\Nat$ are chosen such that $m=m_i^j$. It is straightforward to verify that $F$ is a continuous reduction from $A$ to $U$. This proves that $U$ is $\boldsymbol{\Pi}_{\alpha+1}^0$-hard.
\end{proof}

Now we will concentrate on methods of obtaining $D_2(\boldsymbol{\Pi}_\xi^0)$-complete sets. The following basic observation is most probably well-known and was noticed e.g. in \cite[Section 7]{CDM05}
for $\xi=3$. We give here the proof for the convenience of the reader.

\begin{lemma}\label{lem:firstDifComplete}Let $X$ be a Polish space, $\xi \geq 1$ and let $A\subseteq X$ be a $\boldsymbol{\Pi}_\xi^0$-complete set or a $\boldsymbol{\Sigma}_\xi^0$-complete set. Then $A\times A^c$ is a $D_2(\boldsymbol{\Pi}_\xi^0)$-complete set.
\end{lemma}
\begin{proof}Obviously, we have $A\times A^c\in D_{2}(\boldsymbol{\Pi}_\xi^0)(X\times X) $, we need to show it is $D_2(\boldsymbol{\Pi}_\xi^0)$-hard. Since $A$ is $\boldsymbol{\Pi}_\xi^0$-complete if and only if $A^c$ is $\boldsymbol{\Sigma}_\xi^0$-complete, we may without loss of generality assume that $A$ is $\boldsymbol{\Pi}_\xi^0$-complete.
Let $Z$ be a zero-dimensional Polish space and $C,D\in \boldsymbol{\Pi}_\xi^0(Z)$. Then there exist $f:Z\to X$ continuous with $f^{-1}(A)=C$ and $g:Z\to X$ continuous with $g^{-1}(A^c)=D^c$.  Consider now the continuous mapping $f\triangle g:Z\to X\times X$ given as $(f\triangle g)(x):=(f(x),g(x))$ for $x\in Z$. Then $(f\triangle g)^{-1}(A\times A^c) = \{x\in Z\setsep f(x)\in A\; \&\; g(x)\in A^c\} = C\setminus D$ and, since $C,D$ were arbitrary, this shows that $A\times A^c$ is a $D_2(\boldsymbol{\Pi}_\xi^0)$-complete set.
\end{proof}

The method of constructing $D_2(\boldsymbol{\Pi}_{\alpha+1}^0)$-complete sets is the following (in the proposition below we also recall well-known methods of constructing $\boldsymbol{\Pi}_{\alpha+1}^0$-complete and $\boldsymbol{\Sigma}_{\alpha+1}^0$-complete sets).

\begin{prop}\label{prop:difHard}Let $X$ be a zero-dimensional Polish space, $\alpha\geq 1$ a countable ordinal and $k\in\Nat$. If $U\subseteq X$ is $\boldsymbol{\Pi}_\alpha^0$-complete set (resp. $\boldsymbol{\Pi}_\alpha^0$-hard set), then the following holds.
\begin{enumerate}[label=(\alph*)]
    \item\label{it:PiComplete} The set $A_U:=\{x\in X^\Nat\setsep \forall n\in\Nat\; x_n\notin U\}$ is $\boldsymbol{\Pi}_{\alpha+1}^0$-complete\\ (resp. $\boldsymbol{\Pi}_{\alpha+1}^0$-hard).
    \item\label{it:SigmaComplete} The set $(A_U)^c:=\{x\in X^\Nat\setsep \exists n\in\Nat\; x_n\in U\}$ is $\boldsymbol{\Sigma}_{\alpha+1}^0$-complete\\ (resp. $\boldsymbol{\Sigma}_{\alpha+1}^0$-hard).
    \item\label{it:DifComplete} The set $B_{U,k}:=\{x\in X^\Nat\setsep |\{n\in\Nat\setsep x_n\in U\}| = k\}$ is $D_2(\boldsymbol{\Pi}_{\alpha+1}^0)$-complete\\ (resp. $D_2(\boldsymbol{\Pi}_{\alpha+1}^0)$-hard).
\end{enumerate}
\end{prop}
\begin{proof}In order to shorten the notation we shall write below simply $A$, $A^c$ and $B_k$ instead of $A_U$, $(A_U)^c$ and $B_{U,k}$.

The items \ref{it:PiComplete} and \ref{it:SigmaComplete} are well-known, so in this case we just give a sketch of the proof. We shall suppose $U$ is $\boldsymbol{\Pi}_\alpha^0$-complete, the case when $U$ is only $\boldsymbol{\Pi}_\alpha^0$-hard is proved in the same way. It suffices to prove \ref{it:SigmaComplete}, because then \ref{it:PiComplete} follows by the simple observation that complement of a $\boldsymbol{\Sigma}_{\alpha+1}^0$-complete set is $\boldsymbol{\Pi}_{\alpha+1}^0$-complete. The fact that $A^c$ is a $\boldsymbol{\Sigma}_{\alpha+1}^0$-set is easy to check. In order to see that $A^c$ is $\boldsymbol{\Sigma}_{\alpha+1}^0$-hard, given any zero-dimensional Polish space $Z$ and $C\in\boldsymbol{\Sigma}_{\alpha+1}^0(Z)$ we find a sequence $(C_n)_{n\in\Nat}$ of $\boldsymbol{\Pi}_\alpha^0$-sets with $C=\bigcup_{n\in\Nat} C_n$ and, using that $U$ is $\boldsymbol{\Pi}_{\alpha}^0$-hard, we find continuous functions $f_n:Z\to X$ such that $f_n^{-1}(U) = C_n$, now we easily observe that the continuous function $f:Z\to X^\Nat$ defined as $f(x):=(f_n(x))_{n\in\Nat}$ satisfies $f^{-1}(A^c) = C$, hence $C\leq_W A^c$.

Now, let us prove \ref{it:DifComplete} for the case that $U$ is $\boldsymbol{\Pi}_\alpha^0$-complete. In order to prove that $B_{k}$ is $D_2(\boldsymbol{\Pi}_{\alpha+1}^0)$-hard, we shall show that $A^c\times A\leq_W B_{k}$, this is sufficient by Lemma~\ref{lem:firstDifComplete}. We start with the following.

\begin{claim}\label{claim:BHard}There exists a continuous function $g:X^\Nat\to X^\Nat$ such that for any $x\in X^\Nat$ we have
\[x\in A\Leftrightarrow g(x)\in A\Leftrightarrow g(x)\notin B_{1}.\]
\end{claim}
\begin{proof}[Proof of Claim~\ref{claim:BHard}] Using Lemma~\ref{lem:disjointUnion} we find a sequence of pairwise disjoint $\boldsymbol{\Pi}_\alpha^0$-sets $(A_n)_{n\in\Nat}$ with $A^c = \bigcup_{n\in\Nat} A_n$ and using that $U$ is $\boldsymbol{\Pi}_\alpha^0$-hard, we find continuous mappings $g_n:X^\Nat\to X$ with $g_n^{-1}(U)=A_n$, $n\in\Nat$. Now, we define continuous function $g:X^\Nat\to X^\Nat$ by the formula $g(x):=\big(g_n(x)\big)_{n\in\Nat}$, $x\in X^\Nat$. Then, since $\{A_n\}_n$ are pairwise disjoint sets, $x\notin A$ iff there exists unique $n\in\Nat$ with $x\in A_n$ iff there exists unique $n\in\Nat$ with $g_n(x)\in U$ iff $g(x)\in B_{1}$. On the other hand, $x\in A$ iff for every $n\in\Nat$ we have $x\notin A_n$ iff for every $n\in\Nat$ we have $g_n(x)\notin U$ iff $g(x)\in A$.
\end{proof}

Now, pick any $u\in U$ and let us define the continuous mapping $f:X^\Nat\times X^\Nat\to X^\Nat$ for $x,y\in X^\Nat$ by the formula $f(x,y)(i)=u$ if $i=1,\ldots,k-1$ and for $n\in\Nat\cup\{0\}$ we put
\[
f(x,y)\big(k-1 + 3n + i\big):=\begin{cases} 
y(n+1) & \text{ if }i=1,2,\\
g(x)(n+1) & \text{ if }i=3.
\end{cases}
\]
If $y\in X^\Nat\setminus A$, then there exists $n\in\Nat$ with $y(n)\in U$ and therefore $f(x,y)\big(j)\in U$ for every $j\in \{1,\ldots,k-1,k-1+3(n-1) + 1,k-1+3(n-1) + 2\}$, so $f(x,y)(j)\in U$ for at least $k+1$ indices and therefore $f(x,y)\notin B_{k}$. Similarly, if $x\in A$ and $y\in A$, then by Claim~\ref{claim:BHard} we have $g(x)\in A$ and therefore $f(x,y)(j)\notin U$ for every $j\geq k$ and so $f(x,y)\notin B_{k}$. Thus, we have proved that $(x,y)\notin A^c\times A$ implies $f(x,y)\notin B_k$. On the other hand, if $(x,y)\in A^c\times A$, then using Claim~\ref{claim:BHard} we have $g(x)\in B_1$ and $y\in A$, so there is exactly one index $j\geq k$ such that $f(x,y)(j)\in U$ and therefore $f(x,y)\in B_k$. This proves that $f^{-1}(B_k)=A^c\times A$, so $A^c\times A\leq_W B_k$ and $B_k$ is $D_2(\boldsymbol{\Pi}_{\alpha+1}^0)$-hard.

Now, let us observe that $B_k$ is a $D_2(\boldsymbol{\Pi}_{\alpha+1}^0)$-set. First, note that for any $k\in\Nat$ the set $B_{\geq k}:= \{x\in X^\Nat\setsep |\{n\in\Nat\setsep x_n\in U\}| \geq k\}$ is $\boldsymbol{\Sigma}_{\alpha+1}^0$. Indeed, $B_{\geq k}$ can be written as a countable union over finite sets $F\subseteq \Nat$ containing $k$ points elements of the sets $\bigcap_{n\in F}(\pi_n)^{-1}(U)$, so $B_{\geq k}$ is countable union of $\boldsymbol{\Pi}_\alpha^0$ sets and therefore it is a $\boldsymbol{\Sigma}_{\alpha+1}^0$ set. Since $B_k = B_{\geq k}\cap (B_{\geq k+1})^c$, this shows that $B_k$ is a $D_2(\boldsymbol{\Pi}_{\alpha+1}^0)$-set for every $k\in\Nat$. This finishes the proof for the case that $U$ is $\boldsymbol{\Pi}_\alpha^0$-complete set.

Finally, let us prove \ref{it:DifComplete} under the assumption $U$ is $\boldsymbol{\Pi}_\alpha^0$-hard. Pick a $\boldsymbol{\Pi}_\alpha^0$-complete set $V\subseteq X$, by the already proven part we know that $B_{V,k}$ is $D_2(\boldsymbol{\Pi}_{\alpha+1}^0)$-complete. Since $U$ is $\boldsymbol{\Pi}_\alpha^0$-hard, there is a continuous reduction $f:X\to X$ from $V$ to $U$. Consider now the continuous mapping $g:X^\Nat\to X^\Nat$ given as $g(x):=(f(x_n))_{n\in\Nat}$. Then $g^{-1}(B_{U,k}) = B_{V,k}$, so $g$ is a continuous reduction from $B_{V,k}$ to $B_{U,k}$ and therefore $B_{U,k}$ is $D_2(\boldsymbol{\Pi}_{\alpha+1}^0)$-hard.
\end{proof}

\section{Upper bounds}\label{sec:upperBounds}

In this section we collect all our results concerning the upper estimates from our main theorems mentioned above. Proofs of our upper bounds are based on the following two results from Part~\ref{part1}.

\begin{thm}\label{thm:part1Main1}
    Let $X$ be a separable Banach space, $\alpha<\omega_1$ countable nonzero ordinal and $k\in\Nat$. Then $X$ is isometric to $\C(\omega^\alpha\cdot k)$ if and only if the following conditions hold.
\begin{enumerate}[label=(K\alph*)]
    \item\label{it:CkCountable} $X$ is isometric to some $\C(K)$ space with $K$ zero-dimensional,
    \item\label{it:derivative} $s_2^\alpha(B_{X^*})$ is a nonempty subset of $k$-dimensional space, and if $k>1$ then it is not a subset of $(k-1)$-dimensional space.
\end{enumerate}
\end{thm}

\begin{thm}\label{thm:part1Main2}
    Let $X$ be a separable Banach space. Then $X$ is isometric to some $\C(K)$ space with $K$ zero-dimensional if and only if $X$ is a $\mathcal{L}_{\infty,1+}$-space satisfying the following condition
    \begin{equation}\label{eq:conditionCKcountable}
    \forall\varepsilon>0:\; \overline{\aco}\Big\{a\in S_X\setsep \forall y\in S_X:\;\max\{\|a-y\|,\|a+y\|\}\geq 2-\varepsilon\Big\} = B_X.
    \end{equation}
    Moreover, $X$ is isometric to $\C(2^\Nat)$ if and only if $X$ is a $\mathcal{L}_{\infty,1+}$-space with the Daugavet property satisfying condition \eqref{eq:conditionCKcountable}.
\end{thm}
Theorem~\ref{thm:part1Main1} follows from Theorem~\ref{thm:charactCkcountableInCk} and Theorem~\ref{thm:part1Main2} follows from Theorem~\ref{thm:mainPart1Strong} using that any $\C(K)$ space is a $\mathcal{L}_{\infty,1+}$-space and using the fact mentioned in the Introduction that $X$ is $L_1$-predual if and only if it is a $\mathcal{L}_{\infty,1+}$-space.

We begin by observing that condition~\eqref{eq:conditionCKcountable} is a $\boldsymbol{\Pi}^0_3$ condition.
\begin{lemma}\label{lem:FsigmaDeltaCondition}
The set $\{\mu\in\PP\setsep X_\mu\text{ satisfies condition }\eqref{eq:conditionCKcountable}\}$ is $\boldsymbol{\Pi}^0_3$ in $\PP$.
\end{lemma}
\begin{proof}Let $X$ be a Banach space and let $D \subseteq S_X$ be a dense subset of $S_X$. It is rather easy to check that then $X$ satisfies condition \eqref{eq:conditionCKcountable} if and only if the following holds
\[\begin{split}
    & \forall x_0\in D \;\forall \varepsilon,\delta\in\Rat\cap(0,\infty)\; \exists A\in [D]^{<\omega}\; \exists q\in \Rat^A:\; \sum_{a\in A}|q_a|\leq 1 \; \&\\
    &\quad \Big(\Big\Vert x_0 - \sum_{a\in A} q_{a} a \Big\Vert \leq \delta \quad \& \quad \forall y\in D\;\forall a\in A:\; \max \{ \Vert a - y \Vert, \Vert a + y \Vert \} \geq 2 - \varepsilon\Big).
    \end{split}\]
Now, for any $\mu\in\PP$ we apply the above to the set $D_\mu:=\{\tfrac{v}{\mu(v)}\setsep v\in V,\; \mu(v)\neq 0\}$ which is dense in the sphere of $X_\mu$. This enables us to easily check that the set of $\mu\in\PP$ which satisfy \eqref{eq:conditionCKcountable} is indeed a $\boldsymbol{\Pi}^0_3$-set.
\end{proof}

Using certain previously known results, we obtain the following, which implies the upper estimates from Theorem~\ref{thm:main3Part2} and from Theorem~\ref{thm:main1Part2}~\ref{it:cantorComplete}.

\begin{thm}\label{thm:CkcountableAndCantor}Let $\N\in \{\PP_\infty,\B\}$. Let us denote by $\C$ the class of Banach spaces isometric to some $\C(K)$ space with $K$ zero-dimensional.
\begin{itemize}
    \item The set $\{\mu\in \N\setsep X_\mu\in \C\}$ is $\boldsymbol{\Pi}^0_3$ in the space $\N$.
    \item The isometry class of $\C(2^\Nat)$ is $\boldsymbol{\Pi}^0_3$ in the space $\N$.
\end{itemize}
\end{thm}
\begin{proof}Since $\B\subseteq \PP_\infty$, it suffices to give the proof for $\N=\PP_\infty$. By \cite[Theorem 3.6]{CDDK2}, we have that $L:=\{\mu\in\PP_\infty\setsep X_{\mu} \text{ is a }\mathcal{L}_{\infty,1+}\text{-space}\}$ is $\boldsymbol{\Pi}^0_2$. Hence, by Theorem~\ref{thm:part1Main2} and Lemma~\ref{lem:FsigmaDeltaCondition} we obtain that the set
\[
\{\mu\in \N\setsep X_\mu\in \C\} = L\cap \{\mu\in\PP_\infty\setsep X_\mu\text{ satisfies condition }\eqref{eq:conditionCKcountable}\}
\]
is $\boldsymbol{\Pi}^0_3$. By \cite[Theorem 4.3]{complexityDaugavet}, the set $D:=\{\mu\in\PP_\infty\setsep X_{\mu} \text{ has the Daugavet property}\}$ is $\boldsymbol{\Pi}^0_2$. Hence, by Theorem~\ref{thm:part1Main2} and Lemma~\ref{lem:FsigmaDeltaCondition} we obtain that
\[
\isomtrclass{\C(2^\Nat)} = L\cap D\cap \{\mu\in\PP_\infty\setsep X_\mu\text{ satisfies condition }\eqref{eq:conditionCKcountable}\}
\]
is $\boldsymbol{\Pi}^0_3$ as well.
\end{proof}

To estimate the isometry class of $\mathcal{C}(K)$ spaces with $K$ countable, it remains to analyze the complexity of condition~\ref{it:derivative}. This will be carried out in a series of lemmas.

\begin{lemma}
\label{lem:useckyUzavrene}
Let $X$ be a Banach space and let $\mathcal K:=\mathcal K(B_{X^*},w^*)$ be the hyperspace of all weak* compact subsets of $B_{X^*}$ (equipped with the Vietoris topology).
Let $n\in\mathbb N$ be fixed.
Then the set
\[
\mathcal H_n:=\big\{K\in\mathcal K:K\text{ is contained in an $n$-dimensional subspace of }X^*\big\}
\]
is closed in $\mathcal K$.
\end{lemma}
\begin{proof}
We must show that the complement of the set $\mathcal H_n$ in $\mathcal K$ is open.
So fix $L\in\mathcal K\setminus\mathcal H_n$.
Then we can find linearly independent $x_1^*,\ldots,x_{n+1}^*\in L$.
Further, we can find $\eta>0$ and $x_1,\ldots,x_{n+1}\in X$ such that $x_i^*(x_i)>\eta$ for $i=1,\ldots,n+1$, and such that $x_i^*(x_j)=0$ for $i\neq j$ with $i,j\in\{1,\ldots,n+1\}$.
For every $i\in\{1,\ldots,n+1\}$, let $U_i$ be the weak* open subset of $B_{X^*}$ given by
\[
U_i:=\Big\{z^*\in B_{X^*}\setsep-\frac\eta{n+1}<z^*(x_j)<\frac\eta{n+1}<\eta<z^*(x_i),j\in\{1,\ldots,n+1\}\setminus\{i\}\Big\}.
\]
We define an open subset $\mathcal U$ of $\mathcal K$ by
\[
\mathcal U:=\big\{K\in\mathcal K\setsep K\cap U_i\neq\emptyset\text{ for every } i=1,\ldots,n+1\big\}.
\]
Then $\mathcal U$ is an open neighborhood of $L$ as $x_i^*\in L\cap U_i$, $i=1,\ldots,n+1$.
So, to complete the proof, it suffices to verify that $\mathcal U\cap\mathcal H_n=\emptyset$.
And to this end, it is enough to show that $z_1^*,\ldots,z_{n+1}^*$ are linearly independent whenever $z_i^*\in U_i$, $i=1,\ldots,n+1$.
So pick some $z_i^*\in U_i$, $i=1,\ldots,n+1$, and suppose that $a_1,\ldots,a_{n+1}\in\mathbb R$ are such that
\begin{equation}
\label{eq:linKomb}
\sum_{i=1}^{n+1}a_iz_i^*=0.
\end{equation}
Using the fact that $z_i^*\in U_i$, $i=1,\ldots,n+1$, by plugging $x_i$ into~\eqref{eq:linKomb}, we easily obtain that
\begin{equation}
\label{eq:aritmPrumer}
|a_i|\le\frac 1{n+1}\sum_{j\neq i}|a_j|\le\frac 1{n+1}\sum_{j=1}^{n+1}|a_j|,\quad i=1,\ldots,n+1.
\end{equation}
This means that, for every $i=1,\ldots,n+1$, the value of $|a_i|$ is less than or equal to the arithmetic average of all the values $|a_1|,\ldots,|a_{n+1}|$,
and that can happen only if $|a_1|=\ldots=|a_{n+1}|$.
But then~\eqref{eq:aritmPrumer} implies that $a_1=\ldots=a_{n+1}=0$,
and the linear independency of $z_1^*,\ldots,z_{n+1}^*$ follows.
\end{proof}

\begin{lemma}\label{lem:szlenkComplexity}
Let $\beta$ be either $0$ or a countable limit ordinal and let $n\in\Nat\cup\{0\}$ be such that $\beta+n>0$. Then for every $\varepsilon>0$ the mapping 
\[
\K(\ell_\infty,w^*)\ni F\mapsto s_\varepsilon^{\beta + n}(F)\in \K(\ell_\infty,w^*)
\]
is $\boldsymbol{\Sigma}_{\beta+2n+1}^0$-measurable.
\end{lemma}
\begin{proof}If $\beta=0$ or $n=0$, then this was exactly shown in the proof of \cite[Theorem 5.7]{CDDK2}. For $\alpha = \beta + n$, where $\beta$ is a limit ordinal and $n\in\Nat$ we observe that by the above $s_\varepsilon^\alpha = s_\varepsilon^n\circ s_\varepsilon^\beta$ is a composition of a $\boldsymbol{\Sigma}_{\beta+1}^0$-measurable mapping with a $\boldsymbol{\Sigma}_{2n+1}^0$-measurable mapping, which is a $\boldsymbol{\Sigma}_{\beta+2n+1}^0$-measurable mapping.
\end{proof}

We will later show that the estimate in Lemma~\ref{lem:szlenkComplexity} is optimal for $\varepsilon = 2$, as the mapping $s_2^{\beta+n}$ fails to be $\boldsymbol{\Pi}_{\beta+2n+1}^0$-measurable (see Corollary~\ref{cor:szlenkComplexity}).

We are now ready to derive the upper bounds from Theorem~\ref{thm:main1Part2}~\ref{it:kIsOneComplete}, \ref{it:kIsNotOneComplete}.

\begin{thm}\label{thm:upperBoundFinal}
Let $\beta$ be either $0$ or a countable limit ordinal and let $n\in\Nat\cup\{0\}$ be such that $\beta+n>0$. Pick $\N\in\{\PP_\infty,\B\}$.
\begin{enumerate}[label=(\roman*)]
    \item The isometry class of $\C(\omega^{\beta+n})$ in $\N$ is $\boldsymbol{\Pi}^0_{\beta+2n+1}$.
    \item  The isometry class of $\C(\omega^{\beta+n}\cdot k)$ in $\N$ is $D_2(\boldsymbol{\Pi}^0_{\beta+2n+1})$ for every $k\in\Nat$ with $k\geq 2$.
\end{enumerate}
\end{thm}
\begin{proof}We shall use Theorem~\ref{thm:part1Main1}. Recall that \ref{it:CkCountable} is a $\boldsymbol{\Pi}^0_3$-condition by Theorem~\ref{thm:CkcountableAndCantor}. In order to shorten the notation, put $\alpha:=\beta+n$.

By \cite[Lemma 4.7 and Lemma 4.9]{CDDK2}, there is a continuous mapping $\Omega:\P\to \K(B_{\ell_\infty},w^*)$ such that for every $\mu\in\P$ there exists a bijection from $\Omega(\mu)\subseteq \ell_\infty$ onto $B_{X_\mu^*}$ which is simultaneously a $\|\cdot\|$-$\|\cdot\|$ isometry, a $w^*$-$w^*$ homeomorphism and an affine isomorphism. (In particular, for every $\mu\in\P$ and every $m\in\Nat$, subsets of $B_{X_\mu^*}$ contained in an $m$-dimensional space correspond to subsets of $\Omega(\mu)$ contained in an $m$-dimensional space).
Further, by Lemma~\ref{lem:szlenkComplexity}, the mapping $s_2^{\alpha}$ is $\boldsymbol{\Sigma}_{\beta + 2n + 1}^0$-measurable.

Since $\{\emptyset\}$ is clopen in $\K(B_{\ell_\infty},w^*)$, the set
\[
\{\mu\in\P\setsep s^\alpha_2(\Omega(\mu))\neq \emptyset\}
\]
is both $\boldsymbol{\Sigma}_{\beta + 2n + 1}^0$ and $\boldsymbol{\Pi}_{\beta + 2n + 1}^0$. Thus, the set
\[
P:=\{\mu\in\P_\infty\setsep X_\mu \text{ satisfies \ref{it:CkCountable} and $s^{\alpha}_2(B_{X_\mu^*})\neq \emptyset$}\}
\]
is the intersection of a $\boldsymbol{\Pi}_3^0$-set with a set which is $\boldsymbol{\Pi}_{\beta + 2n + 1}^0$, which is a $\boldsymbol{\Pi}_{\beta + 2n + 1}^0$-set.

Therefore, by Lemma~\ref{lem:useckyUzavrene} we obtain that
\[
Q = P \cap \{\mu\in\P\setsep s^\alpha_2(\Omega(\mu)) \text{ is subset of a $k$-dimensional space}\}
\]
is a $\boldsymbol{\Pi}_{\beta + 2n + 1}^0$-set. This finishes the proof for $k=1$, because for $k=1$ we have $Q=\isomtrclass{\C(\omega^\alpha\cdot k)}$.

For $k>1$ we subtract from $Q$ the set
\[
\{\mu\in\P\setsep s^\alpha_2(\Omega(\mu)) \text{ is subset of a $(k-1)$-dimensional space}\},
\]
which is a $\boldsymbol{\Pi}_{\beta + 2n + 1}^0$-set. Thus, in this case we obtain that $\isomtrclass{\C(\omega^\alpha\cdot k)}$ is a $D_2(\boldsymbol{\Pi}_{\beta + 2n + 1}^0)$-set.
\end{proof}

We do not state the corresponding upper bounds from Theorem~\ref{thm:main2Part2} explicitly at this point, as they follow directly from Corollary~\ref{cor:Ck}.

\section{Lower bound for the isometry class of \texorpdfstring{$\C(2^\Nat)$}{C(Cantor)}}\label{sec:cantorLower}

By Theorem~\ref{thm:CkcountableAndCantor}, the isometry class of $\C(2^\Nat)$ is $\boldsymbol{\Pi}^0_3$. The aim of this section is to show that it is, in fact, $\boldsymbol{\Pi}^0_3$-complete. Since $\B \subseteq \P_\infty$, it suffices to prove that $\isomtrclass[\B]{\C(2^\Nat)}$ is $\boldsymbol{\Pi}^0_3$-hard. This will follow almost immediately from the following central result of this section.

\begin{prop} \label{prop:Cantorvscircle}
There exists a continuous mapping $ \eta : 2^{\mathbb{N} \times \mathbb{N}} \to \mathcal{B} $ such that
\begin{enumerate}[label=(\roman*)]
    \item if each row $ \sigma(i, \cdot) $ of $ \sigma \in 2^{\mathbb{N} \times \mathbb{N}} $ contains $ 1 $ only finitely many times, then $ X_{\eta(\sigma)} $ is isometric to $ \C(2^{\mathbb{N}}) $,
    \item if $ \sigma \in 2^{\mathbb{N} \times \mathbb{N}} $ contains a row $ \sigma(i, \cdot) $ in which $ 1 $ appears infinitely many times, then $ X_{\eta(\sigma)} $ is isometric to $ \C(K) $ for some compact space $ K $ that contains a copy of the circle. 
\end{enumerate}
\end{prop}
\begin{proof}\setcounter{claim}{0}
In what follows, we will consider the interval $ [0, 1) $ as a compact space homeomorphic to a circle (for instance via the mapping $ t \mapsto e^{2\pi it} $). Let $ D $ denote the set of all dyadic numbers in $ [0, 1) $. We consider the decomposition
$$ D = \bigcup_{n=0}^{\infty} D_{n}, $$
where $ D_{0} = \{ 0 \}, D_{1} = \{ \frac{1}{2} \}, D_{2} = \{ \frac{1}{4}, \frac{3}{4} \}$, and, in general, $ D_{n} = \{ \frac{1}{2^{n}}, \frac{3}{2^{n}}, \dots, \frac{2^{n}-1}{2^{n}} \} $ for $n\in\Nat$.

Given $ \varrho \in 2^{\mathbb{N}} $, we define points $ u_{s}^{\varrho} \in \ell_{\infty}(D), s \in D, $ as follows. For $ s \in D_{n} $, we define for $t\in \bigcup_{m=0}^{n} D_{m}$
\[
u_s^{\varrho}(t)=\begin{cases}
1, & \text{ if }t=s,\\
0, & \text{ if }t\neq s,\\
\end{cases}
\]
and recursively for $ m \geq n+1 $ and $t\in D_m$ we put
\[
u_s^{\varrho}(t)=\begin{cases}
u_{s}^{\varrho}(t-\tfrac{1}{2^{m}}), & \text{ if }\rho(m)=0,\\
\frac{1}{2} \big[ u_{s}^{\varrho}(t-\tfrac{1}{2^{m}}) + u_{s}^{\varrho}(t+\tfrac{1}{2^{m}}) \big], & \text{ if }\rho(m)=1,\\
\end{cases}
\]
where we use the convention $ u_{s}^{\varrho}(1) = u_{s}^{\varrho}(0) $.

Let $ Y^{\varrho} \subseteq \ell_{\infty}(D) $ be the closed linear span of the vectors $ u_{s}^{\varrho}, s \in D $. The basic idea behind this construction is that, as we shall show below, $Y^\varrho$ is isometric either to $\C([0,1))$ or to $\C(2^\Nat)$, depending on whether $\varrho$ contains the value $1$ infinitely often or not. We arrive at this conclusion through a sequence of claims, the first of which addresses the case in which $Y^\varrho$ is isometric to $\C(2^\Nat)$.

\begin{claim}\label{claim:CodeCantor}
\begin{enumerate}[label=(\roman*)]    \item\label{it:sameSeqSameSpaces} If $ \varrho, \varrho' \in 2^{\mathbb{N}} $ have the same coordinates with finitely many exceptions, then $ Y^{\varrho} = Y^{\varrho'} $.
    \item\label{it:whenWeObtainCantor} If $ \varrho $ contains $ 1 $ finitely many times, then $ Y^{\varrho} $ is isometric to $ \C(2^{\mathbb{N}}) $. Moreover, the point $ u_{0}^{\varrho} = \mathbf{1}_{D} $ corresponds to the constant function $ 1 $.
\end{enumerate}
\end{claim}
\begin{proof}[Proof of Claim~\ref{claim:CodeCantor}] \ref{it:sameSeqSameSpaces} It is sufficient to assume that $ \varrho \in 2^{\mathbb{N}} $ and $ \varrho' \in 2^{\mathbb{N}} $ have the same coordinates with only one exception, say $ \varrho(m) = 0 $ and $ \varrho'(m) = 1 $ or vice versa. Pick $n\in\Nat$ and $s\in D_n$. Then obviously we have
\begin{equation}\label{eq:smallU}
n\geq m\;\Rightarrow\; u_{s}^{\varrho'} = u_{s}^{\varrho}.  
\end{equation}
Moreover, we shall observe that
\begin{equation}\label{eq:bigU}
    n\in \{0,\ldots,m-1\}\;\Rightarrow\;u_{s}^{\varrho'} = u_{s}^{\varrho} + \sum_{t \in D_{m}} [u_{s}^{\varrho'}(t) - u_{s}^{\varrho}(t)] u_{t}^{\varrho}.
\end{equation}
Indeed, pick $n\in \{0,\ldots,m-1\}$. We first observe that the equality
\[
u_{s}^{\varrho'}(a) = u_{s}^{\varrho}(a) + \sum_{t \in D_{m}} [u_{s}^{\varrho'}(t) - u_{s}^{\varrho}(t)] u_{t}^{\varrho}(a)
\]
holds for every $a\in D_k$ with $k<m$, then we easily check that it also holds for every $a\in D_m$ and finally, we prove by induction on $k$ that it holds for every $a\in D_k$ with $k\geq m$. This shows that \eqref{eq:bigU} holds. Hence, using \eqref{eq:smallU} and \eqref{eq:bigU} we obtain that the vectors $ u_{s}^{\varrho'} $ have the same linear span as the vectors $ u_{s}^{\varrho} $, which finishes the proof of \ref{it:sameSeqSameSpaces}.

\noindent\ref{it:whenWeObtainCantor} Using the already proven part \ref{it:sameSeqSameSpaces}, we can assume that $\varrho(m)=0$ for every $m\in\Nat$. Then, given $n\in\Nat\cup\{0\}$ and $s\in D_n$, we observe $u_s^\varrho$ is equal to the characteristic function $\chi_{[s,s+2^{-n})\cap D}$. Further, for any $s\in D_n$ with $n\in\Nat$ there exists a unique sequence $x_s\in \{0,1\}^n$ with $s = \sum_{i=1}^n \tfrac{x_s(i)}{2^i}$, for such an $x_s$ we denote $C(x_s):=\{x\in 2^\Nat\setsep x|_n=x_s\}\subseteq 2^\Nat$. For $s=0$ we denote $C(x_s) = 2^\Nat$. Then it is easy to check that the mapping $Y^\varrho\ni u_s^\varrho\mapsto \chi_{C(x_s)}$ extends to a linear isometry between $Y^\varrho$ and $\C(2^\Nat)$.
\end{proof}

Our second claim addresses the case in which $Y^\varrho$ is isometric to $\C([0,1))$.

\begin{claim}\label{claim:CodeCircle}
Let $\varrho\in2^\Nat$.
\begin{enumerate}[label=(\roman*)]
    \item\label{it:biorthogonalSeqInSpace} Let $ n \in \mathbb{N} $, $k\in\{0,\ldots,n\}$ and $s\in D_k$. Then there exists $ \widetilde{u}_{s}^{\varrho} \in Y^\varrho$ satisfying the following conditions. For $t\in D_j$ with $j\leq n$ we have $\widetilde{u}_{s}^{\varrho}(t)=1$ if $t=s$ and $\widetilde{u}_{s}^{\varrho}(t)=0$ otherwise. For $t\in D_j$ with $j\geq n+1$ we have 
    \[
        \widetilde{u}_s^{\varrho}(t)=\begin{cases}
\widetilde{u}_{s}^{\varrho}(t-\tfrac{1}{2^{j}}), & \text{ if }\rho(j)=0,\\
\frac{1}{2} \big[ \widetilde{u}_{s}^{\varrho}(t-\tfrac{1}{2^{j}}) + \widetilde{u}_{s}^{\varrho}(t+\tfrac{1}{2^{j}}) \big], & \text{ if }\rho(j)=1.\\
\end{cases}
\]
    \item\label{it:restrictionInurSpace}  The restriction of any $ f \in \C([0, 1)) $ to $ D $ belongs to $ Y^{\varrho} $.
    \item\label{it:ManyOnesImpliesCircle} If $ \varrho $ contains $ 1 $ infinitely many times, then $ Y^{\varrho} $ is isometric to $ \C([0, 1)) $. Moreover, the point $ u_{0}^{\varrho} = \mathbf{1}_{D} $ corresponds to the constant function $ 1 $.
\end{enumerate}
\end{claim}
\begin{proof}[Proof of Claim~\ref{claim:CodeCircle}] 
\ref{it:biorthogonalSeqInSpace} Pick $n\in\Nat$, $k\leq n$ and $s\in D_k$. If $k=n$, it suffices to put $\widetilde{u}_{s}^{\varrho} = u_{s}^{\varrho}$. Assuming $k\leq n-1$ and $\widetilde{u}_{t}^{\varrho}\in Y^\varrho$ was defined for every $t\in D_{l}$ with $k+1\leq l\leq n$, we put
\[
\widetilde{u}_{s}^{\varrho} := u_{s}^{\varrho} - \sum_{l=k+1}^{n} \sum_{t \in D_{l}} u_{s}^{\varrho}(t) \widetilde{u}_{t}^{\varrho}.
\]
Now, it is straightforward to prove that $\widetilde{u}_{s}^{\varrho}$ is as required.

\smallskip

\noindent\ref{it:restrictionInurSpace} For $ n \in \mathbb{N} $, we can consider the linear combination
\[ w_{n} = \sum_{k=0}^{n} \sum_{s \in D_{k}} f(s) \widetilde{u}_{s}^{\varrho},
\]
where $ \widetilde{u}_{s}^{\varrho} $ are as in the already proven part \ref{it:biorthogonalSeqInSpace}. Then $ w_{n} \in Y^{\varrho} $ and it is equal to $ f $ in the dyadic numbers $ \frac{i}{2^{n}}, 0 \leq i < 2^{n} $. Moreover, for $ t \in [\frac{i}{2^{n}},\frac{i+1}{2^{n}}]\cap D $, the value $ w_{n}(t) $ is between the values $ w_{n}(\frac{i}{2^{n}}) $ and $ w_{n}(\frac{i+1}{2^{n}}) $. It follows that $ w_{n} \to f|_{D} $ in $ \ell_{\infty}(D) $.

\smallskip

\noindent\ref{it:ManyOnesImpliesCircle} Due to the already proven part \ref{it:restrictionInurSpace} it remains to show that each $ u \in Y^{\varrho} $ can be extended to a continuous function on $ [0, 1) $. It is sufficient to check this for the points $ u_{s}^{\varrho} $. For $ s \in D_{n} $, the function $ u_{s}^{\varrho} $ is initially defined in the dyadic numbers $ \frac{i}{2^{n}}, 0 \leq i < 2^{n} $, with one value $ 1 $, and the others $ 0 $. In the steps $ m \geq n+1 $ where $ \varrho(m) = 0 $, the value in a new dyadic number $ t \in D_{m} $ is copied from the nearest left one, i.e., $ t - \frac{1}{2^{m}} $. This does not increase the maximal difference of values in neighboring dyadic numbers. In the steps $ m \geq n+1 $ where $ \varrho(m) = 1 $, the value in a new dyadic number $ t \in D_{m} $ is the average of the values in the nearest left one and the nearest right one, i.e., $ t \pm \frac{1}{2^{m}} $. This halves the maximal difference of values in neighboring dyadic numbers. Since there are infinitely many such steps, the function $ u_{s}^{\varrho} $ is uniformly continuous, and thus it can be extended to a continuous function on $ [0, 1) $.
\end{proof}

Our third claim will help us later to show that the mapping $\eta$ constructed below is continuous.

\begin{claim} \label{claim:Cantorvscircle4}
Let $ F \subseteq D $ be finite and $ \alpha \in \Rea^F $. Then the norm $ \Vert \sum_{s \in F} \alpha_{s} u_{s}^{\varrho} \Vert $ depends only on finitely many coordinates of $ \varrho $.
\end{claim}

\begin{proof}[Proof of Claim~\ref{claim:Cantorvscircle4}]
Let $ n \in \mathbb{N} $ be such that $ F \subseteq \bigcup_{m=0}^{n} D_{m} $. Then
$$ w^{\varrho} = \sum_{s \in F} \alpha_{s} u_{s}^{\varrho} $$
is monotone on $ [\frac{i}{2^{n}}, \frac{i+1}{2^{n}}] \cap D $ for $ 0 \leq i < 2^{n} $. Hence
$$ \Vert w^{\varrho} \Vert = \max \{ |w^{\varrho}(\tfrac{i}{2^{n}})| \setsep 0 \leq i < 2^{n} \}. $$
It remains to note that the value $ w^{\varrho}(\tfrac{i}{2^{n}}) $ depends only on the coordinates $ \varrho(m) $ for $ 1 \leq m \leq n $.
\end{proof}

Let $ \sigma \in 2^{\mathbb{N} \times \mathbb{N}} $. For $i\in\Nat$ and $s\in D$ we define $ v_{i, s}^{\sigma} \in \ell_{\infty}(\mathbb{N} \times D) $ by
\[ v_{i, s}^{\sigma}(j, t) :=\begin{cases} u_{s}^{\sigma(i, \cdot)}(t), &  \text{ if }j=i \text{ and }t\in D,\\
0, & \text{ if } j \neq i\text{ and }t\in D.\end{cases}
\]
Let $ Z^{\sigma} \subseteq \ell_{\infty}(\mathbb{N} \times D) $ be the closed linear span of the points $\{ v_{i, s}^{\sigma} \setsep i\in\Nat,\; s\in D\}$ and the point $ \mathbf{1}_{\mathbb{N} \times D} $. Using Claim~\ref{claim:CodeCantor} and Claim~\ref{claim:CodeCircle}, we deduce that $ Z^{\sigma} $ is isometric to $ \C(K) $ for the space $ K = (\bigcup_{i=1}^{\infty} K_{i}) \cup \{ \infty \} $ (i.e., the one-point compactification of the topological sum $ \bigcup_{i=1}^{\infty} K_{i} $), where $ K_{i} $ is a copy of $ 2^{\mathbb{N}} $ if $ \sigma(i, \cdot) $ contains $ 1 $ finitely many times, and it is a copy of $ [0, 1) $ otherwise.

It remains to show that there is a continuous mapping $ \eta : 2^{\mathbb{N} \times \mathbb{N}} \to \mathcal{B} $ such that $ X_{\eta(\sigma)} $ is isometric to $ Z^{\sigma} $ for each $ \sigma \in 2^{\mathbb{N} \times \mathbb{N}} $. We can identify $\B$ with the set of all norms on the space $c_{00}((\Nat\times D)\cup \{*\})$ and, for $\gamma\in (\Nat\times D)\cup\{*\}$ we denote by $e_{\gamma}\in\ell_\infty((\Nat\times D)\cup \{*\})$ the characteristic function of $\{\gamma\}$. Then, given $\sigma\in 2^{\Nat\times\Nat}$ we define $\eta(\sigma)\in \B$ as follows. For finite sets $L\subseteq \Nat$ and $F_i\subseteq D$, $i\in L$, we put $G:=\{(i,s)\in\Nat\times D\setsep i\in L,\; s\in F_i\}$ and for $\alpha\in \Rea^{\{*\}\cup G}$ we put
\[
\eta(\sigma)\Big(\alpha_{*} e_{*} + \sum_{\gamma\in G}\alpha_\gamma e_\gamma\Big):=\Big\|\alpha_{*} \mathbf{1}_{\mathbb{N} \times D} + \sum_{i \in L} \sum_{s \in F_{i}} \alpha_{i, s} v_{i, s}^{\sigma}\Big\|_{Z^\sigma},
\]
which immediately implies that $\eta(\sigma)\in \PP_\infty$ and $X_{\eta(\sigma)}$ is linearly isometric to $Z^\sigma$. It remains to show that $\eta(\sigma)$ is actually a norm and that $\eta$ is continuous. In order to show this, notice that given $L$, $F_i$ for $i\in L$ and $\alpha$ as above, for $z^\sigma:=\alpha_{*} \mathbf{1}_{\mathbb{N} \times D} + \sum_{i \in L} \sum_{s \in F_{i}} \alpha_{i, s} v_{i, s}^{\sigma}$ we have
\[
\Vert z^{\sigma} \Vert = \max \bigg( \{ |\alpha_*| \} \cup \Big\{ \Big\Vert \alpha_* u_{0}^{\sigma(i, \cdot)} + \sum_{s \in F_{i}} \alpha_{i, s} u_{s}^{\sigma(i, \cdot)} \Big\Vert_{\ell_\infty(D)} \setsep i \in L \Big\} \bigg).
\]
Hence, if $z^\sigma=0$ then $\alpha_*=0$ and, since the set $\{ u_{s}^{\sigma(i, \cdot)}\setsep s\in F_i\}$ is linearly independent in $\ell_\infty(D)$ for every $i\in L$, we obtain that $\alpha\equiv 0$. Thus, $\eta(\sigma)$ is indeed a norm. Moreover, using Claim~\ref{claim:Cantorvscircle4} we observe the norm $ \Vert z^{\sigma} \Vert $ depends only on finitely many coordinates of $ \sigma $, hence $\eta$ is a continuous function on $ 2^{\mathbb{N} \times \mathbb{N}} $.
\end{proof}

As indicated above, now we have all the necessary ingredients in order to obtain main outcome of Section~\ref{sec:cantorLower}.

\begin{proof}[Proof of Theorem~\ref{thm:main1Part2}\ref{it:cantorComplete}] As mentioned at the beginning of Section~\ref{sec:cantorLower}, by Theorem~\ref{thm:CkcountableAndCantor}, it suffices to prove that $\isomtrclass[\B]{\C(2^\Nat)}$ is $\boldsymbol{\Pi}^0_3$-hard.

Let us recall that the set $U:=\{x\in 2^\Nat\setsep \forall m_0\in\Nat\; \exists m\geq m_0:\; x_m=1\}$ is $\boldsymbol{\Pi}_2^0$-complete, see \cite[p. 179]{Kechrisbook}. Hence, using Proposition~\ref{prop:difHard} we deduce that
\[
P_3:=\{A\in (2^{\Nat})^\Nat\setsep \forall n\in\Nat\; A_n\notin U\}
\]
is $\boldsymbol{\Pi}_3^0$-complete.

Consider now the continuous mapping $\eta:2^{\Nat\times\Nat}\to \B$ from Proposition~\ref{prop:Cantorvscircle}, which in particular satisfies that \[
\eta^{-1}(\isomtrclass[\B]{\C(2^\Nat)}) = \{\sigma\in 2^{\Nat\times\Nat}\setsep \forall i\in\Nat\; \sigma(i,\cdot)\notin U\},
\]
which is a set that can be easily identified with $P_3$. Hence, we have $P_3\leq_W \isomtrclass[\B]{\C(2^\Nat)}$ and therefore $\isomtrclass[\B]{\C(2^\Nat)}$ is $\boldsymbol{\Pi}^0_3$-hard.
\end{proof}

\begin{proof}[Proof of Theorem~\ref{thm:main3Part2}] 
    Let us denote by $\C$ the class of Banach spaces isometric to some $\C(K)$ space with $K$ zero-dimensional. By Theorem~\ref{thm:CkcountableAndCantor} and the fact that $\B\subseteq \P_\infty$, it suffices to show that the set $S:=\{\mu\in\B\setsep X_\mu\in \C\}$ is $\boldsymbol{\Pi}^0_3$-hard.
    
    Consider the continuous mapping $\eta:2^{\Nat\times\Nat}\to \B$ from Proposition~\ref{prop:Cantorvscircle}. Then, in the same way as in the proof of Theorem~\ref{thm:main1Part2}\ref{it:cantorComplete} above we deduce $\eta^{-1}(\isomtrclass[\B]{\C(2^\Nat)})$ is $\boldsymbol{\Pi}_3^0$-complete set. But the construction in Proposition~\ref{prop:Cantorvscircle} ensures that for every $ \sigma \in 2^{\Nat\times\Nat} $, either $X_{\eta(\sigma)}$ is isometric to $\C(2^\Nat)$, or $X_{\eta(\sigma)}$ is isometric to $\C(K)$ for some compact space $K$ containing a copy of the circle. In the latter case, $ \eta(\sigma) \notin S $ since $K$ is not zero-dimensional. Indeed, if $X_{\eta(\sigma)}$ were isometric to $\C(L)$ for some zero-dimensional compact space $L$, then the Banach--Stone theorem would imply that $K$ and $L$ are homeomorphic, which is impossible. Hence, we have $\eta^{-1}(S) = \eta^{-1}(\isomtrclass[\B]{\C(2^\Nat)})$ and so $\eta$ is a continuous reduction from a $\boldsymbol{\Pi}_3^0$-complete set to $S$, which implies that $S$ is $\boldsymbol{\Pi}_3^0$-hard.
\end{proof}

\begin{remark}
    Let us note that from the proof of Proposition~\ref{prop:Cantorvscircle} we also obtain that $\isomtrclass[\B]{\C([0,1))}$ is $G_\delta$-hard, where we think of $[0,1)$ as a compact set homeomorphic to a circle. However, this estimate is suboptimal because there does not exist an infinite-dimensional separable $\C(K)$ space with $G_\delta$ isometry class, see e.g. \cite[Theorem 3.3]{CDDK2}.
\end{remark}

\begin{remark}\label{rem:cantorHomeo}
We note that analogous result for the homeomorphism class of the Cantor set does not hold as it is $G_\delta$. Indeed, let $X$ be an uncountable metrizable compact space. By \cite[Exercise 4.30 and 4.31]{Kechrisbook}, the set 
\[A:=\{K\in\K(X)\setsep K\text{ is infinite and perfect}\}\]
is $G_\delta$ and by \cite[Proposition 5.1]{DR18}, the set 
\[B:=\{K\in\K(X)\setsep K\text{ is zero-dimensional}\}\]
is $G_\delta$ as well. By Brouwer's theorem (see e.g. \cite[Theorem 7.4]{Kechrisbook}) we therefore obtain that $\homeoclass{2^\Nat}^X = A\cap B$ is $G_\delta$.

Let us note that $\homeoclass{2^\Nat}^X$ is even $G_\delta$-complete. Let us present the argument. Since there exists an embedding $\iota:2^\Nat\to X$, the mapping $\K(2^\Nat)\ni L\mapsto \iota(L)\in \K(X)$ witnesses that $\homeoclass{2^\Nat}^{2^\Nat}\leq _W \homeoclass{2^\Nat}^X$; hence, we may without loss of generality assume $X= 2^\Nat$. Since $X$ is perfect, $A$ is known to be dense in $\K(X)$ (see e.g. \cite[Exercise 8.8]{Kechrisbook}) and since $X$ is zero-dimensional we have $B=\K(X)$, so we obtain that $\homeoclass{2^\Nat}^X = A\cap B = A$ is dense and $G_\delta$. Also, it is well known (and easy to check) that finite sets are dense in $\mathcal{K}(X)$, thus the complement of  $\homeoclass{2^\Nat}^X$ is dense in $\K(X)$ as well.  Hence, $\homeoclass{2^\Nat}^X$ cannot be $F_\sigma$ as otherwise $\homeoclass{2^\Nat}^X$ and its complement would be comeager sets with empty intersection. Thus, $\homeoclass{2^\Nat}^X$ is $G_\delta$ but not $F_\sigma$ and so it is $G_\delta$-complete, see e.g. \cite[Lemma 1.1]{CDDK2}.
\end{remark}

\section{Lower bounds for the homeomorphism classes}\label{sec:lowBoundHom}

Building on the construction by Cenzer and Mauldin from \cite{CM82} and \cite{CM83} we shall obtain lower estimate on the Borel class of $\homeoclass{K}$ for any infinite countable compact space $K$ and present the proof of Theorem~\ref{thm:main2Part2}.

We begin by outlining several consequences of the construction of Cenzer and Mauldin from \cite{CM82,CM83}, which we shall use later. Concerning the initial step, using the proof of \cite[Proposition 4.1]{CM82}, we obtain the following.

\begin{prop}\label{prop:initialOrdinal}
    For any $B\in\boldsymbol{\Pi}^0_2(\Nat^\Nat)$ there exists a continuous reduction $H:\Nat^\Nat\to \K(2^\Nat)$ of $B$ to $\homeoclass{[0,\omega]}$ such that $H(x)$ is either finite or homeomorphic to $[0,\omega]$ for every $x\in \Nat^\Nat$.
\end{prop}
\begin{proof}This follows immediately from the proof of \cite[Proposition 4.1]{CM82}. Indeed, by the statement of this result we obtain a continuous mapping $H:\Nat^\Nat\to \K(2^\Nat)$ such that $H(x)$ is finite if and only if $x\notin B$ and from the construction provided in the proof we easily observe that each $H(x)$ has at most one accumulation point.
\end{proof}

Concerning limit ordinals, slightly modifying the construction from the proof of \cite[Theorem 5]{CM83}, we obtain the following.
\begin{thm}\label{thm:countableOrdinal}For any countable limit ordinal $\lambda$ and any $B\in\boldsymbol{\Pi}^0_\lambda(\Nat^\Nat)$ there exists a continuous reduction $H:\Nat^\Nat\to \K(2^\Nat)$ of $B$ to $\homeoclass{[0,\omega^\lambda]}$ such that $H(x)$ is infinite compact set, homeomorphic to a subset of $[0,\omega^{\lambda}]$ for every $x\in \Nat^\Nat$.
\end{thm}
\begin{proof}
This follows essentially from the proof of \cite[Theorem 5]{CM83}\footnote{Note that our notation differs slightly from that of \cite[Theorem 5]{CM83}: the set $B$ considered here is the complement of the set denoted by $B$ in that paper.}. Indeed, by the statement of this result, we obtain a continuous mapping $H:\Nat^\Nat\to \K(2^\Nat)$ such that for every $x\in \Nat^\Nat$ we have $D^{\lambda+1}(H(x))=\emptyset$, and such that $x\notin B$ is equivalent to $D^{\lambda}(H(x))=\emptyset$. Hence, it remains to observe that by the construction of $H$ provided in the proof of \cite[Theorem 5]{CM83}, each $H(x)$ is infinite compact and $D^{\lambda}(H(x))$ contains at most one point.

In the construction, the set $H(x)$ is defined as $H(x) = G(x)\cap C_\lambda$ and, by the proof of \cite[Lemma 5.2]{CM82}, we have $D^{\alpha}(H(x)) = D^\alpha(G(x))\cap D^{\alpha}(C_\lambda)$ for any countable ordinal $\alpha$.

First, let us prove that $D^{\lambda}(H(x))$ contains at most one point. Indeed, by \cite[Definition 5.9]{CM82} we have $C_\lambda = \theta((C_{\lambda_n})_n)$, where $\lambda_n$ is an increasing sequence with $\sup_n\lambda_n = \lambda$ and therefore, using \cite[Proposition 5.10 and Lemma 5.7]{CM82} we have 
\[
D^{\lambda}(C_\lambda) = \theta\Big((D^\lambda(C_{\lambda_n}))_n\Big) = \theta(\emptyset,\emptyset,\ldots) \stackrel{\text{\cite[Definition 5.3]{CM82}}}{=} \{0\},
\]
so we obtain that $D^{\lambda}(H(x))\subseteq D^\lambda(C_\lambda)$ contains at most one point.

Now, since by \cite[Lemma 5.2]{CM82} we have $\height(H(x)) = \min\{\height(G(x)),\height(C_\lambda)\}$, it remains to prove that each $G(x)$ is infinite, that is, $D(G(x))\neq \emptyset$.
The mapping $G$ is obtained from the proof of \cite[Theorem 6.2]{CM82} (the case when $\beta = \lambda +1$ and $k=0$\footnote{The terminology in \cite{CM82} agrees with ours, but the indexing differs by one: the additive and multiplicative classes of rank $\lambda$ are denoted there by $\mathbf{\Sigma}^0_{\lambda+1}$ and $\mathbf{\Pi}^0_{\lambda+1}$.}). We shall slightly modify the construction from the proof of \cite[Theorem 6.2]{CM82} in such a way that each $G(x)$ is infinite (in the proof of \cite[Theorem 6.2]{CM82}, $G(x)$ is called $H(x)$, so below we shall say $H(x)$ instead of $G(x)$). If $\lambda = \omega$, it suffices to assume without loss of generality on \cite[p. 377, line -12]{CM82} that $A=\bigcup_{n=2}^\infty A_n$, where each $A_n$ is $\boldsymbol{\Pi}^0_n$, and then the construction gives that $\height(H_n(x)) \geq n-1\geq 1$ for every $n$ and then we obtain $\height(H(x))\geq 1$. In the case when $\lambda > \omega$, it suffices to assume on \cite[p. 378, line 16]{CM82} that $\alpha_0\geq \omega$ and then in the proof we obtain $\height(H(x))\geq \alpha_0\geq \omega$. This finishes the proof that in any case $G(x)$ is an infinite compact.
\end{proof}

\begin{remark}\label{rem:cenzerMauldinLoweEstimate}Note that Proposition~\ref{prop:initialOrdinal} implies that $\homeoclass{[0,\omega]}$ is $\boldsymbol{\Pi}^0_2$-hard and Theorem~\ref{thm:countableOrdinal} implies that $\homeoclass{[0,\omega^\lambda]}$ is $\boldsymbol{\Pi}^0_{\lambda}$-hard for any countable limit ordinal $\lambda$. Those estimates will be further improved. Note that using the construction from the proof of \cite[Theorem 7]{CM83}, it is possible to obtain a result similar to Proposition~\ref{prop:initialOrdinal} and Theorem~\ref{thm:countableOrdinal} also for the case of a countable successor ordinal. However, we shall not need this variant and therefore we do not formulate the corresponding result here.
\end{remark}

The main new ingredient is the following.

\begin{prop}\label{prop:countableOrdinal2}
Let $\gamma>0$ and $\alpha>0$ be countable ordinals.
Suppose that there are a zero-dimensional Polish space $X$, a $\boldsymbol{\Pi}^0_\gamma$-hard set $A\subseteq X$ and a continuous reduction $H:X\to\K(2^\Nat)$ of $A$ to $\homeoclass{[0,\omega^\alpha]}$ such that $H(x)$ is homeomorphic to a subset of $[0,\omega^\alpha]$ for every $x\in X$.

Then there exist a $\boldsymbol{\Pi}^0_{\gamma+1}$-hard set $A_0\subseteq X^\Nat$, $D_2(\boldsymbol{\Pi}^0_{\gamma+1})$-hard sets $A_k\subseteq X^\Nat$, $k\in\Nat$, a $\boldsymbol{\Pi}^0_{\gamma+2}$-hard set $A_\infty\subseteq X^\Nat$ and a continuous map $H':X^\Nat\to\K(2^\Nat)$ such that
\begin{enumerate}[label=(\alph*)]
	\item\label{it:konecne} $H'$ is a continuous reduction of $A_{k-1}$ to $\homeoclass{[0,\omega^\alpha\cdot k]}$ for every $k\in\Nat$;
	\item\label{it:nekonecne} $H'$ is a continuous reduction of $A_\infty$ to $\homeoclass{[0,\omega^{\alpha+1}]}$;
    \item\label{it:indPredp} $H'(x)\in \bigcup_{k\in\Nat}\homeoclass{[0,\omega^{\alpha}\cdot k]} \cup \homeoclass{[0,\omega^{\alpha+1}]}$ for every $x\in X^\Nat$.
\end{enumerate}
\end{prop}

\begin{proof}
Let $L_1$ be the one-point compactification of $\bigcup_{n=1}^\infty(2^\Nat\times\{n\})$ and let $\infty$ be the single element of the remainder $L_1\setminus\bigcup_{n=1}^\infty(2^\Nat\times\{n\})$.
Let $L_2$ be the topological sum of $L_1$ and $[0,\omega^\alpha]$.
Finally, let $L$ be the compact space obtained from $L_2$ by gluing the points $\infty$ and $\omega^\alpha$.
(That is, $L$ is the quotient space of $L_2$ given by the equivalence relation $xEy\iff(x=y)\vee\{x,y\}=\{\infty,\omega^\alpha\}$).
In the following, we will identify elements of $L_2$ with the corresponding elements of $L$.
Then, in particular, we will have $\infty=\omega^\alpha$ and the set $[0,\omega^\alpha]$ becomes a compact subset of $L$.
Note that $L$ is second countable compact Hausdorff space, and therefore it is metrizable. Note also that $L$ is zero-dimensional.
Recall that every zero-dimensional separable metrizable space can be embedded into $2^\Nat$ (see e.g.~\cite[Theorem~7.8]{Kechrisbook}).
In particular, there is an embedding $\iota:L\to2^\Nat$ of $L$ into $2^\Nat$.
In the following, we will identify the compact space $L$ with its homeomorphic image $\iota(L)\subseteq 2^\Nat$, so that $\K(L)$ becomes a subspace of $\K(2^\Nat)$.
This allows us to construct the desired map $H'$ in such a way that all its values are compact subsets of $L$ (instead of $2^\Nat$).

We put
\[
A_k=\big\{(x_n)\in X^\Nat\setsep|\{n\in\Nat\setsep x_n\in A\}|=k\big\},\quad k\in\Nat\cup\{0,\infty\}.
\]
Then $A_k$ is $\boldsymbol{\Pi}^0_{\gamma+1}$-hard in $X^\Nat$ if $k=0$ (by Proposition~\ref{prop:difHard}\ref{it:PiComplete}), $D_2(\boldsymbol{\Pi}^0_{\gamma+1})$-hard if $k\in\Nat$ (by Proposition~\ref{prop:difHard}\ref{it:DifComplete}) and $\boldsymbol{\Pi}^0_{\gamma+2}$-hard if $k=\infty$ (by Lemma~\ref{lem:completeSetInductionStep}).
Let $H':X^\Nat\to\K(L)$ be given by
\[
H'((x_n)_{n=1}^\infty)=[0,\omega^\alpha]\cup\bigcup_{n=1}^\infty(H(x_n)\times\{n\}),\quad(x_n)_{n=1}^\infty\in X^\Nat.
\]
The correctness of the definition (that is, the fact that every set from the range of $H'$ is compact) easily follows from the definition.
Further, $H'$ is continuous.
To see that, we must check for a fixed open set $U\subseteq L$ that the preimages of $\{K\in\K(L)\setsep K\cap U\neq\emptyset\}$ and $\{K\in\K(L)\setsep K\subseteq U\}$ are open subsets of $X^\Nat$.
The former easily follows from continuity of $H$.
The latter follows from continuity of $H$ together with the facts that a) if $[0,\omega^\alpha]\not\subseteq U$ then the corresponding preimage is empty and b) if $[0,\omega^\alpha]\subseteq U$ then there is $N\in\Nat$ such that $\bigcup_{n=N}^\infty(2^\Nat\times\{n\})\subseteq U$.

We fix $x=(x_n)_{n=1}^\infty\in X^\Nat$.
Note that the $\alpha$th Cantor-Bendixson derivative of $H'(x)$ is the union of $\{\omega^\alpha\}$ and of disjoint homeomorphic copies of all the $\alpha$th Cantor-Bendixson derivatives of $H(x_n)$, $n\in\Nat$.
By our assumption, for every $n\in\Nat$, the $\alpha$th Cantor-Bendixson derivative of $H(x_n)$ is either empty (if $x_n\notin A$) or a singleton (if $x_n\in A$).
Consequently, the $\alpha$th Cantor-Bendixson derivative of $H'(x)$ consists exactly of $1+|\{n\in\Nat\setsep x_n\in A\}|$ distinct points.
So, for every $k\in\Nat$, we have $x\in A_{k-1}$ if and only if the $\alpha$th Cantor-Bendixson derivative of $H'(x)$ consists exactly of $k$ distinct points if and only if $H'(x)$ is homeomorphic to $[0,\omega^\alpha\cdot k]$.
This proves~\ref{it:konecne}.
If $x\in A_\infty$ then the $\alpha$th Cantor-Bendixson derivative of $H'(x)$ consists of infinitely many points which accumulate at $\infty$ (and nowhere else).
Consequently, $H'(x)$ is homeomorphic to $[0,\omega^{\alpha+1}]$ in this case.
As $X^\Nat=A_\infty\cup\bigcup_{k=0}^\infty A_k$, both~\ref{it:nekonecne} and~\ref{it:indPredp} easily follow.
\end{proof}

We now have all the necessary ingredients to establish the following result, which constitutes the final step in the proof of Theorem~\ref{thm:main2Part2}.

\begin{cor}\label{cor:contReduction}
Let $\beta$ be either $0$
or a countable limit ordinal and let $n\in\Nat\cup\{0\}$.
Then there exists a zero-dimensional Polish space $X$ and a continuous map $H:X\to \K(2^\Nat)$ such that
\begin{enumerate}[label=(\roman*$_n$)]
    \item\label{it:image} $H(x)\in \bigcup_{k\in\Nat}\homeoclass{[0,\omega^{1+\beta+n}\cdot k]} \cup \homeoclass{[0,\omega^{1+\beta+n+1}]}$ for every $x\in X$;
    \item\label{it:small} $H$ is a continuous reduction of some $\boldsymbol{\Pi}^0_{2+\beta+2n+1}$-hard set to $\homeoclass{[0,\omega^{1+\beta+n}]}$;
    \item\label{it:mid} for every $k\geq 2$, $H$ is a continuous reduction of some $D_2(\boldsymbol{\Pi}^0_{2+\beta+2n+1})$-hard set to $\homeoclass{[0,\omega^{1+\beta+n}\cdot k]}$;
    \item\label{it:next} $H$ is a continuous reduction of some $\boldsymbol{\Pi}^0_{2+\beta+2n+2}$-hard set to $\homeoclass{[0,\omega^{1+\beta+n+1}]}$.
\end{enumerate}
\end{cor}
\begin{proof}We shall prove the assertion by induction. For $n=0$ this follows directly from Proposition~\ref{prop:countableOrdinal2} using that its assumption is satisfied by Proposition~\ref{prop:initialOrdinal} (with $\gamma=2$ and $\alpha=1$) if $\beta=0$ and by Theorem~\ref{thm:countableOrdinal} (with $\gamma=\alpha=\beta$) if $\beta$ is a limit ordinal.
Now, assuming that we have found $H$ satisfying \ref{it:image}--\ref{it:next}, using conditions \ref{it:image} and \ref{it:next} we see that assumptions of Proposition~\ref{prop:countableOrdinal2} hold for $\gamma=2+\beta+2(n+1)$ and $\alpha=1+\beta+n+1$, so an application of Proposition~\ref{prop:countableOrdinal2} gives us $H'$ satisfying conditions $(\text{i}_{n+1})$--$(\text{iv}_{n+1})$.
\end{proof}

\begin{proof}[Proof of Theorem~\ref{thm:main2Part2}]
    Let $X$, $\beta$ and $n$ be as in the assumptions. Pick $k\in\Nat$, put $K:=[0,\omega^{\beta+n}\cdot k]$ and denote $\boldsymbol{\Gamma}:=\boldsymbol{\Pi}^0_{\beta+2n+1}$ if $k=1$ and $\boldsymbol{\Gamma}:=D_2(\boldsymbol{\Pi}^0_{\beta+2n+1})$ if $k\geq 2$. By Theorem~\ref{thm:upperBoundFinal} and Corollary~\ref{cor:Ck}, the homeomorphism class of $K$ in $\mathcal{K}(X)$ belongs to the class $\boldsymbol{\Gamma}$, so it suffices to prove it is $\boldsymbol{\Gamma}$-hard. Since there exists an embedding $\iota:2^\Nat\to X$, the mapping $\K(2^\Nat)\ni L\mapsto \iota(L)\in \K(X)$ witnesses that $\homeoclass{K}^{2^\Nat}\leq _W \homeoclass{K}^X$; hence, we may without loss of generality assume $X= 2^\Nat$. Finally, by Corollary~\ref{cor:contReduction}\ref{it:small} and \ref{it:mid}, there exists a continuous reduction from a $\boldsymbol{\Gamma}$-hard set to the homeomorphism class of $K$, which proves that $\homeoclass{K}^{2^\Nat}$ is $\boldsymbol{\Gamma}$-hard.
\end{proof}

Let us note that, as we shall see in Section~\ref{sec:appl}, from the results above we can easily deduce Theorem~\ref{thm:main1Part2}~\ref{it:kIsOneComplete} and \ref{it:kIsNotOneComplete} for the case of $\N=\PP_\infty$ and with a little additional effort also for the case of $\N=\B$ and countable compacta of infinite height (that is, the case when $\beta\neq 0$). The missing case is therefore when $\N=\B$ and $\beta=0$. This will be handled in the next section.

\section{Lower bounds for compacta of finite height in the coding \texorpdfstring{$\B$}{B}}\label{sec:betaKonecne}

The main aim of this section is to develop methods leading to the proof of the lower estimate in Theorem~\ref{thm:main1Part2}~\ref{it:kIsOneComplete} and \ref{it:kIsNotOneComplete} for the case $\N=\B$ and compacta of finite height (that is, for the case when $\beta = 0$). In Subsection~\ref{subsec:boolean} we gather certain preliminary observations concerning Boolean algebras. Those are used in Subsection~\ref{subsec:generalFromBoolean} to describe a general method of assigning to each point of a Polish space a norm $\mu\in \B$ in such a way that $X_\mu$ is linearly isometric to a certain $\C(K)$ space with $K$ zero-dimensional. This is used further in Subsection~\ref{subsec:omegaTimesK}, where we obtain results leading to lower bounds for the complexities of the isometry classes of $\C(\omega\cdot k)$ spaces in $\B$. Finally, in Subsection~\ref{subsec:successor} we provide a step towards the successor step which completes our construction leading to the lower bound in Theorem~\ref{thm:main1Part2}~\ref{it:kIsOneComplete} and \ref{it:kIsNotOneComplete} for the case $\N=\B$ and compacta of finite height.

\subsection{Boolean algebras}\label{subsec:boolean}

Let $T$ be a set and $E\subseteq \pot(T)$ be such that $T\in E$. Then we shall denote as $B(E)\subseteq \pot(T)$ the Boolean algebra generated by $E$. It is well-known and easy to check, see e.g. \cite[Chapter 11]{boolBookHalmos}, that then $B(E)$ consists of finite unions of finite intersections of elements and complements of elements from $E$.

By the Stone duality, see e.g. \cite[Theorem 7.8]{HandbookBool1}, for any Boolean algebra $A$ there exists a unique (up to homeomorphism) zero-dimensional compact space $K_A$ for which there exists a surjective and injective homomorphism (i.e., an isomorphism) of Boolean algebras $s:A\to \Clop(K_A)$, where $\Clop(K_A):=\{C\subseteq K_A\setsep C \text{ is clopen}\}$. Thus, for the moment consider the compact $K_{B(E)}$ associated to the Boolean algebra $B(E)$ as mentioned above.

Since by the Stone-Weierstrass theorem we have $\C(K_{B(E)}) = \closedSpan\{\chi_C\setsep C\in\Clop(K_{B(E)})\}$, the Banach space $\C(K_{B(E)})$ is linearly isometric to
\[
X(E):=\closedSpan\{\chi_B\setsep B\in B(E)\}\subseteq \ell_\infty(T).
\]
We shall need the following easy observation.
\begin{lemma}\label{lem:closedHull}
    Let $T$ be a set and $E\subseteq \pot(T)$ be such that $T\in E$ and $E\cup\{\emptyset\}$ is closed under finite intersections. Then
    \[
    X(E) = \closedSpan\{\chi_B\setsep B\in E\}\subseteq \ell_\infty(T).
    \]
\end{lemma}
\begin{proof}
Without loss of generality, we may assume that $\emptyset\in E$. Let $\mathcal C\subseteq \pot(T)$ be the collection of all sets $C\subseteq T$ for which $\chi_C\in\span\{\chi_B\setsep B\in E\}$;
we trivially have $E\subseteq \mathcal C$. Now, let us observe that $\C$ is a Boolean algebra. Since $T\in\C$, it suffices to show that $\C$ is closed under finite intersections and complements. Let $A,B\in \C$. Then $\chi_A = \sum_{V\in \F_A}\alpha_V\chi_{V}$ and $\chi_B = \sum_{W\in\F_B}\beta_W\chi_{W}$ for some finite families $\F_A,\F_B\subseteq E$ and coefficients $\alpha\in \Rea^{\F_A}$, $\beta\in\Rea^{\F_B}$. Therefore, 
\[
\chi_{A\cap B} = \chi_A \chi_B = \sum_{V\in \F_A}\sum_{W\in\F_B} \alpha_V\beta_W\chi_{V\cap W}.
\]
Since $E$ is closed under finite intersections, we obtain $A\cap B\in \C$. Next, let $A\in\C$. Since $T\in E$, we have $\chi_{T\setminus A} = \chi_T - \chi_A\in \span\{\chi_B\setsep B\in E\}$,
which implies that $T\setminus A\in \C$. Thus, $\C$ is a Boolean algebra containing $E$. Therefore, $B(E)\subseteq \C$ which yields
\[
X(E)
=\closedSpan\{\chi_B\setsep B\in B(E)\}
\subseteq \closedSpan\{\chi_B\setsep B\in\mathcal C\}
\subseteq \closedSpan\{\chi_B\setsep B\in E\}.
\]
This completes the proof, as the opposite inclusion is obvious.
\end{proof}

In what follows, we will be interested in the construction where we associate to a countable sequence $(A_n)$ of Boolean algebras a Boolean algebra $A$ such that $K_A$ is homeomorphic to the one-point compactification of the disjoint union of compact spaces $K_{A_n}$, $n\in\Nat$. This is achieved by the following classical construction, see e.g. \cite[Proposition 8.7 and Proposition 8.10]{HandbookBool1}.
\begin{prop}\label{prop:onePointCompactification}
    Let $(A_i)_{i\in I}$ be a family of nontrivial Boolean algebras.
    
    If $I$ is infinite, then the one-point compactification of the disjoint union of $K_{A_i}$, $i\in I$, is homeomorphic to the Stone space corresponding to the Boolean subalgebra
    \[
    \Pi_{i\in I}^w\; A_i:=\big\{a\in \Pi_{i\in I}A_i\setsep \{i\in I\setsep a_i\neq 0\}\text{ is finite or }\{i\in I\setsep a_i\neq 1
    \}\text{ is finite}\big\}
    \]
    of the Boolean algebra $\Pi_{i\in I} A_i$.

    If $I$ is finite, then the disjoint union of compact spaces $K_{A_i}$, $i\in I$, is homeomorphic to the Stone space corresponding to the product Boolean algebra $\Pi_{i\in I}A_i$.
\end{prop}

In what follows, given a Boolean algebra $A$, we shall be interested in certain classical topological properties of the space $K_A$. Let us recall briefly the necessary notions and results, for a more detailed account we refer the reader e.g. to \cite[Construction 17.6 and 17.7]{HandbookBool1}. We say that $a\in A$ is \emph{atom} if $a>0$ but there is no $x\in A$ with $0<x<a$. Further, by $I(\operatorname{At}A)$ we denote the ideal generated by atoms of $A$. We emphasize the following.

\begin{lemma}\label{lem:cantorBendixson}
    Let $A$ be a Boolean algebra, $K_A$ its Stone space. Then the Cantor-Bendixson derivative $(K_A)'$ is homeomorphic to the Stone space $K_{A_1}$, where $A_1 = A/I(\operatorname{At}A)$.
\end{lemma}
\noindent(Note that given a trivial Boolean algebra $A = \{0\}$, for its Stone space we have $K_A = \emptyset$ and, more generally, for a finite Boolean algebra with $|A| = 2^n$ its Stone space $K_A$ has $n$ points.)

\subsection{General construction}\label{subsec:generalFromBoolean}

We shall use the following way of assigning, to a given point $x$ in a Polish space $X$, a zero-dimensional compact space $K_x$.

\begin{definition}\label{def:R}Let $X$ be a Polish space and $M$, $N$ be infinite sets with $M$ countable.
Let $R\colon X\times M\to\pot(N)$ be an arbitrary mapping. We define the following conditions.
\begin{enumerate}[label=(R\arabic*), series=propertiesR]
\item\label{it:R1} For every $x\in X$, the collection $\{R(x,m)\setsep m\in M\}\cup\{\emptyset\}$ is closed under finite intersections.
\item\label{it:R2} For every $x\in X$, there is $m\in M$ such that $R(x,m)=N$.
\item\label{it:R3} For every $x\in X$, the family of characteristic functions $\{\chi_{R(x,m)}\setsep m\in M\}$ is linearly independent in $\ell_\infty(N)$.
\end{enumerate}
If the mapping $R$ satisfies \ref{it:R1}, then for every $x\in X$ the Banach space $Z_x:=\overline\span(\{\chi_{R(x,m)}\setsep m\in M\}\cup \{\chi_N\})\subseteq\ell_\infty(N)$ is by Lemma~\ref{lem:closedHull} isometric to $\C(K_x)$, where $K_x$ is the zero-dimensional compact space obtained by Stone duality from the Boolean algebra $B_x$ generated by $\{R(x,m)\setsep m\in M\}\cup \{N\}$. We say that $K_x$ and $B_x$ are \emph{the compact space associated to $x$} and \emph{the Boolean algebra associated to $x$}, respectively. Whenever we want to emphasize the role of $R$ we write $K_{R,x}$ and $B_{R,x}$ instead of $K_x$ and $B_x$.
\end{definition}

Conditions \ref{it:R2} and \ref{it:R3} enable us to deduce that there is an assignment mapping each $x\in X$ to $F(x)\in\B$ with $X_{F(x)}$ being isometric to the Banach space $\C(K_x)$. Let us describe the way how this is achieved.

\begin{definition}
    Let $X$ be a Polish space and $M$, $N$ be infinite sets with $M$ countable. Let $R\colon X\times M\to\pot(N)$ be a mapping that meets condition \ref{it:R1}. We define $F:X\to \PP$, the \emph{map induced by $R$}, as follows. Since $M$ is countable, we can identify $\PP$ with the set of all pseudonorms on the space $c_{00}(M)$ and, for $m\in M$, we denote by $e_m\in c_{00}(M)$ the characteristic function of $\{m\}$. Then the mapping $F:X\to \PP$ is defined for $x\in X$ by the formula
\[
F(x)\big(\sum_{m\in G}\lambda_me_m\big):=\Big\|\sum_{m\in G}\lambda_m\chi_{R(x,m)}\Big\|_{Z_x},\qquad G\subseteq M\text{ finite},\; \lambda=(\lambda_m)_{m\in G}\in\Rea^G.
\]
The following condition on the mapping $R$ will be used in order to guarantee the map $F$ is continuous. Given $(x,z)\in X\times N$ and $G\subseteq M$ we denote $S_z^G(x):=\{m\in G\setsep z\in R(x,m)\}$.
\begin{enumerate}[label=(R\arabic*), resume*=propertiesR]
\item\label{it:R4} Given finite $G\subseteq M$ and $x\in X$ there exists an open neighborhood $U\subseteq X$ of the point $x$ such that for any $y\in U$ there exists $H\subseteq M$ finite with $H\supseteq G$ satisfying $\{S_z^H(y)\setsep z\in N\} = \{S_z^H(x)\setsep z\in N\}$.\footnote{In condition \ref{it:R4}, we could equivalently consider only the possibility $ H = G $. However, our version will be somewhat more suitable for our purposes.}
\end{enumerate}
\end{definition}

The role of conditions \ref{it:R2} and \ref{it:R3} is clear from the following.

\begin{lemma}\label{lem:toB}Let $X$ be a Polish space and $M$, $N$ be infinite sets with $M$ countable. Let $R\colon X\times M\to\pot(N)$ be a mapping that meets conditions \ref{it:R1} and \ref{it:R2}. Then for the map $F$ induced by $R$ we have that $X_{F(x)}$ is linearly isometric to $\C(K_{x})$ for every $x\in X$.

Moreover, if $R$ satisfies condition \ref{it:R3}, then $F$ has values in $\B\subseteq\P$.
\end{lemma}
\begin{proof}Let $F:X\to \PP$ be the map induced by $R$ and pick $x\in X$. Since \ref{it:R2} holds for $R$, it follows $\C(K_x)$ is linearly isometric to $\overline{\span}\{\chi_{R(x,m)}\setsep m\in M\}\subseteq \ell_\infty(N)$, therefore the mapping $X_{F(x)}\ni e_m\mapsto \chi_{R(x,m)}$ naturally extends to an isometry between $X_{F(x)}$ and $\C(K_x)$. If \ref{it:R3} holds for $R$, the family $\{\chi_{R(x,m)}\setsep m\in M\}$ is linearly independent, so $F(x)$ is a norm, which implies that $F$ has values in $\B\subseteq \PP$.
\end{proof}

Continuity of the mapping $F$ is guaranteed by condition \ref{it:R4}.

\begin{lemma}\label{lem:Fcont}Let $X$ be a Polish space and $M$, $N$ be infinite sets with $M$ countable. Let $R\colon X\times M\to\pot(N)$ be a mapping that meets conditions \ref{it:R1} and \ref{it:R4}. Then the map $F$ induced by $R$ is continuous.
\end{lemma}
\begin{proof}
We pick a finite set $G\subseteq M$ and $\lambda\in\Rea^G$,
we must show that the map $X\ni x\mapsto F(x)\big(\sum_{m\in G}\lambda_me_m\big)$ is continuous. Given $x\in X$, let $U\subseteq X$ be as in condition \ref{it:R4}. Then for any $y\in U$ we pick a finite set $H\subseteq M$ with $G\subseteq H$ as in \ref{it:R4} and denoting $\lambda_m=0$ for $m\in H\setminus G$ we obtain
\begin{equation*}
\begin{split}
F(x)\big(\sum_{m\in G}\lambda_me_m\big) & =\Big\|\sum_{m\in H}\lambda_m\chi_{R(x,m)}\Big\|_{\ell_\infty(N)}
=\sup\Big\{\Big|\sum_{\substack{m\in S_z^H(x)}}\lambda_m\Big|\setsep z\in N\Big\}\\
& =\sup\Big\{\Big|\sum_{\substack{m\in S_z^H(y)}}\lambda_m\Big|\setsep z\in N\Big\} = F(y)\big(\sum_{m\in G}\lambda_me_m\big).
\end{split}
\end{equation*}
\end{proof}

\subsection{Lower bound for spaces \texorpdfstring{$\C(\omega\cdot k)$}{C(omega.k)}}\label{subsec:omegaTimesK}

The initial step of our argument is to define mapping $R_1$, which results in the proof of the lower bound for spaces $\C(\omega\cdot k)$.

\begin{definition}\label{def:R1}
    Let $X_1:=2^{\Nat\times \Nat}$ and $M_1:=\{\emptyset\}\cup \Nat\cup \Nat^2$. Given $A = (a_{i,j})_{i,j\in\Nat}\in X_1$ and $n\in\Nat$ we put $A_n:=\{m\in\Nat\setsep a_{n,m}=1\}$. Let $R_1:X_1\times M_1\to \pot(M_1)$ be defined for $(A,t)\in X_1\times M_1$ as follows
\[
R_1(A,t):=\begin{cases} M_1 & \text{ if $t=\emptyset$},\\
\{t\} & \text{ if $t\in\Nat^2$},\\
\{t\}\cup (A_t\times \{t\}) & \text{ if $t\in\Nat$}.
\end{cases}
\]
\end{definition}

\begin{lemma}\label{lem:R1isok}
    The mapping $R_1:X_1\times M_1\to \pot(M_1)$ from Definition~\ref{def:R1} satisfies \ref{it:R1}, \ref{it:R2}, \ref{it:R3} and \ref{it:R4}. Consequently, there exists a continuous mapping $F:X_1\to \B$ such that $X_{F(x)}$ is linearly isometric to $\C(K_x)$ for every $x\in X_1$.
\end{lemma}
\begin{proof}The proof that $R_1$ meets conditions \ref{it:R1}, \ref{it:R2} and \ref{it:R3} is straightforward and therefore left to the reader. In order to prove \ref{it:R4}, pick $G\subseteq M_1$ finite and $A\in X_1$. Let $U$ be the neighborhood of $A$ defined as $U:=\bigcap_{(i,j)\in G}\{B\in X_1\setsep B(j,i)=A(j,i)\}$. Put $H:=G\cup\{\emptyset\}$. Then given any $B\in U$ we have $S_\emptyset^H(B) = \{\emptyset\} = S_\emptyset^H(A)$ and for $j\in\Nat$ we obtain
\[
S_j^H(B) = S_j^H(A) = \begin{cases}\{\emptyset,j\} & \text{ if }j\in H,\\
\{\emptyset\} & \text{ if }j\notin H.
\end{cases}
\]
Further, for $(i,j)\in \Nat^2$ we have
\[
S_{(i,j)}^H(B) = \begin{cases}\{\emptyset,(i,j),j\} & \text{if $(i,j)\in H$ and $j\in H$ and $i\in B_j$},\\
\{\emptyset,(i,j)\} & \text{if $(i,j)\in H$ and ($j\notin H$ or $i\notin B_j$)},\\
\{\emptyset,j\} & \text{if $(i,j)\notin H$ and $j\in H$ and $i\in B_j$},\\
\{\emptyset\} & \text{otherwise}.
\end{cases}
\]
Hence, $\{S_{(i,j)}^H(B)\setsep (i,j)\in \Nat^2\setminus H\} \subseteq \{S_\emptyset^H(B), S_j^H(B)\setsep j\in \Nat\} = \{S_\emptyset^H(A), S_j^H(A)\setsep j\in \Nat\}$. Finally, since for $(i,j)\in H$ we have 
\[i\in B_j\Leftrightarrow B(j,i)=1\Leftrightarrow A(j,i)=1\Leftrightarrow i\in A_j,
\]
we obtain $\{S_{(i,j)}^H(B)\setsep (i,j)\in H\} = \{S_{(i,j)}^H(A)\setsep (i,j)\in H\}$ and therefore $\{S_z^H(B)\setsep z\in M_1\} = \{S_z^H(A)\setsep z\in M_1\}$.
\end{proof}

Now, we shall identify the homeomorphic types of the compact sets $K_x$, $x\in X_1$ using Lemma~\ref{lem:cantorBendixson}.

\begin{prop}\label{prop:characterizationOfKALevel2}Let $R_1:X_1\times M_1\to \pot(M_1)$ be as in Definition~\ref{def:R1}, let $A\in 2^{\Nat\times\Nat}$ and $k\in\Nat$.
    \begin{enumerate}
        \item\label{it:c} $K_A$ is homeomorphic to $[0,\omega\cdot k]$ if and only if $|F|=k-1$ where $F=\{n\in\Nat: |A_n|=\omega\}$.
        \item $K_A$ is homeomorphic to $[0,\omega^2]$ if and only if $A_n$ is infinite for infinitely many $n\in\Nat$.
    \end{enumerate}
\end{prop}

\begin{proof}First, we shall find the atoms of $B_A$, let us denote those as $\At(A)$. Of course, we have $\{(n,m)\}\in \At(A)$ for every $(n,m)\in \Nat\times\Nat$. Further, notice that all the elements of $B_A$ are, modulo finite subsets of $\Nat\times\Nat$, sets of the form
\[
\bigcup_{i\in G}\{i\}\cup (A_i\times \{i\})\quad \text{and}\quad M_1\setminus \bigcup_{i\in G}\Big(\{i\}\cup (A_i\times \{i\})\Big),
\]
where $G\subseteq \Nat$ is a finite set. Therefore,
\[
\At(A) = \Big\{\{(n,m)\}\setsep (n,m)\in\Nat\times\Nat\Big\} \cup \Big\{\{n\}\setsep n\in\Nat\text{ is such that $A_n$ is finite}\Big\}.
\]
Now, we shall deal with individual cases.

If for every $n\in\Nat$ the set $A_n$ is finite, modulo finite subsets of $M_1$, the Boolean algebra $B_A$ consists just of $\emptyset$ and $M_1$, therefore the Boolean algebra $B_A/I(\At(A))$ is trivial. Thus, by Lemma~\ref{lem:cantorBendixson}, $(K_A)'$ consists of just one point and $K_A$ is homeomorphic to $[0,\omega]$.

If $F=\{n\in\Nat\setsep |A_n|=\omega\}\neq\emptyset$ is finite set,
then $B_A/I(\At(A))$ consists of
(equivalence classes of)
unions of finitely many sets of the form $\{i\}\cup (A_i\times \{i\})$, $i\in F$ and their complements, so $B_A/I(\At(A))$ has exactly $2\cdot 2^{|F|} = 2^{|F|+1}$ elements and so $(K_A)'$ has exactly $|F|+1$ points, so $K_A$ is homeomorphic to $[0,\omega\cdot (|F|+1)]$.

Finally, assume there are infinitely many $n\in\Nat$ such that the set $A_n$ is infinite. Then $B_A/I(\At(A))$ is isomorphic to a Boolean algebra generated by $\{\{n\}\setsep n\in\Nat\}\subseteq \pot(\Nat)$, so $(K_A)'$ is homeomorphic to $[0,\omega]$ and therefore $K_A$ is homeomorphic to $[0,\omega^2]$.
\end{proof}

\begin{prop}\label{prop:R1step}Let $R_1:X_1\times M_1\to \pot(M_1)$ be as in Definition~\ref{def:R1}. Then the following holds.
\begin{enumerate}[label=(\roman*)]
    \item\label{it:K1sets} For every $x\in X_1$ the compact $K_x$ is homeomorphic either to $[0,\omega\cdot k]$ for some $k\in\Nat$ or to $[0,\omega^{2}]$.
    \item\label{it:K1sets+2} The set $\{x\in X_1\setsep K_x\text{ is homeomorphic to }[0,\omega^{2}]\}$ is $\boldsymbol{\Pi}_{4}^0$-hard.
    \item\label{it:K1sets+1} The set $\{x\in X_1\setsep K_x\text{ is homeomorphic to }[0,\omega]\}$ is $\boldsymbol{\Pi}_{3}^0$-hard.
    \item\label{it:K_1sets+D1} For every $k\geq 2$, the set $\{x\in X_1\setsep K_x\text{ is homeomorphic to }[0,\omega\cdot k]\}$ is $D_2(\boldsymbol{\Pi}_{3}^0)$-hard.
\end{enumerate}
\end{prop}
\begin{proof}Let us recall that the set $U:=\{x\in 2^\Nat\setsep \forall m_0\in\Nat\; \exists m\geq m_0:\; x_m=1\}$ is $\boldsymbol{\Pi}_2^0$-complete, see \cite[p. 179]{Kechrisbook}. Hence, using Proposition~\ref{prop:difHard} we deduce that the sets
\[
P_3:=\{A\in (2^{\Nat})^\Nat\setsep \forall n\in\Nat\; A_n\notin U\},\quad  P_3^k:=\{A\in (2^\Nat)^\Nat\setsep |\{n\in\Nat\setsep A_n\in U\}|=k\}
\]
are $\boldsymbol{\Pi}_3^0$-complete and $D_2(\boldsymbol{\Pi}_3^0)$-complete, respectively (for every $k\geq 1$). Moreover, by Lemma~\ref{lem:completeSetInductionStep} the set
\[
Q_4:=\{A\in (2^{\Nat})^\Nat\setsep A_n\in U\text{ for infinitely many $n$'s}\}
\]
is $\boldsymbol{\Pi}_4^0$-complete. We shall identify each $A\in (2^\Nat)^\Nat$ with $A\in 2^{\Nat\times \Nat} = X_1$ and we note that for $A\in 2^{\Nat\times\Nat}$ and $n\in\Nat$ we have $A_n\in U$ iff the set $\{m\in\Nat\setsep a_{n,m}=1\}$ is infinite.

Now, using Proposition~\ref{prop:characterizationOfKALevel2} we obtain that \ref{it:K1sets} holds and moreover,
\begin{itemize}
    \item $A\in P_3$ if and only if $K_A$ is homeomorphic to $[0,\omega]$,
    \item for every $k\geq 1$, $A\in P_3^k$ if and only if $K_A$ is homeomorphic to $[0,\omega\cdot (k+1)]$,
    \item $A\in Q_4$ if and only if $K_A$ is homeomorphic to $[0,\omega^2]$.
\end{itemize}
Hence, since the set $P_3$ is $\boldsymbol{\Pi}^0_3$-hard, we obtain that \ref{it:K1sets+1} holds and since $P_3^k$ are $D_2(\boldsymbol{\Pi}_3^0)$-hard sets for every $k\in\Nat$, \ref{it:K_1sets+D1} holds. Finally, since $Q_4$ is $\boldsymbol{\Pi}^0_4$-hard, we obtain that \ref{it:K1sets+2} holds.
\end{proof}

To summarize the current state of the argument, we summarize the results obtained so far. The remaining successor step will be addressed in the following subsection.

\begin{thm}\label{thm:mainBInitialStep} Given $k\in\Nat$, the isometry class $\isomtrclass{\C(\omega\cdot k)}$ in $\B$ is $\boldsymbol{\Pi}^0_3$-hard if $k=1$ and $D_2(\boldsymbol{\Pi}^0_3)$-hard if $k\geq 2$.
\end{thm}
\begin{proof}
Pick the mapping $R_1$ from Definition~\ref{def:R1}. By Lemma~\ref{lem:R1isok}, there exists a continuous mapping $F:X_1\to \B$ such that $F(x)\in \isomtrclass{\C(K_x)}$ for every $x\in X_1$. Hence, for every compact space $K$ we have 
\[
F^{-1}(\isomtrclass{\C(K)}) = \{x\in X_1\setsep K_x\text{ is homeomorphic to }K\}
\]
and an application of Proposition~\ref{prop:R1step}\ref{it:K1sets+1} and \ref{it:K_1sets+D1} gives  that $F$ is a continuous reduction from a $\boldsymbol{\Pi}^0_3$-hard set to $\isomtrclass{\C(\omega)}$ and, for any $k\geq 2$, from a $D_2(\boldsymbol{\Pi}^0_3)$-hard set to $\isomtrclass{\C(\omega\cdot k)}$.
\end{proof}

\subsection{Successor step}\label{subsec:successor}

Now, given a mapping $R\colon X\times M\to\pot(N)$ we would like to have a way to produce another mapping $R^+$ in such a way that compact spaces $K_{R^+,x}$ are more complicated than compact spaces $K_{R,x}$. This is achieved using the following construction.

\begin{definition}\label{def:succ}Let $X$ be a Polish space and $M$, $N$ be infinite sets with $M$ countable. Let $R\colon X\times M\to\pot(N)$ be an arbitrary mapping. Then we define $M^+:=(M\times \Nat)\cup \{\infty\}$, $N^+:=(N\times \Nat)\cup \{\infty\}$ and $R^+:X^\Nat\times M^+\to \pot(N^+)$ for $x = (x_i)_{i\in\Nat}\in X^\Nat$ and $m\in M^+$ as
    \[
        R^+(x,m)=
        \begin{cases}
            R(x_i,m')\times \{i\},&\text{if }m=(m',i)\in M\times\Nat,\\
            N^+,&\text{if }m=\infty.
        \end{cases}
\]
We say that $R^+$ is the \emph{successor of the mapping $R$}.
\end{definition}

The idea behind this construction is the following.

\begin{lemma}\label{lem:R+IsOnePointCompactification}Let $X$ be a Polish space and $M$, $N$ be infinite sets with $M$ countable. Let $R\colon X\times M\to\pot(N)$ be a mapping that meets the conditions \ref{it:R1} and \ref{it:R2}. Then for $R^+:X^\Nat\times M^+\to \pot(N^+)$, the successor of $R$, and for any  $x\in X^\Nat$, $K_{R^+,x}$ is homeomorphic to the one-point compactification of the disjoint union of compact spaces $K_{R,x_i}$, $i\in \Nat$.
\end{lemma}
\begin{proof}First, we note that by definition $B_{R,x}$ always contains $\emptyset$ and $N$, so $K_{R,x}\neq\emptyset$ for every $x\in X$. Fix $x = (x_i)_{i\in\Nat}\in X^\Nat$. Then the Boolean algebra $B_{R^+,x}$ consists exactly of sets of the form $\bigcup_{i\in F}(B_i\times\{i\})$ where $B_i\in B_{R,x_i}$, $i\in F$, and $F\subseteq\Nat$ is a finite set, and of their complements (here we are using the assumption that $R$ meets \ref{it:R2} as otherwise it could happen that $N\in B_{R,x_i}$ by definition, but $ N \times \{ i \} \notin B_{R^+,x}$). Given $i\in I$ and $B_i\in B_{R,x_i}$ we consider $\varphi(B_i\times \{i\})\in \Pi_{i\in I}^w\; B_{R,x_i}$ defined as $\varphi(B_i\times \{i\})(i)=B_i$ and $\varphi(B_i\times \{i\})(j)=\emptyset$ if $j\neq i$. Now, it is easy to see that $\varphi$ naturally extends to an isomorphism between $B_{R^+,x}$ and $\Pi_{i\in I}^w\; B_{R,x_i}$, so application of Proposition~\ref{prop:onePointCompactification} finishes the proof.
\end{proof}

The following shows that the construction of a successor of the mapping $R$ is well-designed in the sense it preserves condition \ref{it:R4} and therefore also continuity of the induced map.

\begin{lemma}\label{lem:R+isok}Let $X$ be a Polish space and $M$, $N$ be infinite sets with $M$ countable. Let $R\colon X\times M\to\pot(N)$ be a mapping that meets conditions \ref{it:R1} and \ref{it:R4}. Then $R^+$ meets conditions \ref{it:R1}, \ref{it:R2} and \ref{it:R4} and therefore there exists a continuous mapping $F:X^\Nat\to \P$ such that $X_{F(x)}$ is isometric to $\C(K_x)$ for every $x\in X^\Nat$.

If moreover $R$ meets the condition \ref{it:R3}, then $R^+$ meets condition \ref{it:R3} as well and so the mapping $F$ has values in $\B\subseteq \P$.
\end{lemma}
\begin{proof}The proof that $R^+$ satisfies conditions \ref{it:R1} and \ref{it:R2} is straightforward and therefore omitted. In order to prove $R^+$ meets \ref{it:R4}, pick $x\in X^\Nat$ and finite $G\subseteq M^+$. We may without loss of generality enlarge the set $G$ and so we assume $G = \{\infty\}\cup (G'\times N')$ for some finite sets $G'\subseteq M$ and $N'\subseteq \Nat$. Now, using the assumption for each $i\in N'$ find an open neighborhood $U_i\subseteq X$ of the point $x_i$ such that for every $y_i\in U_i$ there exists a finite set $H_i(y_i)\supseteq G'$ satisfying $\{S_z^{H_i(y_i)}(y_i)\setsep z\in N\} = \{S_z^{H_i(y_i)}(x_i)\setsep z\in N\}$. Put $U:=\bigcap_{i\in N'} \pi_i^{-1}(U_i)$ and pick any $y\in U$. Consider now the set $H:=\{\infty\}\cup \bigcup_{i\in N'} (H_i(y_i)\times \{i\})\supseteq G$. Note that $S_\infty^H(y) = \{\infty\} = S_\infty^H(x)$ and for $z = (z_0,i_0)\in N\times \Nat$ we have
\[\begin{split}
S_z^H(y)  & = \{\infty\}\cup \{(m,i)\in H\setsep (z_0,i_0)\in R^+(y,(m,i))\}\\
&  = \{\infty\}\cup \{(m,i_0)\in H\setsep z_0\in R(y_{i_0},m)\},
\end{split}\]
so if $i_0\notin N'$ we obtain $S_z^H(y) = \{\infty\} = S_z^H(x)$ and if $i_0\in N'$ we obtain
\[\begin{split}
 S_z^H(y)\setminus\{\infty\} & = \{(m,i_0)\setsep m\in H_{i_0}(y_{i_0})\;\&\; z_0\in R(y_{i_0},m)\}\\
 & = \{(m,i_0)\setsep m\in S_{z_0}^{H_{i_0}(y_{i_0})}(y_{i_0})\} = S_{z_0}^{H_{i_0}(y_{i_0})}(y_{i_0})\times \{i_0\},
\end{split}\]
and therefore, since $\{S_{z_0}^{H_{i_0}(y_{i_0})}(y_{i_0})\setsep z_0\in N\} = \{S_{z_0}^{H_{i_0}(y_{i_0})}(x_{i_0})\setsep z_0\in N\}$ for every $i_0\in N'$, we obtain
\[
\{S_z^H(y)\setsep z\in N^+\} = \{S_z^H(x)\setsep z\in N^+\},
\]
which is what was to be proved in order to verify condition \ref{it:R4} for $R^+$.

The rest of the proof is immediate, using Lemma~\ref{lem:toB} and Lemma~\ref{lem:Fcont}.
\end{proof}

The following shows that complexity of sets corresponding to the successor of a suitable mapping $R$ is increasing.

\begin{prop}\label{prop:sucStep}Let $X$ be an uncountable Polish space and $M$, $N$ be infinite sets with $M$ countable and $R\colon X\times M\to\pot(N)$ be a mapping satisfying conditions \ref{it:R1}, \ref{it:R2} and $\alpha$,$\beta$ be countable nonzero ordinals. Assume the following holds.
\begin{enumerate}[label=(\roman*)]
    \item\label{it:KalphaMoznosti} For every $x\in X$ the compact $K_x$ is homeomorphic either to $[0,\omega^\alpha\cdot k]$ for some $k\in\Nat$ or to $[0,\omega^{\alpha + 1}]$.
    \item\label{it:hardAlphaSet} The set $\{x\in X\setsep K_x\text{ is homeomorphic to }[0,\omega^{\alpha +1}]\}$ is $\boldsymbol{\Pi}_\beta^0$-hard.
\end{enumerate}
Then for the mapping $R^+:X^\Nat\times M^+\to\pot(N^+)$ the following holds.
\begin{enumerate}[label=(\roman*')]
    \item\label{it:Kalpha+OnePointCompactification} For every $y\in X^\Nat$ the compact $K_y$ is homeomorphic either to $[0,\omega^{\alpha+1}\cdot k]$ for some $k\in\Nat$ or to $[0,\omega^{\alpha + 2}]$.
    \item\label{it:alpha+2} The set $\{y\in X^\Nat\setsep K_y\text{ is homeomorphic to }[0,\omega^{\alpha +2}]\}$ is $\boldsymbol{\Pi}_{\beta+2}^0$-hard.
    \item\label{it:alpha+1} The set $\{y\in X^\Nat\setsep K_y\text{ is homeomorphic to }[0,\omega^{\alpha +1}]\}$ is $\boldsymbol{\Pi}_{\beta+1}^0$-hard.
    \item\label{it:Dalpha+1} For every $k\geq 2$, the set $\{y\in X^\Nat\setsep K_y\text{ is homeomorphic to }[0,\omega^{\alpha +1}\cdot k]\}$ is $D_2(\boldsymbol{\Pi}_{\beta+1}^0)$-hard.
\end{enumerate}
\end{prop}
\begin{proof}Given $y\in X^\Nat$, by Lemma~\ref{lem:R+IsOnePointCompactification} the compact space $K_y$ is homeomorphic to the one-point compactification of the compact spaces $K_{y_i}$, $i\in\Nat$. By \ref{it:hardAlphaSet}, the set $Z:=\{x\in X\setsep K_x\text{ is homeomorphic to }[0,\omega^{\alpha + 1}]\}$ is $\boldsymbol{\Pi}_{\beta}^0$-hard and we have
\begin{itemize}
    \item $K_y$ is homeomorphic to $[0,\omega^{\alpha + 1}]$ if and only if $y_i\notin Z$ for every $i\in \Nat$,
    \item more generally, for every $k\in\Nat$, compact space $K_y$ is homeomorphic to $[0,\omega^{\alpha + 1}\cdot k]$ if and only if $|\{i\in\Nat\setsep y_i\in Z\}| = k-1$,
    \item $K_y$ is homeomorphic to $[0,\omega^{\alpha + 2}]$ if and only if $\{i\in\Nat\setsep y_i\in Z\}$ is an infinite set.
\end{itemize}
This proves that \ref{it:Kalpha+OnePointCompactification} holds. Since $Z$ is $\boldsymbol{\Pi}_{\beta}^0$-hard, the set
\[
\{y\in X^\Nat\setsep K_y\sim[0,\omega^{\alpha +1}]\} = \{y\in X^\Nat\setsep \forall i\in \Nat\;y_i\notin Z\}
\]
is $\boldsymbol{\Pi}_{\beta+1}^0$-hard by Proposition~\ref{prop:difHard}, which proves \ref{it:alpha+1} and for every $k\geq 2$ the set
\[
\{y\in X^\Nat\setsep K_y\sim [0,\omega^{\alpha +1}\cdot k]\} = \{y\in X^\Nat\setsep |\{i\in\Nat\setsep y_i\in Z\}| = k-1\}
\]
is $D_2(\boldsymbol{\Pi}_{\beta+1}^0)$-hard by Proposition~\ref{prop:difHard}, which proves \ref{it:Dalpha+1}. Finally, the set
\[
\{y\in X^\Nat\setsep K_y\sim [0,\omega^{\alpha +2}]\} = \{y\in X^\Nat\setsep y_i\in Z \text{ for infinitely many $i$'s}\}
\]
is $\boldsymbol{\Pi}_{\beta+2}^0$-hard by Lemma~\ref{lem:completeSetInductionStep}, which proves \ref{it:alpha+2}.
\end{proof}

The main outcome of the section may be summarized by the following theorem from which the lower bound in Theorem~\ref{thm:main1Part2}~\ref{it:kIsOneComplete} and \ref{it:kIsNotOneComplete} for the case $\N=\B$ and compacta of finite height easily follow, we refer the reader to Section~\ref{sec:appl}, where the details are provided.

\begin{thm}\label{thm:finiteOrdinalB} For every $n\in\Nat$ there exists a Polish space $X_n$ and a continuous mapping $F_n:X_n\to \B$ with the following properties.
\begin{itemize}
    \item $F_n(x)\in \bigcup_{k\in\Nat}\isomtrclass{\C(\omega^n\cdot k)}\cup \isomtrclass{\C(\omega^{n+1})}$ for every $x\in X_n$;
    \item $F_n$ is a continuous reduction of some $\boldsymbol{\Pi}^0_{2n+1}$-hard set to $\isomtrclass{\C(\omega^n)}$;
    \item for every $k\geq 2$, $F_n$ is a continuous reduction of some $D_2(\boldsymbol{\Pi}^0_{2n+1})$-hard set to $\isomtrclass{\C(\omega^n\cdot k)}$;
    \item $F_n$ is a continuous reduction of some $\boldsymbol{\Pi}^0_{2n+2}$-hard set to $\isomtrclass{\C(\omega^{n+1})}$.
\end{itemize}     
\end{thm}
\begin{proof}
We define sequence of mappings $R_n:X_n\times M_n\to\pot(M_n)$ satisfying conditions \ref{it:R1}, \ref{it:R2}, \ref{it:R3} and \ref{it:R4} as follows. For $n=1$ we use the mapping defined in Definition~\ref{def:R1} and if $R_n$ was already defined, we let $X_{n+1}=(X_n)^\Nat$ and $R_{n+1}$ be the successor of $R_n$, that is, $R_{n+1}:=(R_n)^+$. Note that then each $R_n$ satisfies conditions \ref{it:R1}, \ref{it:R2}, \ref{it:R3} and \ref{it:R4} by Lemma~\ref{lem:R1isok} and Lemma~\ref{lem:R+isok}. Moreover, using Proposition~\ref{prop:R1step} and Proposition~\ref{prop:sucStep} we inductively obtain that for every $n\in \Nat$ for the mapping $R_n$ the following holds.
\begin{itemize}
    \item For every $y\in X_n$ the compact $K_y$ is homeomorphic either to $[0,\omega^{n}\cdot k]$ for some $k\in\Nat$ or to $[0,\omega^{n+1}]$.
    \item The set $\{y\in X_n\setsep K_y\text{ is homeomorphic to }[0,\omega^{n+1}]\}$ is $\boldsymbol{\Pi}_{2n+2}^0$-hard.
    \item The set $\{y\in X_n\setsep K_y\text{ is homeomorphic to }[0,\omega^{n}]\}$ is $\boldsymbol{\Pi}_{2n+1}^0$-hard.
    \item For every $k\geq 2$, the set $\{y\in X_n\setsep K_y\text{ is homeomorphic to }[0,\omega^{n}\cdot k]\}$ is $D_2(\boldsymbol{\Pi}_{2n+1}^0)$-hard.
\end{itemize}
By Lemma~\ref{lem:toB}, for every $n\in\Nat$ there exists a continuous mapping $F_n:X_n\to \B$ such that $F_n(x)\in \isomtrclass{\C(K_x)}$ for every $x\in X_n$. Hence, for every compact space $K$ we have 
\[
F_n^{-1}(\isomtrclass{\C(K)}) = \{x\in X_n\setsep K_x\text{ is homeomorphic to }K\},
\]
which implies that $F_n$, $n\in\Nat$, are as required.
\end{proof}

\section{Proof of Theorem~\ref{thm:main1Part2}~\ref{it:kIsOneComplete}, \ref{it:kIsNotOneComplete} and further applications}\label{sec:appl}

We start by showing how the proof of Theorem~\ref{thm:main1Part2}~\ref{it:kIsOneComplete} and \ref{it:kIsNotOneComplete} may be deduced from results presented in previous sections. We start with the following (probably well-known) general lemma.

\begin{lemma}\label{lem:PSameAsB}
    Let $X$ and $Y$ be Polish spaces, $Z\subseteq Y$ and $f:X\to Y$ be a $\boldsymbol{\Sigma}^0_2$-measurable mapping. Let $\boldsymbol{\Gamma}$ be one of the classes $\boldsymbol{\Sigma}^0_\alpha, \boldsymbol{\Pi}^0_\alpha, D_2(\boldsymbol{\Pi}^0_\alpha)$ with $\omega\le\alpha < \omega_1$.
    Then whenever $f^{-1}(Z)$ is $\boldsymbol{\Gamma}$-hard, $Z$ is $\boldsymbol{\Gamma}$-hard as well.
\end{lemma}
\begin{proof}
Suppose $f^{-1}(Z)$ is $\boldsymbol{\Gamma}$-hard.
Fix some $\boldsymbol{\Gamma}$-complete set $A\subseteq 2^\Nat$; then there exists a continuous mapping $g:2^\Nat\to X$ with $g^{-1}(f^{-1}(Z)) = A$.
Let us denote $h:=f\circ g$, this is $\boldsymbol{\Sigma}^0_2$-measurable mapping.
Pick a countable base $\B_Y$ of the topology on $Y$ and denote the topology on $2^\Nat$ by $\tau$.
Then $h^{-1}(B)$ is an $F_\sigma$-set for every $B\in \B_Y$ and so there are closed sets $F_k^B$, $k\in\Nat$, such that $h^{-1}(B)$ is a countable union of sets from $\{F_k^B\setsep k\in\Nat\}$ for every $B\in \B_Y$.
By \cite[Lemma 13.2 and Lemma 13.3]{Kechrisbook}, the topology $\tau'$ on $2^\Nat$ generated by $\tau\cup \{F_k^B\setsep B\in\B_Y,\;k\in\Nat\}$ is a Polish topology. Every $\tau'$-open subset of $2^\Nat$ can be written as the countable union of sets of the form $U\cap \bigcap_{(B,k)\in H} F_k^B$ where $U\in\tau$ and $H\subseteq\B_Y\times \Nat$ is a finite set.
In particular, every $\tau'$-open set is a $\boldsymbol{\Sigma}^0_3$-set with respect to the topology $\tau$.
Moreover, the topology $\tau'$ is zero-dimensional because sets $F_k^B$ are $\tau'$-clopen, and the mapping $h:(2^\Nat,\tau')\to Y$ is continuous because all the sets $h^{-1}(B)$, $B\in\B_Y$, are $\tau'$-open.

So we have $\tau\subseteq \tau'\subseteq\boldsymbol{\Sigma}^0_3(2^\Nat,\tau)$.
It follows by an easy induction that $\boldsymbol{\Pi}^0_{k}(2^\Nat,\tau)\subseteq\boldsymbol{\Pi}^0_{k}(2^\Nat,\tau')\subseteq \boldsymbol{\Pi}^0_{k+2}(2^\Nat,\tau)$ and $\boldsymbol{\Sigma}^0_{k}(2^\Nat,\tau)\subseteq\boldsymbol{\Sigma}^0_{k}(2^\Nat,\tau')\subseteq \boldsymbol{\Sigma}^0_{k+2}(2^\Nat,\tau)$ for every $k\in\Nat$.
Therefore $\boldsymbol{\Pi}^0_{\alpha}(2^\Nat,\tau') = \boldsymbol{\Pi}^0_{\alpha}(2^\Nat,\tau)$ and $\boldsymbol{\Sigma}^0_{\alpha}(2^\Nat,\tau') = \boldsymbol{\Sigma}^0_{\alpha}(2^\Nat,\tau)$ for every $\alpha\geq \omega$, so we have $\boldsymbol{\Gamma}(2^\Nat,\tau) = \boldsymbol{\Gamma}(2^\Nat,\tau')$.

Recall that, in any zero-dimensional Polish space, a set is $\boldsymbol{\Gamma}$-complete if and only if it is in $\boldsymbol{\Gamma}$ and not in $\check{\boldsymbol{\Gamma}}$ (in the case $\Gamma=\boldsymbol{\Sigma}^0_\alpha$ or $\Gamma=\boldsymbol{\Pi}^0_\alpha$, see e.g.~\cite[Theorem~22.10]{Kechrisbook}; the case $\Gamma=D_2(\boldsymbol{\Pi}^0_\alpha)$ is analogous).
Since by the above these two notions coincide for topologies $\tau$ and $\tau'$, we obtain that the set $h^{-1}(Z) = A$ is $\boldsymbol{\Gamma}$-complete with respect to the topology $\tau'$.
Thus, $h:(2^\Nat,\tau')\to Y$ is a continuous reduction from a $\boldsymbol{\Gamma}$-complete set $A$ to the set $Z$ and therefore $Z$ is $\boldsymbol{\Gamma}$-hard.
\end{proof}

\begin{proof}[Proof of Theorem~\ref{thm:main1Part2}~\ref{it:kIsOneComplete} and \ref{it:kIsNotOneComplete}] Let $\beta$ and $n$ be as in the assumptions. Pick $k\in\Nat$, put $K:=[0,\omega^{\beta+n}\cdot k]$ and denote $\boldsymbol{\Gamma}:=\boldsymbol{\Pi}^0_{\beta+2n+1}$ if $k=1$ and $\boldsymbol{\Gamma}:=D_2(\boldsymbol{\Pi}^0_{\beta+2n+1})$ if $k\geq 2$. By Theorem~\ref{thm:upperBoundFinal}, the isometry class of $\C(K)$ in $\N$ belongs to the class $\boldsymbol{\Gamma}$, no matter whether $\N=\PP_\infty$ or $\N=\B$. In order to see it is also $\boldsymbol{\Gamma}$-hard, we shall consider several cases.

First, we observe that $\isomtrclass[\PP_\infty]{\C(K)}$ is $\boldsymbol{\Gamma}$-hard by Theorem~\ref{thm:main2Part2} and Corollary~\ref{cor:Ck}. Hence, it remains to deal with the case $\N=\B$.

Assume $\beta$ is a limit ordinal. Then $\boldsymbol{\Gamma}$ is not a finite Borel class. By~\cite[Proposition~3.6]{CDDK1}, there is an $F_\sigma$-measurable map $\Phi\colon\P_\infty\to\B$ such that $\Phi^{-1}\big(\isomtrclass[\mathcal{\B}]{\C(K)}\big)=\isomtrclass[\mathcal{\P_\infty}]{\C(K)}$. Hence, using Lemma~\ref{lem:PSameAsB} and the fact that $\isomtrclass[\mathcal{\P_\infty}]{\C(K)}$ is $\boldsymbol{\Gamma}$-hard by the already proven part, we obtain that the isometry class of $\C(K)$ in $\B$ is also $\boldsymbol{\Gamma}$-hard.

Finally, assume $\beta = 0$. Then by Theorem~\ref{thm:finiteOrdinalB} there exists a continuous reduction from a $\boldsymbol{\Gamma}$-hard set to $\isomtrclass[\B]{\C(K)}$ and therefore $\isomtrclass[\B]{\C(K)}$ is $\boldsymbol{\Gamma}$-hard as well.
\end{proof}

\begin{remark}
Although we do not formulate the precise statement, a result analogous to Theorem~\ref{thm:main1Part2}~\ref{it:kIsOneComplete} and \ref{it:kIsNotOneComplete} holds for $\beta\neq 0$ even for the coding $(SB(X),\tau)$ (instead of $\B$), where $\tau$ is an arbitrary admissible topology (as defined in~\cite{GodSR}).
The proof is still the same, we just apply~\cite[Theorem~3.11]{CDDK1} instead of~\cite[Proposition~3.6]{CDDK1} and \cite[Theorem~3.3]{CDDK1} instead of the fact that $\B\subseteq \PP_\infty$.
\end{remark}

Let us notice at this point that our method of proof implies in particular that Lemma~\ref{lem:szlenkComplexity} cannot be much improved. The next result may be viewed as a refinement of \cite[Corollary~4.14]{CDDK2} in the particular case $\varepsilon=2$.

\begin{cor}\label{cor:szlenkComplexity}
Let $\beta$ be either $0$ or a countable limit ordinal and let $n\in\Nat\cup\{0\}$ be such that $\beta+n>0$. Then the mapping
\[
\K(\ell_\infty)\ni F\mapsto s_{2}^{\beta+n}(F)\in \K(\ell_\infty)
\]
is not $\boldsymbol{\Pi}_{\beta+2n+1}^0$-measurable.
\end{cor}
\begin{proof}Suppose first that $\beta+n\ge 2$. 
Then, if $s_{2}^{\beta+n}$ was $\boldsymbol{\Pi}_{\beta+2n+1}^0$-measurable,
we would obtain in the proof of Theorem~\ref{thm:upperBoundFinal} that $\isomtrclass{\C(\omega^{\beta+n})}$
is a $\boldsymbol{\Sigma}^0_{\beta+2n+1}$-set, which would contradict Theorem~\ref{thm:main1Part2}.

Now suppose that $\beta=0$ and $n=1$.
If $s_{2}^{1}$ was $\boldsymbol{\Pi}_{3}^0$-measurable, then $s_{2}^{2}=s_{2}^{1}\circ s_{2}^{1}$ would be $\boldsymbol{\Pi}_{4}^0$-measurable, a contradiction with what we already proved.
\end{proof}

The continuous reductions constructed in the previous sections have a well-controlled behavior. In particular, the following result is an immediate corollary of Theorem~\ref{thm:countableOrdinal} and Lemma~\ref{lem:Ck}.

\begin{cor}\label{cor:limitOrdinal}
    For any countable limit ordinal $\lambda$ and any $\boldsymbol{\Pi}^0_\lambda$-hard set $A\subseteq \Nat^\Nat$ there exists a continuous mapping $H:\Nat^\Nat\to \P_\infty$ such that 
    \begin{itemize}
        \item $H$ is a continuous reduction of $A$ to $\isomtrclass{\C(\omega^\lambda)}$,
        \item $X_{H(x)}$ is linearly isometric to $\C(K)$ with $K$ being a compact subset of $[0,\omega^\lambda]$ for any $x\in \Nat^\Nat$.
    \end{itemize}
\end{cor}

In what follows we shall be using the fact that whenever $(*)$ is a property of Banach spaces preserved by linear isometries and whenever $\boldsymbol{\Gamma}$ is one of the classes $\boldsymbol{\Sigma}^0_\alpha$, $\boldsymbol{\Pi}^0_\alpha$, $D_2(\boldsymbol{\Pi}^0_\alpha)$ for some $\alpha < \omega_1$ with $\alpha\geq \omega$, then $\{\mu\in \P_\infty\setsep X_\mu\text{ has }(*)\}$ is $\boldsymbol{\Gamma}$-hard iff $\{\mu\in \B\setsep X_\mu\text{ has }(*)\}$ is $\boldsymbol{\Gamma}$-hard. This follows from Lemma~\ref{lem:PSameAsB} together with the fact that by~\cite[Proposition~3.6]{CDDK1}, there is an $F_\sigma$-measurable map $\Phi\colon\P_\infty\to\B$ such that $X_{\Phi(\mu)}\equiv X_\mu$ for every $\mu\in\PP_\infty$.  As an immediate corollary to Corollary~\ref{cor:limitOrdinal} we therefore obtain the following.

\begin{cor}\label{cor:limitOrdinalIsomorphProperty}If $\lambda$ is a countable limit ordinal and $(P)$ is a property of Banach spaces preserved by linear isometries, which holds for the Banach space $ \C(\omega^{\lambda}) $, but not for $\C(\alpha)$ with $\alpha< \omega^{\lambda} $, then the set 
    \[
    \{\mu\in \N\setsep X_\mu\text{ has }(P)\}
    \]
    is $\boldsymbol{\Pi}^0_\lambda$-hard for both $\N=\PP_\infty$ and $\N=\B$.
\end{cor}

The following consequence in particular answers in negative \cite[Question 6]{CDDK2}. It shows that the upper estimates on classes of Banach spaces involving the Szlenk index obtained in \cite{CDDK2} are actually quite close to the lower estimates we are able to obtain using the methods developed above.

\begin{cor}\label{cor:szlenk}Let $\N\in \{\B,\P_\infty\}$.
    \begin{itemize}
        \item The set $\{\mu\in \N\setsep X_\mu \text{ has summable Szlenk index}\}$ is $\boldsymbol{\Sigma}^0_\omega$-hard.\\
        (And it is a $\boldsymbol{\Sigma}^0_{\omega+2}$-set by \cite[Proposition 5.8]{CDDK2}.)
        \item For any $\alpha\in [1,\omega_1)$, the set $\{\mu\in \N\setsep Sz(X_\mu)\leq \omega^\alpha\}$ is $\boldsymbol{\Sigma}^0_{\omega^\alpha}$-hard.\\
        (And it is a $\boldsymbol{\Pi}^0_{\omega^\alpha+1}$-set by \cite[Theorem 5.7]{CDDK2}.)
\end{itemize}
\end{cor}
\begin{proof}We apply Corollary~\ref{cor:limitOrdinalIsomorphProperty} together with the following.
\begin{itemize}
    \item For any compact $K\subseteq [0,\omega^\omega]$, it is well-known that $\C(K)$ has a summable Szlenk index if and only if $K\not\sim [0,\omega^\omega]$.\\[3pt]
    (If $K\subseteq [0,\omega^\omega]$ and $K\not\sim [0,\omega^\omega]$, then by the Bessaga-Pe\l czy\'nski classification $\C(K)$ is isomorphic to $c_0$ and therefore it has summable Szlenk index, see e.g. \cite[Theorem 5.6]{KGL01}; on the other hand if $\C(K)$ has a summable Szlenk index then neccesarily $Sz(\C(K))\leq \omega$ but $Sz(\C(\omega^\omega)) = \omega^2$, see e.g. \cite[Theorem 2.59]{BiortogBook}.)
    \item For any compact $K\subseteq [0,\omega^{\omega^\alpha}]$, it is well-known that $Sz(\C(K))\leq \omega^\alpha$ if and only if $K\not\sim [0,\omega^{\omega^\alpha}]$.\\[3pt]
    (Indeed, the values of Szlenk indices for separable $\C(K)$ spaces are well-known, see e.g. \cite[Section 2.6]{BiortogBook}.)
\end{itemize}
\end{proof}

We note that in a similar way one could obtain consequences concerning the dentability index, see \cite{HLP09} where this index is computed for $\C(\alpha)$ spaces with $\alpha$ countable.

A variant of Corollary~\ref{cor:limitOrdinalIsomorphProperty} for successor ordinals is the following.

\begin{cor}\label{cor:succConsequences}
Let $\beta$ be either $0$
or a countable limit ordinal and let $n\in\Nat\cup\{0\}$ be such that $ \beta + n > 0 $. Assume $(P)$ is a property of Banach spaces preserved by linear isometries, which holds for the Banach space $\C(\omega^{\beta + n + 1})$, but not for $\C(\omega^{\beta+n}\cdot k)$ with $k\in\Nat$, then the set
\[
    \{\mu\in \N\setsep X_\mu\text{ has }(P)\}
\]
is $\boldsymbol{\Pi}^0_{\beta+2n+2}$-hard for both $\N=\PP_\infty$ and $\N=\B$.
\end{cor}
\begin{proof}If $\beta$ is a limit ordinal, for $\N=\PP_\infty$ this follows from Corollary~\ref{cor:contReduction}
and Lemma~\ref{lem:Ck}, and for $\N=\B$ we then use Lemma~\ref{lem:PSameAsB} as mentioned above.

If $\beta=0$, for $\N=\B$ this follows from Theorem~\ref{thm:finiteOrdinalB} and for $\N=\PP_\infty$ we obtain it directly from the fact that $\B\subseteq \PP_\infty$.
\end{proof}

An example of a property $(P)$ to which the above result applies is provided by a recent result concerning measures of weak non-compactness. Namely, it follows from the recent thesis \cite{Smerciak} that the property ``for every bounded set $A\subseteq X$ we have $\omega(A) \leq 2n\cdot \mathrm{wk}(A)$'' (where $\omega(A)$ and $\mathrm{wk}(A)$ are measures of weak non-compactness, we refer to \cite{Smerciak} for more details) holds in $\C(\omega^n\cdot k)$ with $k\in\Nat$ but not in $\C(\omega^{n+1})$.

In relation to our Theorem~\ref{thm:main3Part2} let  us note that the following is not known to us.

\begin{ques}\label{que:CKBorel}Is it true that the set 
\[
\{\mu\in\B\setsep X_\mu\text{ is isometric to a }\C(K)\text{ space for some metric compact $K$}\}
\]
is Borel?
\end{ques}
We remark that several characterizations of Banach spaces that are isometric to a $\mathcal{C}(K)$ space are known (see, for instance, \cite{LaceyBook}). However, many of these rely on additional structure -- such as that of a lattice or a Banach algebra -- while those that avoid such extra assumptions do not appear to yield a Borel characterization.

{
\makeatletter
\let\addcontentsline\@gobblethree
\bibliography{refDst}\bibliographystyle{acm}
} 
\end{document}